\documentclass{amsart}

\makeatletter
\@namedef{subjclassname@2020}{%
  \textup{2020} Mathematics Subject Classification}
\makeatother

\title{Poles of finite-dimensional representations of Yangians}
\author[S. Gautam]{Sachin Gautam}
\address{Department of Mathematics, The Ohio State University, Columbus, OH 43210, USA}
\email{gautam.42@osu.edu}

\author[C. Wendlandt]{Curtis Wendlandt}
\address{Department of Mathematics and Statistics, University of Saskatchewan, Saskatoon, SK S7N 5E6, Canada}
\email{wendlandt@math.usask.ca}

\subjclass[2020]{Primary 17B37; Secondary 81R10}

\usepackage{amsmath,amssymb,mathtools}
\usepackage{mathrsfs}
\usepackage[usenames,dvipsnames]{color}
\usepackage{enumitem}
\usepackage{upgreek}
\usepackage{comment}
\usepackage{hyperref}
\hypersetup{
    colorlinks=true,                            
    linkcolor=blue, 
    citecolor=blue,                             
    filecolor=blue,                                     
    urlcolor=blue                               
}

\usepackage{float}
\usepackage[all]{xy}
\usepackage{tikz-cd} 
\xyoption{arc}
\newtheorem*{thm}{Theorem}
\newtheorem*{prop}{Proposition}
\newtheorem*{lem}{Lemma}
\newtheorem*{cor}{Corollary}
\newtheorem*{conj}{Conjecture}
\newenvironment{pf}{\paragraph{{\sc Proof}}}{\qed\par\medskip}
\theoremstyle{definition}
\newtheorem*{defn}{Definition}
\newtheorem*{rem}{Remark}

\newtheorem*{ex}{Example}

\numberwithin{equation}{section}

\newcommand{\Pseries}[2]{#1[\negthinspace[#2]\negthinspace]}

\newcommand{\spec}{\sigma}
\newcommand{\order}{\operatorname{ord}}
\newcommand{\zeroes}{\mathsf{Z}}

\newcommand{\id}{\mathbf{1}}

\newcommand{\lp}{\left(}
\newcommand{\rp}{\right)}

\newcommand{\lb}{\left[}
\newcommand{\rb}{\right]}

\newcommand{\g}{\mathfrak{g}}
\newcommand{\h}{\mathfrak{h}}

\newcommand{\Sym}{\mathfrak{S}}

\newcommand{\vup}[1]{v_+}
\newcommand{\vdown}[1]{v_-}

\newcommand{\bra}[1]{\left|#1\right\rangle}
\newcommand{\vac}{\Omega}

\newcommand{\bfA}{\mathbf{A}}

\newcommand{\V}{\mathcal{V}}

\newcommand{\X}{\mathcal X}
\newcommand{\Y}{\mathcal Y}

\newcommand{\C}{\mathbb{C}}
\newcommand{\nC}{\mathbb{C}^{\times}}
\newcommand{\N}{\mathbb{Z}_{\geq 0}}

\newcommand{\Z}{\mathbb{Z}}

\renewcommand{\Re}{\operatorname{Re}}

\newcommand{\fp}{{\mathbf p}}

\newcommand {\wh}[1]{\widehat{#1}}
\newcommand {\ol}[1]{\overline{#1}}
\newcommand {\ul}[1]{\underline{#1}}

\newcommand{\End}{\operatorname{End}}
\newcommand{\Hom}{\operatorname{Hom}}
\newcommand{\ad}{\operatorname{ad}}
\newcommand{\Ad}{\operatorname{Ad}}

\newcommand {\aand}{\qquad\text{and}\qquad}
\renewcommand{\dim}{\operatorname{dim}}

\newcommand {\Rep}{\operatorname{Rep}}

\newcommand {\fd}{finite-dimensional }
\newcommand {\lhs}{left-hand side }

\newcommand {\irr}{irreducible }

\newcommand{\Ryang}{\Rep_{fd}(\Yhg)}

\newcommand{\ds}{\displaystyle}

\newcommand {\Comment}[1]{}
\newcommand {\Omit}[1]{}

\newcommand {\Yhg}{Y_\hbar(\g)}

\newcommand{\Yhsl}[1]{Y_{\hbar}(\sl_{#1})}

\newcommand {\isom}{\xrightarrow{\,\smash{\raisebox{-0.5ex}{\ensuremath{\scriptstyle\sim}}}\,}}
\newcommand {\longisom}{\stackrel{\sim}{\longrightarrow}}

\newcommand{\cbin}[2]{\left(\begin{array}{c} #1\\ #2\end{array}\right)}

\renewcommand {\sl}{\mathfrak{sl}}

\newcommand {\eg}{{\it e.g., }}

\newcommand {\Id}{\operatorname{Id}}
\newcommand {\ie}{{\it i.e.}, }
\newcommand{\bfI}{{\mathbf I}}
\newcommand{\bfJ}{{\mathbf J}}

\newcommand {\curtiscomment}[1]{\footnote{\textcolor{magenta}{C:\,#1}}}

\allowdisplaybreaks

\newcommand{\Tr}{\operatorname{Tr}}

\newcommand{\FH}{\scriptstyle{\text{FH}}}

\newcommand{\Tnorm}{\ol{T}}

\begin{document}

\begin{abstract}
Let $\g$ be a finite-dimensional simple Lie algebra over $\C$, and let
$\Yhg$ be the Yangian of $\g$. In this paper, we study
the sets of poles of the rational currents defining the action of $\Yhg$
on an arbitrary finite-dimensional vector space $V$. 
Using a weak, rational version of Frenkel and Hernandez' Baxter polynomiality, we obtain a  uniform description of these sets in terms of the Drinfeld polynomials encoding the composition factors of $V$ and the inverse of the $q$-Cartan matrix of
$\g$. We then apply this description to obtain a concrete set of sufficient conditions
for the cyclicity and simplicity of the tensor product of any two irreducible representations, and to classify the finite-dimensional irreducible representations of the Yangian double.
\end{abstract}

\maketitle

{\setlength{\parskip}{0pt}
\setcounter{tocdepth}{1}
\tableofcontents
}

\section{Introduction}

\subsection{}\label{ssec:overview}

Let $\g$ be a simple Lie algebra over $\C$ with Cartan matrix $(a_{ij})_{i,j\in \bfI}$, and let $\Yhg$ be the Yangian associated to $\g$, where $\hbar\in\C^\times$ is fixed. It is well-known that the generating series $\{\xi_i(u),x_i^{\pm}(u)\}_{i\in\bfI}\subset \Yhg[\![u^{-1}]\!]$ arising in Drinfeld's second realization of the Yangian \cite{drinfeld-yangian-qaffine} operate on each finite-dimensional representation $V$ of $\Yhg$ as the Taylor expansions at infinity of $\End(V)$-valued rational functions of $u$ \cite{beck-kac,  sachin-valerio-2, hernandez-drinfeld}. Consequently, to each index $i\in \bfI$ we may associate a finite subset $\sigma_i(V)\subset \C$, called the \textit{$i$-th set of poles} of $V$, by setting
\begin{equation*}
\sigma_i(V):=\{\text{Poles of }\xi_i(u),x_i^{\pm}(u)\in\End(V)(u)\}\subset \C.
\end{equation*}

In this paper, we compute these sets explicitly in terms of the Drinfeld polynomials \cite{drinfeld-yangian-qaffine}  parametrizing the composition factors of $V$ and the entries of the inverse $q$-Cartan matrix of $\g$, which already appears in a number of important constructions for quantum groups of affine type \cite{frenkel-mukhin, frenkel-reshetikhin-Walg, frenkel-reshetikhin-qchar, fujita2020, fujita-oh, sachin-valerio-III, hernandez-qt, hernandez-leclerc2, khoroshkin-tolstoy}. We then highlight the role played by the sets $\sigma_i(V)$ in the representation theory of the Yangian by narrowing our focus to two distinct problems where they naturally feature; see Sections \ref{ssec:results-III} and \ref{ssec:results-IV}. 
As we shall explain in more detail below, our study makes essential use of a weak, rational version of \textit{Baxter polynomiality} -- a polynomiality result for certain transfer operators which was obtained in generality for the quantum loop algebra $U_q(L\g)$ in the work \cite{frenkel-hernandez} of Frenkel and Hernandez.

\subsection{A uniform formula for $\sigma_i(V)$}\label{ssec:results}
Let us now give a more detailed outline of our results. The first half of our paper culminates with a uniform formula for the $i$-th set of poles $\sigma_i(V)$ of an arbitrary finite-dimensional irreducible representation $V$ of $\Yhg$. 
To state it precisely, we first recall from \cite{drinfeld-yangian-qaffine} that such representations of $\Yhg$ are classified by $\bfI$-tuples of monic polynomials $(P_i(u))_{i\in \bfI}$ -- these are the Drinfeld polynomials alluded to above (see Section \ref{ssec:dp}). Under this parametrization, the irreducible representation associated to the tuple $(u^{\delta_{ji}})_{i\in \bfI}$ is called the $j$-th fundamental representation, and denoted $L_{\varpi_j}$. We then have the following result, which is the content of Theorem \ref{thm:Comb}.
\begin{thm}\label{thm:intro-1}
Let  $V$ be the unique, up to isomorphism, finite-dimensional irreducible representation of $\Yhg$ with tuple of Drinfeld polynomials
$(P_j(u))_{j\in \bfI}$. Then the $i$-th set of poles $\sigma_i(V)$ is given explicitly by
\begin{gather*}
\sigma_i(V)
=\bigcup_{j\in \bfI} \left(\zeroes(P_j(u))+\sigma_i(L_{\varpi_j})\right), \\
\sigma_i(L_{\varpi_j})=\left\{\frac{\hbar}{2}(s-d_j): d_i\leq s\leq 2\kappa - d_i \; \text{ and }\; v_{ij}^{(s)}>0\right\},
\end{gather*}
where $\zeroes(P_j(u))$ is the set of roots of $P_j(u)$ and $\sum_r v_{ij}^{(r)}q^r\in \Z[\![q]\!]$ is the $(i,j)$-th entry of the inverse $q$-Cartan matrix $([a_{ij}]_{q^{d_i}})_{i,j\in \bfI}^{-1}$.
%
\end{thm}
Here  $d_j$ is the $j$-th symmetrizing integer of $\g$ and $\kappa$ is defined to be one fourth the eigenvalue of the Casimir element $C\in U(\g)$ in the adjoint representation; see Sections \ref{ssec:g} and \ref{ssec:CPthm}, respectively.
More generally, the $i$-th set of poles $\sigma_i(V)$ of an arbitrary finite-dimensional $\Yhg$-module $V$ can be computed from the above formulas upon identifying a composition series for $V$. Indeed, this follows from Theorem \ref{thm:diagonal-poles}, which shows that $\sigma_i(V)$ is equal to the set of poles of those rational functions $\mu_i(u)\in \C(u)$ belonging to the spectrum of the single operator $\xi_i(u)$. This theorem itself has several consequences which go a long way towards describing the behaviour of the sets $\sigma_i(V)$ and do not trivially follow from their original definition; see Corollaries \ref{cor:sigma} and \ref{cor:tensor} in addition to Propositions \ref{P:q-Poles} and \ref{prop:STP}.
\subsection{}\label{ssec:results-II}

The proof of Theorem \ref{thm:intro-1} (\textit{i.e.,} Theorem \ref{thm:Comb})
 is directly inspired by the work of Frenkel and Hernandez, in that it makes crucial use of the Yangian version of a polynomiality result  from \cite[\S5]{frenkel-hernandez}, established in Theorem \ref{thm:Baxter}.
  To elaborate, in Section \ref{sec:baxter} we define an $\End(V)$-valued meromorphic function $T_i(u)$ as the solution to an abelian difference equation encoded by diagonal part of the universal $R$-matrix of the Yangian. In Theorem \ref{thm:Baxter} it is established that, up to a normalization factor depending only on $V$, $T_i(u)$ is an operator valued polynomial with monic eigenvalues -- these are the so-called \textit{specialized Baxter polynomials} associated to $V$. In addition, Theorem \ref{thm:Baxter} asserts that the joint set of roots of these eigenvalues is exactly the set of poles $\sigma_i(V)$ and that, crucially, one has $\sigma_i(V)=\zeroes(\mathcal{Q}_{i,V}^\g(u))$, where $\mathcal{Q}_{i,V}^\g(u)$ is the specialized Baxter polynomial associated to the lowest weight space of $V$. In Section \ref{sec:Combinatorial}, $\mathcal{Q}_{i,V}^\g(u)$ is computed explicitly as
\begin{equation}\label{intro:QiV}
\mathcal{Q}_{i,V}^\g(u)
= \prod_{j\in \bfI} \prod_{s=d_i}^{2\kappa-d_i}P_j\!\left(u-(s-d_j)\frac{\hbar}{2}\right)^{\!v_{ij}^{(s)}},
\end{equation}
which yields the uniform formula for $\sigma_i(V)$ stated in Theorem \ref{thm:intro-1}. 

As will be reviewed in Section \ref{ssec:FH-summary},  the polynomiality of the $U_q(L\g)$-analogue $T^{\FH}_i(z)$ of $T_i(u)$ was established in Theorem 5.17 of \cite{frenkel-hernandez} as a corollary to a  stronger polynomiality result for an operator $T^{\FH}_i(z,\ul{t})$ which specializes to $T_i^{\FH}(z)$. This stronger result -- Theorem 5.9 of \cite{frenkel-hernandez} -- was applied in \textit{loc.~cit.} to prove a generalization of a conjecture from \cite{frenkel-reshetikhin-qchar} which describes the spectra of the (twisted) transfer matrices associated to finite-dimensional representations of $U_q(L\g)$ in terms of the eigenvalues of $T^{\FH}_i(z,\ul{t})$, which are known as \textit{Baxter polynomials}. 
An important ingredient in this sequence of results is that $T^{\FH}_i(z,\ul{t})$ is itself defined from a twisted variant of the usual transfer matrix formalism
 using the so-called \textit{prefundamental representations} $\mathbb{L}_{i,\alpha}$  of the
Borel subalgebra $U_q(\wh{\mathfrak{b}}_+)\subset U_q(L\g)$ \cite{hernandez-jimbo}. 
At present, it is not known how to define the Yangian analogues of the operators  $T^{\FH}_i(z,\ul{t})$, and the polynomiality of $T_i(u)$ established in Section \ref{sec:baxter} does not rely on this approach\footnote{See, however, \cite{hernandez-zhang} and \S\ref{ssec:shifted} for an interpretation of $T_i(u)$ itself as a transfer matrix.}. Rather, it makes direct use of the definition of the sets $\sigma_i(V)$ which are at the heart of the present article.

\subsection{Poles and tensor products}\label{ssec:results-III}
In the last two sections of this paper (Sections \ref{sec:cyc-R} and \ref{sec:yangian-Dbl}), we shift our attention to highlighting the role played by the sets of poles $\sigma_i(V)$ in two different representation theoretic problems for $\Yhg$. The first, which is the focus of Section \ref{sec:cyc-R}, is the problem of determining whether a tensor
product of two irreducible representations of $\Yhg$ is highest weight or irreducible. 
This has been studied at great length 
in the literature; see \cite{chari-pressley-dorey, chari-pressley-RepsChar, guay-tan, molev-yangian, nazarov-tarasov-I, nazarov-tarasov-II,  tan-braid} in addition to
\cite{AK-qAffine, chari-braid, chari-pressley-qaffine, frenkel-mukhin, 
frenkel-reshetikhin-qchar, Hern-simple, Hern-cyclic, kashiwara-level-zero, vv-standard}
in the quantum loop case. The relation between the sets $\sigma_i(V)$ and this problem established in the present article is given by the following theorem, which combines Theorem \ref{T:cyclic} and Corollary \ref{C:irr-cond}.
\begin{thm}\label{thm:intro-2}
Let $V$ and $W$ be finite-dimensional irreducible representations of $\Yhg$ with Drinfeld polynomials $(P_i(u))_{i\in \bfI}$ and $(Q_i(u))_{i\in \bfI}$, respectively. Then the tensor product $V\otimes W$ is a highest weight module provided 
\begin{equation*}
\zeroes(Q_i(u+\hbar d_i))\subset \C\setminus \sigma_i(V) \quad \forall\quad i\in \bfI.
\end{equation*}
If in addition $\zeroes(P_i(u+\hbar d_i)) \subset\C\setminus \spec_i(W)$ for all $i\in \bfI$, then $V\otimes W$ is irreducible.
\end{thm}
We note that, in light of Theorem \ref{thm:intro-1}, these are concrete, computable cyclicity and irreducibility criteria, expressible in terms of the zeroes of the Drinfeld polynomials of $V$ and $W$ and the Cartan data of $\g$. We refer the reader to Corollary \ref{C:sln-comb} and \eqref{cyc-rmks:3} of Section \ref{ssec:cyclic-rmks} for a detailed example in the $\g=\sl_n$ case. 

Our proof of the above theorem relies on the observation that the sets $\sigma_i(V)$ feature indirectly in an argument first given for $U_q(L\g)$ in the work of Chari \cite{chari-braid}, which was translated to the Yangian setting by Guay and Tan \cite{guay-tan,tan-braid}. At the same time, a connection between the sets $\sigma_i(V)$ and the cyclicity of $V\otimes W$ is to be expected from the point of view of $q$-characters, Baxter polynomials, and their relation to the problem of describing the poles of the normalized rational $R$-matrix associated to the tensor product $V\otimes W$ \cite{frenkel-mukhin, frenkel-reshetikhin-qchar, frenkel-hernandez}, which is well-known to be closely intertwined with its cyclicity and irreduciblity \cite{AK-qAffine, chari-pressley-singularities, Hern-simple, Hern-cyclic, kashiwara-level-zero}. 

In fact, the reader may recognize that, when $\g$ is simply laced and $V=L_{\varpi_j}$, the formula for $\mathcal{Q}_{i,V}^\g(u)$ given in \eqref{intro:QiV} recovers, up to a shift, the rational degeneration of the uniform formula for the \textit{denominator} of the rational $R$-matrix associated to the $U_q(L\g)$-analogue of the tensor product $L_{\varpi_j}\otimes L_{\varpi_i}$ obtained by Fujita in \cite[Thm.~2.10]{fujita2020}; see also \cite[Prop.~6.5]{fujita-oh} and Section \ref{ssec:conjecture} below. These polynomials have a rich history and feature in many modern representation theoretic constructions for $U_q(L\g)$; see, for example, \cite{AK-qAffine, DaOk94, fujita2020, fujita-oh, kashiwara-level-zero, kkk, kkop, OhScr19} and references therein. 
We note, however, that the Yangian counterpart of this story has been far less developed. Nonetheless, in Sections \ref{ssec:R-matrix}--\ref{ssec:conjecture} we make some preliminary strides towards interpreting our results on the sets $\sigma_i(V)$ in terms of the denominators of rational $R$-matrices; see in particular Propositions \ref{P:R-pole} and \ref{P:fun-denom}, in addition to Conjecture \ref{conj:dij}. We intend to return to this topic in future work.

\subsection{Poles and the Yangian double}\label{ssec:results-IV}

The second representation theoretic problem we consider in relation to the sets $\{\sigma_i(V)\}_{i\in \bfI}$ is that of characterizing the finite-dimensional representation theory of the so-called \textit{Yangian double} $\mathrm{D}\Yhg$. This quantum algebra was introduced in generality in the foundational work \cite{khoroshkin-tolstoy} of Khoroshkin and Tolstoy in connection with a number of conjectures describing the quantum double of the Yangian and its universal $R$-matrix; see also \cite{enriquez-PBW, enriquez-quasi, Io96, JYL-20, khoroshkin-central, Curtis-DYg}, for instance, in addition to the recent articles \cite{andrea-sachin-curtis, Curtis-DYg2}.  

%
It has long been understood that one should be able to obtain representations of $\mathrm{D}\Yhg$ from those of $\Yhg$ by re-expanding the generating series $\{\xi_i(u),x_i^\pm(u)\}_{i\in \bfI}$ at points on the extended complex plane other than $u=\infty$; see \cite[\S1]{khoroshkin-tolstoy} and \cite[\S4]{Io96}. In Proposition \ref{pr:DYg-Yg}, we both clarify and strengthen this statement by constructing an isomorphism of categories
\begin{equation*}
\mathrm{Rep}^0_{fd}(\Yhg)\isom \mathrm{Rep}_{fd}(\mathrm{D}\Yhg),
\end{equation*}
where $\mathrm{Rep}_{fd}(\Yhg)$ and $\mathrm{Rep}_{fd}(\mathrm{D}\Yhg)$ are the categories of finite-dimensional representations of $\Yhg$ and $\mathrm{D}\Yhg$, respectively, and $\mathrm{Rep}^0_{fd}(\Yhg)$ is the full subcategory of $\Ryang$ consisting of those $V$  which satisfy
\begin{equation*}
\spec(V):=\bigcup_{i\in \bfI} \spec_i(V) \subset \C^\times.
\end{equation*}
When coupled with our results on the sets $\spec_i(V)$, this  identification allows us to draw a number of conclusions about the category $\mathrm{Rep}_{fd}(\mathrm{D}\Yhg)$. For instance, in Corollary \ref{C:RepDYg} it will be shown that, although the coproduct $\Delta$ on $\Yhg$ (see \S\ref{ssec:delta}) does not extend\footnote{For $\g=\sl_2$, one
can deduce that any such extension must be given as in Part (ii)
of \cite[Prop.~2.1]{khoroshkin-tolstoy}. However, it is easy to see that
these formulae do not converge in $\mathrm{D}\Yhsl{2}^{\otimes 2}$.} to a coproduct
on $\mathrm{D}\Yhg$ with values in the algebraic tensor product $\mathrm{D}\Yhg^{\otimes 2}$,  $\mathrm{Rep}^0_{fd}(\Yhg)$ is a tensor closed Serre subcategory of $\mathrm{Rep}_{fd}(\Yhg)$. Furthermore, Proposition \ref{pr:DYg-Yg} paves  the way for the following classification theorem (Theorem \ref{T:DYg-class}), which provides the final main result of this article.
\begin{thm}\label{thm:intro-3}
The isomorphism classes of finite-dimensional irreducible
representations of $\mathrm{D}\Yhg$ are parametrized by $\bfI$-tuples
of monic polynomials $\ul{P}=(P_i(u))_{i\in \bfI}$ satisfying the condition
\begin{equation*}
\zeroes(P_j(-u))\subset \C\setminus \sigma(L_{\varpi_j}) \quad \forall\; j\in \bfI. 
\end{equation*}
Moreover, if $\g$ is simply laced then $\sigma(L_{\varpi_j})$ is independent of $j$, and
is given by
\begin{equation*}
\sigma(L_{\varpi_j})=\hbar\left\{\frac{b}{2}:b\in 
\Z\;\text{ with }\;  0\leq b\leq 2\kappa-2\right\}.
\end{equation*}
\end{thm}
%
%
 To obtain the explicit description of the sets $\sigma(L_{\varpi_j})$ given in the second assertion of Theorem \ref{thm:intro-3}, we apply Theorem \ref{thm:intro-1} in conjunction with a combinatorial description of the entries $v_{ij}^{(r)}$ of the inverse $q$-Cartan matrix obtained by Hernandez and Leclerc in  \cite[Prop.~2.1]{hernandez-leclerc2}; see also \cite[Thm.~4.8]{fujita-oh} and Section \ref{ssec:Fuj-Her} below. This  characterization was used in \cite{hernandez-leclerc2} to establish that the specialized quantum Grothendieck ring $\mathscr{K}_{t=\sqrt{|F|}}$ of a certain tensor category $\mathscr{C}_\Z$ of finite-dimensional $U_q(L\g)$-modules  is isomorphic to the derived Hall algebra $\mathsf{DH}(\mathsf{Q})$ of the bounded derived category of $\mathrm{Rep}(\mathsf{Q})$ over a finite field $F$, where $\mathsf{Q}$ is a fixed quiver structure on the Dynkin diagram of $\g$.
\subsection{Outline}
The article is organized as follows.
Section \ref{sec:yangians} reviews the definition and basic properties of the Yangian $\Yhg$, while Section \ref{sec:support} surveys its finite-dimensional representation theory, including the definition of the sets of poles $\sigma_i(V)$ and their first properties (Proposition \ref{pr:sigma-i}). In Section \ref{sec:baxter}, we introduce the abelian transfer operators $\{T_i(u)\}_{i\in \bfI}$ associated to a finite-dimensional highest weight module $V$, establish their polynomiality, and show that $\sigma_i(V)$ is precisely the set of zeroes of the eigenvalue $\mathcal{Q}_{i,V}^\g(u)$ of $T_i(u)$ associated to the lowest weight space of $V$ (Theorem \ref{thm:Baxter}). In Section \ref{sec:Combinatorial}, we derive the formula \eqref{intro:QiV} for $\mathcal{Q}_{i,V}^\g(u)$, from which the description of $\sigma_i(V)$ given in Theorem \ref{thm:intro-1} follows immediately (Theorem \ref{thm:Comb}). Afterwards, we prove in Section \ref{sec:char} that the set of poles of an arbitrary $V\in \Ryang$ is equal to the set of poles of the eigenvalues of the single operator $\xi_i(u)$, and is thus encoded by the $q$-character of $V$ (Theorem \ref{thm:diagonal-poles}). In  Sections \ref{sec:cyc-R} and \ref{sec:yangian-Dbl}, we establish the results on cyclicity/irreducibility of tensor products, denominators of rational $R$-matrices, and the representation theory of the Yangian double which have been outline in Sections \ref{ssec:results-III} and \ref{ssec:results-IV}, respectively. Finally, 
Appendix \ref{sec:App-B} translates certain properties of denominators of the rational $R$-matrices associated to fundamental representations of $U_q(L\g)$ \cite{AK-qAffine} to the Yangian setting (Proposition \ref{P:cyc-symm}). These are applied in Section \ref{ssec:R-fun}. 
\subsection{Acknowledgments}

The first author was supported through the Simons foundation
collaboration grant 526947.
The second author gratefully acknowledges the support of the
Natural Sciences and Engineering Research Council of Canada (NSERC), provided
via the Postdoctoral Fellowships and Discovery Grants programs (Grant RGPIN-2022-03298 and DGECR-2022-00440). We are extremely
grateful to the anonymous referee for pointing out the possible connection
between the sets of poles and the work of Frenkel--Hernandez \cite{frenkel-hernandez},
which helped us to significantly simplify our arguments, and
strengthen our results. In addition, we would like to thank David Hernandez, Ryo Fujita and Sasha Tsymbaliuk for several enlightening discussions and helpful remarks. 

\section{Yangians}\label{sec:yangians}

\subsection{}\label{ssec:g}
Let $\g$ be a finite-dimensional, simple Lie algebra over $\C$,
and let $(\cdot,\cdot)$ be the invariant bilinear form on
$\g$ normalized so as to have the squared length of every short
root equal to $2$. Let $\h\subset\g$ be a Cartan subalgebra of
$\g$, $\{\alpha_i\}_{i\in\bfI}\subset\h^*$ a basis of
simple roots of $\g$ relative to $\h$ and $a_{ij} = 2
\frac{(\alpha_i,\alpha_j)}{(\alpha_i,\alpha_i)}$ the
entries of the corresponding Cartan matrix $\bfA$.
Set $d_i = \frac{(\alpha_i,\alpha_i)}{2}\in\{1,2,3\}$,
so that $d_ia_{ij} = d_ja_{ji}$ for every $i,j\in\bfI$.

\subsection{The Yangian $\Yhg$}\label{ssec:yhg}
Let $\hbar\in\nC$ be fixed throughout. The Yangian $\Yhg$
is the unital, associative algebra over $\C$ generated
by elements $\{\xi_{i,r},x_{i,r}^{\pm}\}_{i\in\bfI,r\in\N}$, subject
to the following relations:
\begin{enumerate}[label=(Y\arabic*)]\setlength{\itemsep}{5pt}
\item\label{Y1} For every $i,j\in\bfI$ and $r,s\in\N$, we have
$\ds [\xi_{i,r},\xi_{j,s}]=0$.

\item\label{Y2} For every $i,j\in\bfI$ and $s\in\N$, we have
$\ds [\xi_{i,0},x_{j,s}^{\pm}] = \pm d_ia_{ij} x_{j,s}^{\pm}$.

\item\label{Y3} For every $i,j\in\bfI$ and $r,s\in\N$, we have
\[
[\xi_{i,r+1},x_{j,s}^{\pm}] - [\xi_{i,r},x_{j,s+1}^{\pm}]
= \pm\hbar\frac{d_ia_{ij}}{2} \lp
\xi_{i,r}x_{j,s}^{\pm} + x_{j,s}^{\pm}\xi_{i,r}
\rp\ .
\]

\item\label{Y4}  For every $i,j\in\bfI$ and $r,s\in\N$, we have
\[
[x^{\pm}_{i,r+1},x_{j,s}^{\pm}] - [x^{\pm}_{i,r},x_{j,s+1}^{\pm}]
= \pm\hbar\frac{d_ia_{ij}}{2} \lp
x^{\pm}_{i,r}x_{j,s}^{\pm} + x_{j,s}^{\pm}x^{\pm}_{i,r}
\rp\ .
\]

\item\label{Y5}  For every $i,j\in\bfI$ and $r,s\in\N$, we have
$\ds [x_{i,r}^+,x_{j,s}^-] = \delta_{ij}\xi_{i,r+s}$.

\item\label{Y6}  Let $i\neq j\in\bfI$. Set $m=1-a_{ij}$.
Then, for every $r_1,\ldots,r_m,s\in\N$, we have:
\[
\sum_{\pi\in\Sym_m} \lb x^{\pm}_{i,r_{\pi(1)}},\lb
x^{\pm}_{i,r_{\pi(2)}},\lb \cdots
\lb x^{\pm}_{i,r_{\pi(m)}}, x^{\pm}_{j,s}\rb \cdots
\rb\rb\rb = 0.
\]
\end{enumerate}

\subsection{Remark}\label{ssec:rem-serre}
The relation \ref{Y6} follows from \ref{Y1}--\ref{Y3} and the special
case of \ref{Y6} when $r_1=\ldots=r_m=0$ (see \cite[Lemma 1.9]{levendorskii}).
The latter automatically holds on \fd representations of
the algebra defined by relations \ref{Y2} and \ref{Y5} alone
(see \cite[Prop. 2.7]{sachin-valerio-2}).

\subsection{The inclusion $U(\g)\subset \Yhg$}\label{ssec:g->Y}

Let $\nu:\h\to\h^*$ be the isomorphism determined by $(\cdot,\cdot)$. For each fixed $i\in \bfI$, 
set $h_i=\nu^{-1}(\alpha_i)/d_i$ and choose root vectors 
$x_i^\pm \in \g_{\pm \alpha_i}$ such that $[x_i^+,x_i^-]=d_ih_i$. Then the defining relations \ref{Y1}--\ref{Y6} of $\Yhg$ imply that the assignment 
\begin{equation*}
x_i^\pm \mapsto x_{i,0}^\pm \quad \text{ and }\quad d_ih_i\mapsto \xi_{i,0} \quad \forall\; i\in \bfI
\end{equation*}
extends to an algebra homomorphism $U(\g)\to \Yhg$. It is a well-known consequence of the Poincar\'{e}--Birkhoff--Witt theorem for $\Yhg$ \cite{levendorskii-PBW} (see also \cite[Thm. B.6]{FiTsWe19} and
\cite[Prop. 2.2]{GRW-Equivalences})
that this is an embedding (see \cite[\S 2.7]{sachin-valerio-curtis}, for instance).
We shall freely use this fact and view $U(\g)\subset \Yhg$.

\subsection{Shift automorphism}\label{ssec:shift}
The group of translations of the complex plane acts on $\Yhg$ by algebra automorphisms $\{\tau_a\}_{a\in \C}$ defined by
\[
\tau_a(y_r) = \sum_{s=0}^r \cbin{r}{s} a^{r-s} y_s,
\]
where $r\in\N$ and $y$ is one of $\xi_i,x^{\pm}_i$. 
Equivalently, $\tau_a$ is given by
\begin{equation*}
\tau_a(y(u)) = y(u-a),
\end{equation*}
where  $y(u)=\xi_i(u),x_i^{\pm}(u)\in \Pseries{\Yhg}{u^{-1}}$ are the generating series
 defined by
\[
\xi_i(u) = 1+\hbar\sum_{r\geq 0} \xi_{i,r}u^{-r-1}\aand
x_i^{\pm}(u) = \hbar\sum_{r\geq 0}x_{i,r}^{\pm} u^{-r-1}
\]
for each $i\in \bfI$.
Given a representation $V$ of $\Yhg$ and $a\in\C$, we set
$V(a) := \tau_a^*(V)$.

\subsection{Diagram automorphisms}\label{ssec:DDauto}
Let $\omega$ be an automorphism of the Dynkin diagram of $\g$.
That is, $\omega\in\Sym_{\bfI}$ is such that $a_{ij} = a_{\omega(i),\omega(j)}$
for every $i,j\in\bfI$. Then $\omega$ defines an automorphism
of $\Yhg$, determined by
\[
\xi_i(u)\mapsto \xi_{\omega(i)}(u) \aand
x^{\pm}_i(u) \mapsto x^{\pm}_{\omega(i)}(u),
\]
for every $i\in\bfI$.
For a representation $V$ of $\Yhg$, we will denote the pull-back
$\omega^*(V)$ by $V^{\omega}$.

\subsection{Cartan involution}\label{ssec:CPauto}
By \cite[Prop.~2.4]{chari-pressley-dorey}, the Cartan involution $x_i^\pm \mapsto x_i^\mp$ of $\g$ extends to an involutive algebra automorphism $\varphi$ of $\Yhg$, determined by 
\[
\varphi(\xi_i(u)) = \xi_i(-u) \aand
\varphi(x_i^{\pm}(u)) = -x_i^{\mp}(-u)
\]
for all $i\in \bfI$. For a representation $V$ of $\Yhg$, we set $V^{\varphi} = \varphi^*(V)$. 

\subsection{Rank one restriction}\label{ssec:restriction}

For a fixed index $i\in \bfI$, we denote by $Y_\hbar(\g_i)$
the subalgebra of $\Yhg$ generated by
$\{\xi_{i,r},x_{i,r}^{\pm}\}_{r\in\N}$. 
%
%
One has $Y_\hbar(\g_i)\cong Y_{d_i\hbar}(\sl_2)$, with an isomorphism $\varphi_i:Y_{d_i \hbar}(\sl_2)\isom Y_\hbar(\g_i)$ given by  
\begin{equation*}
\varphi_i(\xi(u))=\xi_i(u) \aand \varphi_i(x^\pm(u))=\sqrt{d_i}x_i^\pm(u),
\end{equation*}
where we have dropped the indexing subscript for the generating series (and generators) of $Y_{d_i \hbar}(\sl_2)$. Here we recall the well-known fact that, for any fixed $\g$ and $\hbar,\lambda\in \C^\times$, we have $\Yhg\cong Y_\lambda(\g)$, with an isomorphism $\Yhg\isom Y_\lambda(\g)$ given by 
\begin{equation*}
\xi_i(u)\mapsto \xi_i(\lambda u/\hbar) \aand x_i^\pm(u) \mapsto x_i^\pm(\lambda u/\hbar)
\end{equation*} 
for each $i\in \bfI$. In particular, we have $Y_\hbar(\g_i)\cong Y_{d_i\hbar}(\sl_2)\cong Y_\hbar(\sl_2)$.

\subsection{Coproduct}\label{ssec:delta}
The Yangian $\Yhg$ is known to be a Hopf algebra with coproduct $\Delta:\Yhg\to \Yhg\otimes \Yhg$ defined as follows. For each $i\in \bfI$, set 
\begin{equation*}
t_{i,1} = \xi_{i,1}-\frac{\hbar}{2}\xi_{i,0}^2\in \Yhg^\h.
\end{equation*}
Then, the set of elements $\{\xi_{i,0},x_{i,0}^\pm, t_{i,1}\}_{i\in \bfI}$ generates $\Yhg$ as an algebra, and $\Delta$ is 
uniquely determined by the
formulae 
\begin{gather*}
\Delta(y_0) = y_0\otimes 1 + 1\otimes y_0,\; \text{ where }\; y=\xi_i,x_i^{\pm},\\
\Delta(t_{i,1}) = t_{i,1}\otimes 1 + 1\otimes t_{i,1}
-\hbar \sum_{\alpha\in R_+} (\alpha_i,\alpha) x_{\alpha}^-
\otimes x_{\alpha}^+,
\end{gather*}
for all $i\in \bfI$, where $R_+$ is the set of positive roots of $\g$ and $x^{\pm}_{\alpha}
\in\g_{\pm\alpha}\subset \Yhg$ are chosen so as to have $(x^+_{\alpha},x^-_{\alpha})=1$
and $x_{\alpha_i}^\pm=x_i^\pm$.
We refer the reader to \cite[\S 4.2--4.5]{guay-nakajima-wendlandt}
for a proof that these formulae indeed define a coassociative algebra homomorphism.

An explicit formula for $\Delta(\xi_i(u))$ is not known in general.
However, for our purposes
it will suffice to know that $\{\xi_i(u)\}_{i\in\bfI}$ are all group-like
modulo certain strictly triangular terms. To state this precisely,
let $Y^{\leq 0}$ (resp. $Y^{\geq 0}$) be the subalgebra of $\Yhg$
generated by $\{\xi_{i,r},x^-_{i,r}\}_{i\in\bfI,r\in\N}$
(resp. $\{\xi_{i,r},x^+_{i,r}\}_{i\in\bfI,r\in\N}$). These algebras
are graded by $Q_+$ (the positive cone in the root lattice of $\g$). Then,
we have the following.
\begin{prop}\label{pr:delta-xi-tr}
For each $i\in\bfI$, $\Delta(\xi_i(u))$ satisfies 
\[
\Delta(\xi_i(u)) = \xi_i(u)\otimes\xi_i(u) + 
\Pseries{\lp\bigoplus_{\beta>0} Y^{\leq 0}_{-\beta}\otimes Y^{\geq 0}_{\beta}\rp\!}{u^{-1}}.
\]
\end{prop}
This result can be easily deduced from \cite[Thm.~4.1]{sachin-valerio-curtis}.
Indeed, this theorem implies that, for each $i\in \bfI$, $\Delta(\xi_i(u))$ satisfies 
\[
(\tau_s\otimes 1)\Delta(\xi_i(u)) = 
\mathcal{R}^-(s)^{-1}\cdot (\xi_i(u-s)\otimes \xi_i(u))\cdot
\mathcal{R}^-(s),
\]
where $\mathcal{R}^-(s)\in(Y^{\leq 0}\otimes Y^{\geq 0})[\![s^{-1}]\!]$ is the lower triangular factor in the Gauss decomposition of the universal $R$-matrix of the Yangian (see \S\ref{ssec:R-matrix}), and is of the form 
\begin{equation*}
\mathcal{R}^-(s)=1+\sum_{\beta>0}\mathcal{R}^-(s)_\beta \quad \text{ with }\quad \mathcal{R}^-(s)_\beta\in (Y^{\leq 0}_{-\beta}\otimes 
Y^{\geq 0}_{\beta})[\![s^{-1}]\!].
\end{equation*}
 We give a more direct proof below,
along the lines of \cite[Prop. 2.8]{chari-pressley-singularities}.
\begin{pf}
We begin by observing that, for each $i\in\bfI$, $r\in\Z_{\geq 0}$ and
$\beta\in R_+\setminus\{\alpha_i\}$, we have $[x^+_{i,r},
x^-_{\beta}]\in Y_{-\beta+\alpha_i}^{\leq 0}$, where $Y^{\leq 0}_{-\beta+\alpha_i}=\{0\}$ for $\beta-\alpha_i\notin Q_+$. If instead $\beta=\alpha_i$,
then $x_\beta^-=x_i^-$ and $[x^+_{i,r},x^-_{i,0}] = \xi_{i,r}$.
Using these observations together with the identity $[t_{i,1},x^+_{i,r}] = 2d_ix^+_{i,r+1}$, the formula
for $\Delta(t_{i,1})$ given above, and a simple induction argument on $r$, we find that 
\[
\Delta(x^+_{i,r}) = x_{i,r}^+\otimes 1 + 1\otimes x^+_{i,r}
+\hbar\sum_{k=1}^r \xi_{i,k-1}\otimes x^+_{i,r-k} \mod
\bigoplus_{\beta>0} Y^{\leq 0}_{-\beta}\otimes Y^{+}_{\beta+\alpha_i},
\]
where $Y^{+}$ is the $Q_+$-graded subalgebra of $Y^{\geq 0}$ generated by $\{x_{j,r}^+\}_{j\in \bfI,r\in \N}$. 
Taking the commutator of this identity with $\Delta(x_{i,0}^-) = x^-_{i,0}\otimes 1
+1\otimes x^-_{i,0}$ and using that $\mathrm{ad}(x^-_{i,0})(Y^+_{\beta+\alpha_i})\subset Y^{\geq 0}_\beta$ for all nonzero $\beta\in Q_+$, we obtain
\[
\Delta(\xi_{i,r}) = \xi_{i,r}\otimes 1 + 1\otimes \xi_{i,r}
+\hbar\sum_{k=1}^r \xi_{i,k-1}\otimes \xi_{i,r-k} \mod
\bigoplus_{\beta>0} Y^{\leq 0}_{-\beta}\otimes Y^{\geq 0}_{\beta},
\]
which is exactly what we wanted to prove.
\end{pf}

\section{Finite-dimensional representations and their poles}\label{sec:support}

\subsection{Rational currents}\label{ssec:rational}
The following rationality property was obtained in
\cite[Prop. 3.6]{sachin-valerio-2}. It's analogue for the quantum loop algebra $U_q(L\g)$ was obtained earlier; see \cite[\S6]{beck-kac} and \cite[Prop.~3.8]{hernandez-drinfeld}.
\begin{prop}\label{pr:rational}
Let $V$ be a $\Yhg$-module on which $\h$
acts semisimply with \fd weight spaces. Then, for every
$\g$-weight $\mu\in\h^*$ of $V$, the generating series
\[
\xi_i(u)\in\Pseries{\End(V_\mu)}{u^{-1}} \aand
x_i^{\pm}(u)\in\Pseries{\Hom(V_\mu,V_{\mu\pm\alpha_i})}{u^{-1}}
\]
are the expansions at $\infty$ of rational functions of $u$. Explicitly, let $t_{i,1}=\xi_{i,1}-\frac{\hbar}{2}\xi_{i,0}^2$, as in Section \ref{ssec:delta}. Then the relations 
\[
x_i^{\pm}(u) = \hbar \lp u \mp \frac{1}{2d_i}\ad(t_{i,1})\rp^{-1} \cdot
x_{i,0}^{\pm} 
\quad \text{ and }\quad 
 \xi_i(u) = 1 + [x_i^+(u),x_{i,0}^-]
\]
determine $x_i^\pm(u)$ and $\xi_i(u)$ as the expansions of operator valued rational functions on each $\g$-weight space of $V$.
\end{prop}

\subsection{Finite-dimensional representations}
\label{ssec:reln-current}
A \fd representation $V$ of $\Yhg$ is thus completely determined
by rational, $\End(V)$-valued functions $\{\xi_i(u),x_i^{\pm}(u)\}_{i\in\bfI}$ satisfying  $\xi_i(\infty) = \Id_V$ and $x_i^{\pm}(\infty)=0$. These
rational functions are subject to the following relations
(see \cite[Prop. 2.3]{sachin-valerio-2}).

\begin{enumerate}[label=($\Y$\arabic*)]\setlength{\itemsep}{5pt}
\item\label{sY1} For every $i,j\in\bfI$, $\ds [\xi_i(u),\xi_j(v)]=0$.

\item\label{sY23} For every $i,j\in\bfI$, let $a=\hbar d_ia_{ij}/2$.
Then we have:
\[
(u-v\mp a)\xi_i(u)x_j^{\pm}(v) = (u-v\pm a) x_j^{\pm}(v)\xi_i(u)
\mp 2a x_j^{\pm}(u\mp a)\xi_i(u).
\]

\item\label{sY4} For every $i,j\in\bfI$, let $a=\hbar d_ia_{ij}/2$.
Then we have:
\[
(u-v\mp a)x_i^{\pm}(u)x_j^{\pm}(v) = 
(u-v\pm a)x_j^{\pm}(v)x_i^{\pm}(u) + \hbar \lp
[x_{i,0}^{\pm},x_j^{\pm}(v)] - [x_i^{\pm}(u),x_{j,0}^{\pm}]
\rp.
\]

\item\label{sY5} For every $i,j\in\bfI$:
\[
[x_i^+(u),x_j^-(v)] = \delta_{ij} \frac{\hbar}{u-v} (\xi_i(v)-\xi_i(u)).
\]
\end{enumerate}

\subsection{Poles of finite-dimensional representations}\label{ssec:poles}

Due to Proposition \ref{pr:rational}, for each finite-dimensional representation $V$ of $\Yhg$ and index $i\in \bfI$, we may introduce the \textit{$i$-th set of poles} of $V$ as the subset 
\begin{equation*}
\sigma_i(V):=\{\text{Poles of }\xi_i(u),x_i^{\pm}(u)\in\End(V)(u)\}\subset \C.
\end{equation*}
These sets are the main objects of study in this article. Similarly, we define the \textit{full set of poles} of $V$ to be the union 
\begin{equation*}
\sigma(V):=\bigcup_{i\in \bfI}\sigma_i(V).
\end{equation*}
That is,  $\sigma(V)$ consists of all poles of the rational functions $\{\xi_i(u),x_i^\pm(u)\}_{i\in \bfI}$. By definition, $\sigma_i(V)$ is determined by the $Y_{d_i\hbar}(\sl_2)\cong Y_\hbar(\g_i)\subset Y_\hbar(\g)$ action on $V$, in that 
\begin{equation}\label{sigma-sl2}
\sigma_i(V)=\sigma(\varphi_i^\ast(V)) \quad \forall \quad i\in \bfI,
\end{equation}
where $\varphi_i:Y_{d_i\hbar}(\sl_2)\isom Y_\hbar(\g_i)$ is the isomorphism of Section \ref{ssec:restriction}.  

Let us now define $\spec(\xi_i(u);V)$ and $\spec(x_i^{\pm}(u);V)$ to be the subsets of $\sigma_i(V)$
consisting of the finite sets of poles of the $\End(V)$-valued rational functions
$\xi_i(u)$ and $x_i^{\pm}(u)$, respectively. For $z\in\C$,
let $\order_z(\xi_i(u);V)$ (resp. $\order_z(x_i^{\pm}(u);V)$)
be the order of the pole of $\xi_i(u)$ (resp. $x_i^{\pm}(u)$)
at $u=z$.
By convention, we set $\order_z(\xi_i(u);V)$ (resp. $\order_z(x_i^{\pm}(u);V)$) to be $0$ if $z$ is not a pole of the underlying function. 

The following preliminary result shows that $\sigma_i(V)$ is encoded by any one of the individual operators $\xi_i(u)$, $x_i^+(u)$ and $x_i^-(u)$. 
\begin{prop}\label{pr:sigma-i}
Let $V$ be a finite-dimensional $\Yhg$-module. Then
\[
\spec(\xi_i(u);V) = \spec(x_i^{\pm}(u);V) = \spec_i(V) \quad \forall\quad i\in \bfI.
\]
Moreover, $\order_z(\xi_i(u);V) = \order_z(x_i^{\pm}(u);V)$
for every $z\in \C$.
\end{prop}
\begin{pf}
This is clearly a statement about \fd representations of
$\Yhsl{2}$. To simplify notation, we will drop the indexing subscript $i$ from the generating series of $\Yhsl{2}$ and their coefficients, as in \S\ref{ssec:restriction}. 

Using the relation $\ds \xi(u) = 1 + [x^+(u),x^-_0]
= 1 + [x_{0}^+,x^-(u)]$, we conclude that
\[
\spec(\xi(u);V) \subset \spec(x^{\pm}(u);V) \aand
\order_z(\xi(u);V) \leq \order_z(x^{\pm}(u);V)
\]
for every $z\in\C$. To obtain the converse, let us assume
that $z\in\C$ is a pole of $x^+(u)$. Let $N = \order_z(x^+(u);V)$
and $\ds\mathsf{X} := \lim_{u\to z} (u-z)^N x^+(u)$.
Assume, for the sake of a contradiction, that $\order_z(\xi(u);V)<N$.
Multiplying both sides of the relation $\xi(u) = 1+[x^+(u),x^-_0]$
by $(u-z)^N$ and letting $u\to z$, we get
\[
[\mathsf{X}, x^-_0] = 0.
\]
To obtain the desired contradiction, we regard $\End(V)$ as an $\sl_2$-representation via
\[
h\mapsto \ad(\xi_0), \qquad e\mapsto \ad(x^+_0), \qquad
f\mapsto \ad(x_0^-).
\]
Then $\mathsf{X}\in\End(V)$ is a non-zero lowest-weight vector of weight $+2$ in the \fd
$\sl_2$-representation $\End(V)$. This contradicts
the finite-dimensionality of $\End(V)$.
\end{pf}

\subsection{Drinfeld polynomials}\label{ssec:dp}
In Section \ref{sec:Combinatorial}, we will compute the sets of poles $\sigma_i(V)$ explicitly in the case where $V$ is a finite-dimensional irreducible representation of $\Yhg$. 
To this end, we recall that such representations are classified by
$\bfI$-tuples of monic polynomials. More precisely, one has the following classification theorem established by Drinfeld in the foundational work \cite{drinfeld-yangian-qaffine}.
\begin{thm}\label{thm:dp}
Let $V$ be a \fd \irr representation of $\Yhg$. Then, there exists
a unique, up to scalar multiplication, non-zero vector $\vac\in V$ and an $\bfI$-tuple
of monic polynomials $(P_i(u))_{i\in\bfI}\in \C[u]^\bfI$ such that:
\begin{enumerate}[font=\upshape]
\item $V$ is generated as a $\Yhg$-module by $\vac$.
\item For each $i\in \bfI$, one has 
\begin{equation*}
x_i^+(u)\vac=0\quad \text{ and }\quad
\xi_i(u)\vac
= \frac{P_i(u+d_i\hbar)}{P_i(u)}\vac.
\end{equation*}
\end{enumerate}
Conversely, given an $\bfI$-tuple of monic polynomials
$(P_i(u))_{i\in\bfI}$, there is a unique, up to isomorphism, \fd
\irr representation of $\Yhg$ containing a non-zero vector
$\vac$ which satisfies the properties listed above.
\end{thm}
Given an $\bfI$-tuple $\ul{P} = (P_i(u))_{i\in\bfI}$ of monic polynomials, the corresponding \fd \irr representation of $\Yhg$ will be
denoted by $L(\ul{P})$. The polynomials $(P_i(u))$ are called
the \textit{Drinfeld polynomials} associated to $L(\ul{P})$. In the special case where there is $j\in \bfI$ and $a\in \C$ such that 
\begin{equation*}
P_i(u)
=
\left\{\begin{array}{cl} 1 & \text{if } i\neq j \\
u-a & \text{if } i=j \end{array}\right.
\end{equation*}
the module $L(\ul{P})$ is denoted $L_{\varpi_j}(a)$ (or simply $L_{\varpi_j}$ if $a=0$) and called a \textit{fundamental representation} of $\Yhg$, where $\varpi_j\in \h^\ast$ denotes the $j$-th fundamental weight of $\g$, defined by $\varpi_j(h_i)=\delta_{ij}$ for all $i\in \bfI$. Here we note that the $\g$-weight $\lambda\in \h^\ast$ of the vector $\Omega\in L(\ul{P})$ from Theorem \ref{thm:dp} is given by $\lambda=\sum_{i\in \bfI}\deg(P_i)\varpi_i$.

Encoded in the above theorem is the assertion that every finite-dimensional irreducible module $V$ of $\Yhg$ is a \textit{highest weight module}. More generally, a $\Yhg$-module $V$ is said to be a highest weight module with highest weight 
$\ul{\mu}=(\mu_i(u))_{i\in \bfI}$, where $\mu_i(u)\in 1+u^{-1}\C[\![u^{-1}]\!]$, if it is generated by a non-zero vector $\vac\in V$ satisfying
\begin{equation*}
x_i^+(u)\vac=0\quad \text{ and }\quad
\xi_i(u)\vac
= \mu_i(u)\vac
\end{equation*}
for each $i\in \bfI$. In this case, the vector $\vac$ is called a highest weight vector.

\subsection{Lowest weight polynomials}\label{ssec:CPthm}

Similarly to the above, a $\Yhg$-module $V$ is said to be a \textit{lowest weight module} of lowest weight $\ul{\mu}=(\mu_i(u))_{i\in \bfI}$ if it is generated by a non-zero vector $\Omega$ such that $x_i^-(u)\vac=0$ and $\xi_i(u)\vac=\mu_i(u)\vac$ for each $i\in \bfI$. Equivalently, $V$ is a lowest weight module if and only if $V^\varphi=\varphi^\ast(V)$ is a highest weight module, where $\varphi$ is the Cartan involution of $\Yhg$ defined in Section \ref{ssec:CPauto}. 

By \cite[Prop.~3.7]{chari-pressley-dorey}, for instance, a finite-dimensional highest weight module $V$ of $\Yhg$ is automatically a lowest weight module. If $\lambda\in \h^\ast$ is the $\g$-weight of any highest weight vector, then any nonzero vector in the one-dimensional space $V_{w_0(\lambda)}$ is a lowest weight vector, where  $w_0$ is the longest element of the Weyl group of $\g$. 

The lowest weight of the irreducible module  $L(\ul{P})$ can also be expressed in terms of its Drinfeld polynomials $P_i(u)$, as we now recall.
Consider the involution $i\mapsto i^*$ of the Dynkin diagram of
$\g$ induced by the longest element $w_0$:
\begin{equation*}
\alpha_{i^*} = -w_0(\alpha_i) \quad \forall\; i \in \bfI.
\end{equation*}
Let $\kappa$ be $(1/4)$ times the eigenvalue of the Casimir element $C\in U(\g)$
on the adjoint representation. Explicitly, let $\theta$ be the
longest root of $\g$ and $\rho$ be half the sum of positive roots
of $\g$. Then
\[
\kappa = \frac{1}{4} (\theta,\theta+2\rho) = \frac{1}{2} \frac{(\theta,\theta)}{2}
(1+\rho(\theta^{\vee})) = \frac{1}{2} mh^{\vee},
\]
where $m=1,2,3$ if $\g$ is of type $\mathsf{ADE},\mathsf{BCF},\mathsf{G}$
respectively, and $h^{\vee} = 1+\rho(\theta^{\vee})$ is the dual Coxeter
number of $\g$.
The table below lists the
value of $\kappa$ in each type, where we follow Bourbaki's convention for the 
labels of Dynkin diagrams \cite{bourbaki-lie456}.
%
\begin{table}[H]
\begin{center}
\begin{tabular}{|c|c|c|c|c|c|c|c|c|c|}
\hline\hline
Type of $\g$ & $\mathsf{A}_r$ & $\mathsf{B_r}$ & $\mathsf{C_r}$ &
$\mathsf{D_r}$ & $\mathsf{E}_{6}$ & $\mathsf{E}_7$ & $\mathsf{E}_8$ & 
$\mathsf{F}_4$ & $\mathsf{G}_2$ \\
\hline    
$\kappa$ & \rule{0pt}{4ex}  $\ds\frac{r+1}{2}$ & $r+1$ & $2r-1$ & $r-1$ & $6$ & $9$ & $15$ &
$9$ & $6$\\[.75em]
\hline\hline
\end{tabular}
\end{center}
\caption{Values of $\kappa$}\label{table:kappa}
\end{table}
\vspace{-1em}
\noindent The following result was
obtained in \cite[Prop.~3.2]{chari-pressley-singularities} and \cite[Prop. 3.5]{chari-pressley-dorey}. 
\begin{prop}\label{pr:CPthm}
Let $\ul{P}=(P_i(u))_{i\in\bfI}$ be an $\bfI$-tuple of monic polynomials. Then the \fd irreducible $\Yhg$-module $L(\ul{P})$ satisfies
$L(\ul{P})^{\varphi} \cong L(\ul{P^{\varphi}})$, where
\[
P_i^{\varphi}(u) = (-1)^{\deg(P_{i^*})} P_{i^*}(-u+d_i\hbar-\kappa\hbar) \quad \forall\quad i\in \bfI.
\]
In particular,  the lowest weight $\ul{\mu}=(\mu_i(u))_{i\in \bfI}$ of $L(\ul{P})$ is given by
\begin{equation*}
\mu_i(u)=\frac{P_{i^\ast}(u-\hbar\kappa)}{P_{i^\ast}(u-\hbar\kappa+\hbar d_i)} \quad \forall\quad i\in \bfI.
\end{equation*}
\end{prop}
\subsection{The $q$-character of a representation}\label{ssec:q-char}
In this subsection we recall the definition and basic properties of the $q$-character
of a $\Yhg$-module $V$, following \cite{knight}. 

Let $\mathcal{L}$ denote the multiplicative group
$1+u^{-1}\C[\![u^{-1}]\!]\subset \C[\![u^{-1}]\!]$. Given
$V\in \Ryang$ and
$\ul{\mu} = (\mu_i(u))_{i\in\bfI}\in\mathcal{L}^{\bfI}$, with 
\[
\mu_i(u) = 1+\hbar\sum_{r\geq 0} \mu_{i,r}u^{-r-1},
\]
we define the generalized eigenspace $V[\ul{\mu}]$ of $V$ by
\begin{equation*}
V[\ul{\mu}] = \left\{
v\in V : \forall i\in\bfI, k\geq 0, \exists N_k>0
\text{ such that }
(\xi_{i,k}-\mu_{i,k})^{N_k}v=0
\right\}.
\end{equation*}
If $V[\ul{\mu}]$ is nonzero, then $\ul{\mu}$ is called a \textit{weight} of the $\Yhg$-module $V$, and one has the generalized weight space decomposition $V=\bigoplus_{\ul{\mu}}V[\ul{\mu}]$.  In addition:
\begin{itemize}\setlength{\itemsep}{3pt}
\item For each weight $\ul{\mu}$ of $V$, there is an ordered basis of $V[\ul{\mu}]$ for which each operator $\xi_i(u)|_{V[\ul{\mu}]}$ is triangular, with unique eigenvalue $\mu_i(u)$. 
\item The decomposition $V=\bigoplus_{\ul{\mu}}V[\ul{\mu}]$ is compatible with its $\g$-weight space decomposition, in that $V[\ul{\mu}]\subset V_{\mu}$ for $\mu=\sum_{i\in \bfI}\frac{\mu_{i,0}}{d_i}\varpi_i\in \h^\ast$.
\item  If $\ul{\mu}=(\mu_i(u))_{i\in \bfI}$ is a weight of $V$ then, by \cite[Prop.~4]{knight}, there are uniquely determined monic polynomials $P_i(u)$ and $Q_i(u)$, for each $i\in \bfI$, satisfying
\begin{equation*}
\mu_i(u)=\frac{P_i(u+d_i \hbar)}{P_i(u)}\frac{Q_i(u)}{Q_i(u+d_i \hbar)}
\end{equation*}
and $\zeroes(P_i(u))\subset \C\setminus \zeroes(Q_i(u))$, where $\zeroes(P(u))\subset \C$ is the set of roots of a given polynomial $P(u)$.  Though we shall not directly apply  this result of \cite{knight}, we note that the last relation of Proposition \ref{P:FH-Q} below provides a strengthening of it.
\end{itemize}
Let us now recall the definition of the $q$-character of $V$. 
Let $\Z[\mathcal{L}^{\bfI}]$ be the group algebra of
the direct product $\mathcal{L}^{\bfI}$ of $\bfI$-copies of $\mathcal{L}$, defined over $\Z$.
That is, $\Z[\mathcal{L}^{\bfI}]$ is the free $\Z$-module with basis consisting of all formal
exponentials $e(\ul{\mu})$, where $\ul{\mu}\in\mathcal{L}^{\bfI}$,
and multiplication defined on basis vectors by the rule
\[
e(\ul{\mu})\cdot e(\ul{\nu}) = e((\mu_i(u)\nu_i(u))_{i\in\bfI}).
\]
\begin{defn}
The $q$-character of $V$, denoted by $\chi_q(V)$, is
defined as
\[
\chi_q(V) = \sum_{\ul{\mu}\in\mathcal{L}^{\bfI}}
\dim(V[\ul{\mu}])\,e(\ul{\mu}) \in \Z[\mathcal{L}^{\bfI}].
\]
\end{defn}
This assignment gives rise to a homomorphism of commutative rings 
\begin{equation*}
\chi_q:K(\Yhg)\to \Z[\mathcal{L}^{\bfI}],
\end{equation*}
where $K(\Yhg)$ is the Grothendieck ring of $\Ryang$ equipped with multiplication induced by the coproduct $\Delta$ on $\Yhg$ from Section \ref{ssec:delta} . In particular, it was proven in  \cite[Thm.~2]{knight} that, for any pair of modules $V,W\in \Ryang$, one has 
\begin{equation*}
\chi_q(V\otimes W)=\chi_q(V)\chi_q(W).
\end{equation*}
The homomorphism $\chi_q$ is also known to be injective; see \cite[Thm.~3]{frenkel-reshetikhin-qchar}, for instance. We will not need this fact in the present article.

\section{Poles and Baxter polynomials}\label{sec:baxter}

In this section, we draw inspiration from \cite{frenkel-hernandez} and prove in Theorem \ref{thm:Baxter} that the sets of poles $\sigma_i(V$) of a given finite-dimensional highest weight module $V$ are precisely the zeros of the so-called \textit{specialized Baxter polynomials} associated to $V$. This result passes through a weak, rational version of a polynomiality result established by Frenkel and Hernandez in \cite[\S5]{frenkel-hernandez} for the quantum loop algebra $U_q(L\g)$, which we briefly review in Section \ref{ssec:FH-summary}.

\subsection{Baxter polynomiality}\label{ssec:FH-summary}
To motivate our results, we now recall the key steps from the Frenkel--Hernandez construction of the Baxter polynomials associated to representations of the quantum loop algebra $U_q(L\g)$ from \cite{frenkel-hernandez}. Let
$\mathscr{R}$ denote the universal $R$-matrix of $U_q(L\g)$.

\begin{enumerate}\setlength{\parskip}{5pt}
\item\label{FH:1} For each $i\in\bfI$ and $\alpha\in\nC$, there is an
infinite-dimensional representation $\mathbb{L}_{i,\alpha}$ of the
Borel subalgebra $U_q(\wh{\mathfrak{b}}_+)\subset U_q(L\g)$ which was introduced and
studied in detail in the work \cite{hernandez-jimbo} of Hernandez and Jimbo.
It is called a {\em prefundamental representation} in \cite{frenkel-hernandez}.

\item\label{FH:2} For any finite-dimensional representation $\mathbb{V}$ of
$U_q(L\g)$, an operator $T^{\FH}_i(z,\ul{t})$
on $\mathbb{V}$ is defined in \cite[\S 5.1]{frenkel-hernandez},
by taking the (twisted, or graded by weights)
trace over $\mathbb{L}_{i,z}$ of the evaluation of $\mathscr{R}$
on $\mathbb{L}_{i,z}\otimes\mathbb{V}$:
\[
T^{\FH}_i(z,\ul{t}) = \Tr_{\mathbb{L}_{i,z};\ul{t}}\lp
\left.\mathscr{R}\right|_{\mathbb{L}_{i,z}\otimes\mathbb{V}} \rp
\in \Pseries{\End(\mathbb{V})}{z,(t_j)_{j\in\bfI}}.
\]
Here, $\ul{t}\in (\nC)^\bfI$
are torus parameters, recording the weights of the (finite-dimensional)
weight spaces of $\mathbb{L}_{i,z}$.

\item\label{FH:3} When $\mathbb{V}$ is a highest weight representation, it is shown in
\cite[Thm 5.9]{frenkel-hernandez} that, up to a normalization
factor depending on $\mathbb{V}$, the entries of $T^{\FH}_i(z,\ul{t})$
are polynomials in $z$ (with coefficients being power series
in the variables $(t_j)_{j\in\bfI}$). The eigenvalues of this
normalized operator are called {\em Baxter polynomials}.

\item\label{FH:4} Let $T^{\FH}_i(z)\in \Pseries{\End(\mathbb{V})}{z}$ denote the limit of $T^{\FH}_i(z,\ul{t})$
as $t_j\to 0$ for all $j\in\bfI$ (see \cite[Prop.~5.5]{frenkel-hernandez}). It follows from \eqref{FH:3} (see \cite[Thm.~5.17]{frenkel-hernandez}) that, up to renormalization, this is an  $\End(\mathbb{V})$-valued polynomial in $z$, with monic eigenvalues which we refer to as \textit{specialized Baxter polynomials}.

\item\label{FH:5} As pointed out in
\cite[\S 7.2]{frenkel-hernandez}, $T^{\FH}_i(z)$ can be
obtained directly by making the (formal) substitution
$\ds \phi_j^+(z) \mapsto (1-z)^{-\delta_{ij}}$
on the first tensor factor of the diagonal part $\mathscr{R}^0$ of $\mathscr{R}$.
 Here, $\{\phi^+_j(z)\}_{j\in\bfI}$
are the commuting Cartan currents of the quantum loop algebra; see \cite[\S 2.1]{frenkel-hernandez}.
\end{enumerate}
In Theorem \ref{thm:Baxter} below, we prove that the zeros of the $\Yhg$-variants of the specialized Baxter polynomials from \eqref{FH:4} encode the $i$-th set of poles of the underlying representation. We cannot, however, follow the procedure \eqref{FH:1}--\eqref{FH:4} outlined above to define the $\Yhg$-analogues $T_i(u)$ of the operators $T^{\FH}_i(z)$. Indeed, $\Yhg$ contains no analogue of $U_q(\wh{\mathfrak{b}}_+)$, and it is unclear how to evaluate its universal $R$-matrix on any counterpart of the tensor product $\mathbb{L}_{i,z}\otimes \mathbb{V}$ considered above\footnote{See, however, the recent paper \cite{hernandez-zhang} which gives an alternative construction of $T_i(u)$ using the representation theory of shifted Yangians. This is expanded upon in Section \ref{ssec:shifted}.}. To circumvent this fact, we shall draw inspiration from the observation \eqref{FH:5} and define $T_i(u)$ as one of the two canonical solutions to an abelian difference equation encoded by the diagonal part of the universal $R$-matrix of $\Yhg$; see \eqref{eq:Ti} below.

\subsection{Quantum Cartan matrix}\label{ssec:qC-any}
Let $\mathbf{B}=(d_ia_{ij})_{i,j\in\bfI}$ be the symmetrized
Cartan matrix of $\g$, and let
$\mathbf{B}(q) = ([d_ia_{ij}]_q)_{i,j\in\bfI}$. Here,
we use the standard convention of $q$--numbers:
\[
[n]_q = \frac{q^n-q^{-n}}{q-q^{-1}}
\]
for any $n\in \Z$. By 
\cite[Thm. A.1]{sachin-valerio-III}, for instance, the inverse of $\mathbf{B}(q)$ has the form
\[
\mathbf{B}(q)^{-1} = \frac{1}{[2\kappa]_q}\mathbf{C}(q),
\]
where the entries $c_{ij}(q)=\sum_{r\in\Z}
c_{ij}^{(r)}q^r$ of $\mathbf{C}(q) = (c_{ij}(q))_{i,j\in\bfI}$
belong to $\N[q,q^{-1}]$. For future reference, 
we note that the relation $\mathbf{B}(q)\mathbf{C}(q)
=[2\kappa]_q$ can be expressed equivalently as the collection of equalities 
\begin{equation}\label{eq:BC=l}
\sum_{j\in\bfI} c_{ij}(q)(q^{d_ja_{jk}}-q^{-d_ja_{jk}})
= \delta_{ik}(q^{2\kappa} - q^{-2\kappa}) \quad \forall \quad i,k\in \bfI. 
\end{equation}

\subsection{Transfer operators}\label{ssec:Ti}
For each $i\in\bfI$, define
\begin{equation}\label{eq:Ai}
A_i(u) = \prod_{\begin{subarray}{c} j\in\bfI\\ r\in\Z\end{subarray}}
\xi_j\lp u + \kappa\hbar + \frac{\hbar}{2} r\rp^{-c_{ij}^{(r)}}\ .
\end{equation}
Note that $A_i(u)\in 1+u^{-1}\Pseries{Y^0_{\hbar}(\g)}{u^{-1}}$, where
$Y^0_{\hbar}(\g)$ is the commutative subalgebra of $\Yhg$ generated
by $\{\xi_{j,r}\}_{j\in\bfI,r\in\N}$. Let $A_{i,0}$ be the coefficient
of $u^{-1}$ in $A_i(u)$. Using the fact that the coefficient of $u^{-1}$
in $\xi_j(u)$ is $\hbar\xi_{j,0}$, and shift in $u$ does not alter this
coefficient, we get $
A_{i,0} = -\hbar\sum_{j\in\bfI} c_{ij}(1)\xi_{j,0}$. Therefore, for $k\in\bfI$:
\[
\alpha_k(A_{i,0}) = -\hbar\sum_{j\in\bfI} c_{ij}(1)d_ja_{jk} = 
-2\kappa\hbar\delta_{ik},
\]
where, in the last equality, we have used the $q\to 1$ limit of \eqref{eq:BC=l}.
Thus,
\begin{equation}\label{eq:ResAi}
A_{i,0} = -2\kappa\hbar\varpi_i^{\vee}\ ,
\end{equation}
where $\varpi_i^{\vee}\in\h$ denotes the $i$-th fundamental coweight,
uniquely determined by $\alpha_k(\varpi_i^{\vee})=\delta_{ik}$, for every $k\in\bfI$.
For later purposes, we observe that the logarithmic derivative of $A_i(u)$
can be written more compactly as
\begin{equation}\label{eq:Ai-logder}
A_i(u)^{-1}A_i'(u) = - \tau^{2\kappa} \sum_{j\in\bfI} c_{ij}(\tau)\cdot
(\xi_j(u)^{-1}\xi_j'(u)),
\end{equation}
where $\tau$ is the shift $f(u)\mapsto f(u+\hbar/2)$ acting on
functions of $u$, or formal series in $u^{-1}$.

Now let $V$ be a fixed finite-dimensional representation of $\Yhg$. The
operators $A_i(u)$, evaluated on $V$, are rational 
$\End(V)$-valued functions of $u\in\C$, with
$A_i(\infty)=\Id_V$. For each $i\in \bfI$, consider the associated abelian difference equation 
\begin{equation}\label{eq:Ti}
T_i(u+2\kappa\hbar) = A_i(u)T_i(u).
\end{equation}
 This difference equation admits two \textit{canonical fundamental solutions} $T_i^\pm(u)$,
expressible in terms of Euler's gamma function (see, for example,
\cite[\S 4]{sachin-valerio-2}). These are $\End(V)$-valued meromorphic functions of $u$, uniquely
determined by the following two conditions:
\begin{itemize}
\item $T_i^\pm(u)$ is holomorphic and invertible for $\pm\Re(u/\hbar)\gg 0$,
\item $T_i^\pm(u)$ is asymptotic to $(\pm u)^{-\varpi_i^{\vee}}$ as $|u|\to \infty$  in the zone $\pm\Re(u/\hbar)\gg 0$,
\end{itemize}
where, by \eqref{eq:ResAi}, $A_i(u)=1-2\kappa\hbar\varpi_i^{\vee}u^{-1}+\ldots$ is the Taylor series
of $A_i(u)$ near $u=\infty$.
We refer the reader to \cite[Thm.~6]{sachin-valerio-2} for complete details. 
For definiteness, we set $T_i(u)=T_i^+(u)$ and call it the $i$-th
{\em abelianized transfer operator} associated to $V$. 
\begin{ex}
To illustrate the analytic behaviour of $T_i(u)$ as specfied above, consider the case where $V=L_{\varpi_j}$. Then it follows from Proposition \ref{pr:CPthm} and \eqref{eq:Ai} that the eigenvalues of $T_i(u)$ on the highest and lowest weight spaces are
\begin{equation*}
\prod_{r\in\Z} \lp
\frac{\Gamma_\kappa\lp u+\hbar(\kappa+d_j)+\frac{r\hbar}{2} \rp}
{\Gamma_{\kappa}\lp u+\hbar\kappa+\frac{r\hbar}{2} \rp}
\rp^{-c_{ij}^{(r)}} \quad \text{and }\quad \prod_{r\in\Z} \lp
\frac{\Gamma_\kappa\lp u+\hbar d_j+\frac{r\hbar}{2} \rp}
{\Gamma_{\kappa}\lp u+\frac{r\hbar}{2} \rp}
\rp^{c_{ij^*}^{(r)}},
\end{equation*}
respectively, where $\Gamma_\kappa(z) := \Gamma\lp\frac{z}{2\kappa\hbar}\rp$ and $\Gamma(z)$ is Euler's gamma function (see, \eg \cite[Ch. 12]{whittaker-watson}).
\end{ex}
\begin{rem}\label{rem:T-motive}
As indicated in Section \ref{ssec:FH-summary}, the definition of $T_i(u)$ given above is motivated by the observation \eqref{FH:5} therein. Indeed, one can recover $T_i(u)$ heuristically by applying the formal substitution $\xi_j(u)\mapsto u^{\delta_{ij}}$
on the first tensor factor of the (meromorphic) abelian part $\mathcal{R}^0(u)$ of the universal $R$--matrix of $\Yhg$. In more detail, $\mathcal{R}^0(u)$ is defined in \cite[\S 5.8]{sachin-valerio-III} as one of the two fundamental solutions of 
\begin{equation*}
\mathcal{R}^0(u+2\kappa\hbar)
= \mathcal{A}(u)\mathcal{R}^0(u),
\end{equation*}
where $\mathcal{A}(u)$ is defined in \cite[\S 5.5]{sachin-valerio-III}. Substituting $\xi_j(u)\mapsto u^{\delta_{ij}}$
in the first tensor factor of $\mathcal{A}(u)$, one obtains
$A_i(u)$ from \eqref{eq:Ai}. We defer a more precise discussion of this point to Section \ref{ssec:R-matrix}; see \eqref{R-A}, Lemma \ref{L:A->Ai} and Remark \ref{rem:T-motive2} therein. 
\end{rem}
\subsection{From transfer operators to poles}\label{ssec:thm-Baxter}
The following theorem provides the main result of this section. For the quantum loop algebra $U_q(L\g)$, the counterparts of \eqref{Baxter:1} and \eqref{Baxter:2} are consequences of the stronger results obtained in Theorem 5.9 and Corollary 5.10 of \cite{frenkel-hernandez}; see also Propositions 5.5 and 5.8 therein. 
\begin{thm}\label{thm:Baxter}
Assume that $V$ is a highest-weight module. Let $\lambda\in\h^*$ be
the $\g$-weight of its highest-weight space and denote by $f^+_i(u)$ the eigenvalue of $T_i(u)$ on $V_\lambda$. Then the normalized transfer operator
$
\Tnorm_i(u)=f^+_i(u)^{-1}T_i(u)
$
has the following properties:

\begin{enumerate}[font=\upshape]
\item\label{Baxter:1}  It is an element of $\End(V)[u]$  with monic eigenvalues.
\item\label{Baxter:2} If $\mu\in\h^*$ is a $\g$-weight of $V$, then the restriction $\Tnorm_i(u)|_{V_\mu}$ is an $\End(V_\mu)$-valued monic polynomial of degree 
$(\lambda-\mu)(\varpi_i^{\vee})$. 
\item\label{Baxter:3} Let $\mathcal{Z}_i(V)$ denote the set of zeroes of the eigenvalues
of $\Tnorm_i(u)$. Then
\[
\mathcal{Z}_i(V) = \sigma_i(V).
\]
\item\label{Baxter:4} Let $\mathcal{Q}^\g_{i,V}(u)\in\C[u]$ be the 
monic polynomial which is the eigenvalue
of $\Tnorm_i(u)$ on the lowest-weight space $V_{w_0(\lambda)}$.
Then
\[
\zeroes(\mathcal{Q}^\g_{i,V}(u)) = \sigma_i(V).
\]
\end{enumerate}
\end{thm}

\begin{pf}
The proof of this theorem is based on the commutation relations between
$T_i(u)$ and the lowering operators of the Yangian. More precisely,
since $x_k^\pm(u)$ operate on $V$ as rational functions of $u$ with $x_k^\pm(\infty)=0$,
they admit partial fraction decompositions of the form
\begin{equation}\label{eq:partial-fractions}
x_k^\pm(u) = \sum_{\begin{subarray}{c} b\in \sigma_k(V) \\
n\in\N\end{subarray}} \frac{X^\pm_{k;b,n}}{(u-b)^{n+1}}\ .
\end{equation}
We obtain commutation relations between $\xi_j(u)$ and $X^-_{k;b,n}$
in Lemma \ref{lem:comm-xi-pf},
and use them to obtain, in Proposition \ref{pr:Ti-matrix}, that
\begin{equation}\label{eq:Ti-xk}
\Ad(T_i(u))\cdot X^-_{k;b,n} = \left\{
\begin{array}{ll}
X^-_{k;b,n} & \text{if } i\neq k,\\
(u-b)X^-_{k;b,n}-X^-_{k;b,n+1} & \text{if } i=k.
\end{array}
\right.
\end{equation}
Though not needed in the proof of this theorem, the analogous
commutation relations with the raising operators are given
as follows. The proof is identical to that for the lowering
operators, and hence is omitted.
\begin{equation}\label{eq:Ti-xk+}
\Ad(T_i(u)^{-1})\cdot X^+_{k;b,n} = \left\{
\begin{array}{ll}
X^+_{k;b,n} & \text{if } i\neq k,\\
(u-b)X^+_{k;b,n}-X^+_{k;b,n+1} & \text{if } i=k.
\end{array}
\right.
\end{equation}
Since $V$ is highest weight, it is spanned by vectors obtained from $V_\lambda$
by successive applications of the lowering operators $X^-_{k;b,n}$. The
assertions
\eqref{Baxter:1} and \eqref{Baxter:2} follow from this fact and \eqref{eq:Ti-xk}. 
We also obtain $\sigma_i(V)\subset\mathcal{Z}_i(V)$ as
a consequence of \eqref{eq:Ti-xk}.
To prove \eqref{Baxter:3} and \eqref{Baxter:4}, it is enough
to establish the sequence of inclusions 
\begin{equation*}
\mathcal{Z}_i(V)\subset\sigma_i(V)
\subset \zeroes(\mathcal{Q}_{i,V}^\g(u)).
\end{equation*}

We begin by showing that $\mathcal{Z}_i(V)\subset\sigma_i(V)$. Let $\operatorname{ht}:Q_+\to \N$ denote the height function on the positive cone $Q_+$ in the root lattice of $\g$. 
Assume now that $b\in\mathcal{Z}_i(V)$. By definition, there is a weight $\mu$ of $V$ such that 
$\Tnorm_i(V)$ acting on $V_\mu$ has an eigenvalue divisible
by $u-b$. Let us fix $\mu\in \h^\ast$ with this property for which $\operatorname{ht}(\lambda-\mu)$ is minimal. 
 Note that $\mu<\lambda$ because $\Tnorm_i(u)=1$
on $V_{\lambda}$. Since $V$ is highest-weight, we have
\[
V_\mu = \operatorname{Span}\{X^-_{k;c,n}(V_{\mu+\alpha_k}) : 
k\in\bfI, c\in\sigma_k(V), n\in\N\}.
\]
By our assumption on $\mu$, the
eigenvalues of $\Tnorm_i(u)$ on $V_{\mu+\alpha_k}$ do not vanish at $b$ for any $k\in \bfI$.
Let $\gamma(u)\in\C[u]$ be an eigenvalue of $\Tnorm_i(u)$
on $V_{\mu+\alpha_k}$, and let $V_{\mu+\alpha_k}[\gamma(u)]$
denote the associated generalized eigenspace.

\noindent {\bf Claim.} For every $c\in\sigma_k(V)$, we have
\[
\sum_{n\in\N} X^-_{k;c,n}(V_{\mu+\alpha_k}[\gamma(u)]) \subset
V_\mu[(u-c)^{\delta_{ik}}\gamma(u)].
\]
Hence, for $\Tnorm_i(u)$ to have an eigenvalue on $V_\mu$ which
is divisible by $u-b$, some operator $X^-_{i;b,n}$ must be non-zero,
proving that $b\in\sigma_i(V)$.

\noindent {\em Proof of the claim.} For $k\neq i$, the assertion
follows directly from \eqref{eq:Ti-xk}. Assuming $k=i$, let
us fix a basis $\{v_1,\ldots,v_p\}$ of $V_{\mu+\alpha_i}[\gamma(u)]$
in which $\Tnorm_i(u)$ is lower triangular. Ordering the
set of non-zero vectors in $\{X^-_{i;c,n}(v_r)\}_{1\leq r\leq p, n\in\N}$ 
lexicographically
(\ie $(n_1,r_1)>(n_2,r_2) \iff r_1>r_2,$ or $r_1=r_2$ and $n_1>n_2$),
we conclude from \eqref{eq:Ti-xk} that $\Tnorm_i(u)-(u-c)$
is nilpotent on $\sum_{n\geq 0}X^-_{i;c,n}(V_{\mu+\alpha_i}[\gamma(u)])$.
This completes the proof of the claim, and thus that 
$\mathcal{Z}_i(V)\subset\sigma_i(V)$.

Now we will prove that $\sigma_i(V)\subset\zeroes(\mathcal{Q}^\g_{i,V}(u))$.
Let $b\in\sigma_i(V)\subset\mathcal{Z}_i(V)$. Choose
a weight space $V_\mu$ with minimal $\operatorname{ht}(\mu-w_0(\lambda))$,
such that there exists a non--zero eigenvector $v\in V_\mu$ of
$\Tnorm_i(u)$, with eigenvalue $\gamma(u)$ divisible by $u-b$.

If $\mu=w_0(\lambda)$, then $\gamma(u)=
\mathcal{Q}^\g_{i,V}(u)$ and we are done. Otherwise, there
exists some $j\in\bfI$ such that $x_j^-(u)v\neq 0$. Let $c\in\sigma_j(V)$
be such that $X^-_{j;c,n}v\neq 0$ for some $n$. Taking $n$ to
be largest such, we use \eqref{eq:Ti-xk} to conclude
that $X^-_{j;c,n}(v)\in V_{\mu-\alpha_j}$ is an eigenvector
of $\Tnorm_i(u)$, with eigenvalue
$(u-c)^{\delta_{ij}}\gamma(u)$ divisible by $u-b$.
This contradicts the minimality of $\operatorname{ht}(\mu-w_0(\lambda))$
and finishes the proof of the theorem.
\end{pf}

\subsection{Partial fractions}\label{ssec:partial-fractions}
Recall that, for a fixed $k\in \bfI$, the family of operators $\{X_{k;b,n}^-\}_{b\in \sigma_k(V),n\in \N}$ on $V$ is defined by the partial fraction decomposition \eqref{eq:partial-fractions} of the $\End(V)$-valued rational function $x_k^-(u)$. 
In what follows, we use the divided power notation $\partial_v^{(p)}=\frac{1}{p!}\partial_v^p$, where $\partial_v$ is the partial derivative operator with respect to $v$.
\begin{lem}\label{lem:comm-xi-pf}
Let $j,k\in\bfI$ and set $a=d_ja_{jk}\hbar/2$.
Then, for each
$b\in\sigma_k(V)$ and $n\in\N$, we have:
\begin{align}
\Ad(\xi_j(u))\cdot X_{k;b,n}^-
&= \sum_{p\geq 0} \left.\partial^{(p)}_v\left(\frac{u-v-a}{u-v+a}\right)
\right|_{v=b} X_{k;b,n+p}^-, \label{eq:xi-pf1} \\
\Ad(\xi_j(u)^{-1})\cdot X_{k;b,n}^-
&= \sum_{p\geq 0} \left.\partial^{(p)}_v\left(\frac{u-v+a}{u-v-a}\right)
\right|_{v=b} X_{k;b,n+p}^-. \label{eq:xi-pf2}
\end{align}
\end{lem}
\begin{pf}
We use the relation \ref{sY23} from \S \ref{ssec:reln-current}, which can be written as
\[
\Ad(\xi_j(u))\cdot x_k^-(v) = 
\frac{u-v-a}{u-v+a}x_k^-(v) + \frac{2a}{u-v+a}x_k^-(u+a)\ .
\]
For $b\in\sigma_k(V)$ and $n\in\N$, multiply both sides of
this equation by $(v-b)^n$ and integrate over a small counterclockwise
contour centered around $b$. This immediately gives
\eqref{eq:xi-pf1} using Cauchy's integral formula.

For the second relation, we first set $v=u-a$ in \ref{sY23} above, to get
\[
\Ad(\xi_j(u))\cdot x_k^-(u-a) = x_k^-(u+a).
\]
Substituting this back into \ref{sY23}, we obtain
\[
\Ad(\xi_j(u)^{-1})\cdot x_k^-(v) = 
\frac{u-v+a}{u-v-a}x_k^-(v) - \frac{2a}{u-v-a}x_k^-(u-a)\ ,
\]
from which we can deduce \eqref{eq:xi-pf2} using the same argument
as before.
\end{pf}

\subsection{Main commutation relation}\label{ssec:linear-independence}
Let us fix $k\in\bfI$, a weight space $V_{\mu}$ of $V$,
and an element $b\in\sigma_k(V)$.
Let $N_b$ be the largest non-negative integer such
that $X_{k;b,N_b}^-\neq 0$ as an operator from $V_\mu$
to $V_{\mu-\alpha_k}$.
As an easy application of the following result from linear algebra,
Lemma \ref{lem:comm-xi-pf}
implies that the subset $\{X_{k;b,n}^-\}_{n=0}^{N_b}
\subset \Hom_{\C}(V_\mu,V_{\mu-\alpha_k})$ is linearly independent.
\begin{lem}
Let $W$ be a finite-dimensional vector space over $\C$.
Let $w_0,\ldots,w_N\in W$ with $w_N\neq 0$, and assume that there
exists $X\in \End(W)$ and $x_0,\ldots,x_N\in\nC$
such that
\[
X(w_n) = \sum_{p=0}^{N-n} x_p w_{n+p}\quad \forall\quad 0\leq n\leq N.
\]
Then $\{w_0,\ldots,w_N\}$ is a linearly independent set.
\end{lem}
Now let $\mathfrak{X}\subset\Hom_{\C}(V_\mu,V_{\mu-\alpha_k})$ be the span
of $\{X_{k;b,n}^-\}_{n=0}^{N_b}$ and consider
$\Ad(T_i(u))$ acting on $\mathfrak{X}$.
\begin{prop}\label{pr:Ti-matrix}
The action of $\Ad(T_i(u))$ on $\mathfrak{X}$ is determined by
\[
\Ad(T_i(u))\cdot X_{k;b,n}^- = (u-b)^{\delta_{ik}}X_{k;b,n}^-
- \delta_{ik} X_{k;b,n+1}^-\qquad \forall \quad n\geq 0.
\]
%
Equivalently, $\Ad(T_i(u))$ is the operator on $\mathfrak{X}$ given explicitly by
\begin{equation*}
\Ad(T_i(u))=
\begin{cases}
\Id_{\mathfrak{X}} \; &\text{ if } k\neq i,\\
\mathcal{J}(u) \; &\text{ if } k=i,
\end{cases}
\end{equation*}
where $\mathcal{J}(u)$ is the lower triangular matrix, in the basis $\{X_{k;b,n}^-\}_{n=0}^{N_b}$, defined by
\[
\mathcal{J}(u) = \begin{bmatrix}
\begin{array}{ccccc}
u-b & 0      & 0      & \cdots & 0 \\
-1  & \ddots & 0      & \cdots & 0 \\
0   & \ddots & \ddots & \ddots & \vdots \\
\vdots & \ddots & \ddots & \ddots & 0 \\
0 & \cdots & 0 & -1 & u-b 
\end{array}
\end{bmatrix}
\]
\end{prop}

\begin{pf}
Since $\Ad(u^{-\varpi_i^{\vee}})$ acts as $u^{\delta_{ik}}$ on $\mathfrak{X}$,
by definition (see \eqref{eq:Ti})
$\Ad(T_i(u))$ acting on $\mathfrak{X}$ is uniquely determined by:
\begin{itemize}
\item $u^{-\delta_{ik}} \Ad(T_i(u)) \to \Id_\mathfrak{X}$ as $u\to\infty$ in $\Re(u/\hbar)\gg 0$.
\item $\Ad(T_i(u+2\kappa\hbar)) = \Ad(A_i(u))\Ad(T_i(u))$.
\end{itemize}
The answer given in the statement of the proposition clearly satisfies the
first condition. Thus, it is enough to show that
\[
\Ad(A_i(u)) = \left\{\begin{array}{ll} \Id_\mathfrak{X} &\text{if } i\neq k, \\ 
\mathcal{J}(u+2\kappa\hbar)\mathcal{J}(u)^{-1} & \text{if } i=k. \end{array}\right.
\]
As both sides are rational functions of $u$, valued in a commutative
subalgebra of $\End(\mathfrak{X})$, and equal to $\Id_\mathfrak{X}$ at $u=\infty$, it
is sufficient to prove the equality of their respective logarithmic derivatives.
That is, we have to show that
\begin{equation}\label{eq:pf-maincomm}
\ad\!\lp A_i(u)^{-1}A_i'(u)\rp = \begin{cases} 0 &\text{if } i\neq k, \\ 
\mathcal{J}(u+2\kappa\hbar)^{-1}
- \mathcal{J}(u)^{-1} & \text{if } i=k,
\end{cases}
\end{equation}
where we have used that $\mathcal{J}'(u)=\Id_\mathfrak{X}$. Writing $\mathcal{J}(u) = (u-b)\left( \Id_\mathfrak{X} 
- \tfrac{1}{u-b}F\right)$, where
$F$ is the nilpotent operator $X_{i;b,n}^-\mapsto X_{i;b,n+1}^-$, we see that
$\mathcal{J}(u)^{-1}$ is equal to $\frac{1}{u-b} \sum_{m\geq 0} \frac{1}{(u-b)^m} F^m$. Equivalently, $\mathcal{J}(u)^{-1}$ is determined by
\begin{equation}\label{eq:pf-maincomm-2}
\mathcal{J}(u)^{-1}\cdot X_{i;b,n} = \sum_{m\geq 0}
\frac{1}{(u-b)^{m+1}} X_{i;b,n+m}.
\end{equation}

It remains to determine the left-hand side of \eqref{eq:pf-maincomm}.
We begin by computing the commutator of the logarithmic derivative
of $\xi_j(u)$ with $X_{k;b,n}^-$. Applying $\partial_u$ to the relation \eqref{eq:xi-pf1}
from Lemma \ref{lem:comm-xi-pf}, we get
\[
\xi_j'(u)X_{k;b,n}^-\xi_j(u)^{-1} - \xi_j(u)X_{k;b,n}^-\xi_j(u)^{-2}\xi_j'(u)
= \sum_{p\geq 0} \left.\partial_v^{(p)}\left(\frac{2a}{(u-v+a)^2}\right)
\right|_{v=b} X_{k;b,n+p}^-\ ,
\]
where we recall that $a=\hbar d_ja_{jk}/2$.
Applying $\Ad(\xi_j(u)^{-1})$ to both sides of this equation,
and using \eqref{eq:xi-pf2}, we obtain
\begin{align*}
\ad&(\xi_j(u)^{-1}\xi_j'(u))\cdot X_{k;b,n}^-
\\
&=\sum_{p_1\geq 0} \left.\partial_v^{(p_1)}\left(\frac{2a}{(u-v+a)^2}\right)
\right|_{v=b}
\cdot\left(\sum_{p_2\geq 0} 
\left.\partial_v^{(p_2)}\left(\frac{u-v+a}{u-v-a}\right)
\right|_{v=b} X_{k;b,n+p_1+p_2}^-\right)\\
&= \sum_{m\geq 0} \left(
\sum_{p_1+p_2=m} 
\left.\partial_v^{(p_1)}\left(\frac{2a}{(u-v+a)^2}\right)
\partial_v^{(p_2)}\left(\frac{u-v+a}{u-v-a}\right)
\right|_{v=b}\right) X_{k;b,n+m}^-\\
&= \sum_{m\geq 0} \left.\partial_v^{(m)}\left(\frac{2a}{(u-v+a)^2}\cdot
\frac{u-v+a}{u-v-a}\right)
\right|_{v=b} X_{k;b,n+m}^-\\
&= \sum_{m\geq 0} \left.\partial_v^{(m)}\left(
\frac{1}{u-v-a}-\frac{1}{u-v+a}\right)
\right|_{v=b} X_{k;b,n+m}^-,
\end{align*}
where, in the third equality, we have used Leibniz' rule.
Using the shift $\tau:f(u)\mapsto f(u+\hbar/2)$, the result of this calculation can
be written as
\[
\ad(\xi_j(u)^{-1}\xi_j'(u))\cdot X_{k;b,n}^-
=
\sum_{m\geq 0} (\tau^{-d_ja_{jk}}-\tau^{d_ja_{jk}})\lp
\frac{1}{(u-b)^{m+1}}\rp
X_{k;b,n+m}^-.
\]
Using the definition of $A_i(u)$ from \eqref{eq:Ai-logder}, we get:
\begin{align*}
\ad(A_i(u)^{-1}&A_i'(u))\cdot X_{k;b,n}^- \\
&= \sum_{m\geq 0} X_{k;b,n+m}^-\left(
\tau^{2\kappa}\sum_{j\in\bfI} c_{ij}(\tau)(\tau^{d_ja_{jk}}-\tau^{-d_ja_{jk}})
\right)\cdot \lp\frac{1}{(u-b)^{m+1}}\rp\\
&=\delta_{ik} \sum_{m\geq 0}
\lp
\frac{1}{(u-b+2\kappa\hbar)^{m+1}} - \frac{1}{(u-b)^{m+1}}
\rp X_{k;b,n+m}^-\ ,
\end{align*}
where in the last equation, we have used \eqref{eq:BC=l}, with $q=\tau$.
The proof follows by comparing this answer with \eqref{eq:pf-maincomm-2}.
\end{pf}

\subsection{A note on shifted Yangians}\label{ssec:shifted}
Let us now restrict to the case where $V$ is a finite-dimensional irreducible $\Yhg$-module with highest weight space $V_\lambda$. The recent work \cite{hernandez-zhang} of Hernandez and Zhang, which appeared after a previous version of this article, gives an interpretation of the normalized transfer operator $\bar{T}_i(u)\in \End(V)[u]$ in terms of the representation theory of shifted Yangians. This is achieved by first constructing, in Theorem 5.2 (ii) therein,  a family of $Y_{\varpi_i^\vee}(\g)$-module intertwiners 
\begin{equation*}
R_{L_{i,a}^+,V}:L_{i,a}^+\otimes V\to V\otimes L_{i,a}^+ \quad \forall\; a\in \C,
\end{equation*}
where $Y_{\varpi_i^\vee}(\g)$ is the shifted Yangian associated to the fundamental coweight $\varpi_i^\vee$ and $L_{i,a}^+$ is a one-dimensional $Y_{\varpi_i^\vee}(\g)$-module called a \textit{positive prefundamental representation} (\textit{cf}. Section \ref{ssec:FH-summary}), defined in \cite[Ex.~3.5]{hernandez-zhang}. Taking the trace of $(1\,2)\circ R_{L_{i,a}^+,V}$ over $L_{i,a}^+$ gives an endomorphism of $V$, which is shown in Proposition 5.7 of \cite{hernandez-zhang} to be the evaluation at $u=a$ of a polynomial operator $R_i^V(u)\in \End(V)[u]$.  To relate $R_i^V(u)$ to $\bar{T}_i(u)$, let us introduce the diagonal matrix $C_i$ by 
\begin{equation*}
C_i=\sum_{\mu:V_\mu\neq 0} (-1)^{(\lambda-\mu)(\varpi_i^\vee)} \id_{\mu} \in \End(V),
\end{equation*}
where $\id_{\mu}:V\to V_\mu$ is the projection associated to the $\g$-weight space decomposition of $V$. We then have the following identification. 
\begin{cor}\label{C:Ti-vs-Ri}
Let $V$ be a finite-dimensional irreducible $\Yhg$-module. Then, for each $i\in \bfI$, one has $\bar{T}_i(u)=C_i\cdot R_i^V(u)$. 
\end{cor}
\begin{pf}
The relations \eqref{eq:partial-fractions} and \eqref{eq:Ti-xk}, together with (5.37) of \cite{hernandez-zhang} (see also the proof of \cite[Prop.~5.7]{hernandez-zhang}),  imply that $T_i(u)$ and $R_i^V(u)$ satisfy the relations 
\begin{equation}\label{Ti-xj}
\begin{aligned}
\bar{T}_i(u)\cdot x_{j,n}^- &= (u^{\delta_{ij}}x_{j,n}^--\delta_{ij}x_{i,n+1}^-)\bar{T}_i(u) \\
R_i^V(u)\cdot x_{j,n}^- &= (\delta_{ij}x_{i,n+1}^-+(-u)^{\delta_{ij}}x_{j,n}^-)R_i^V(u)
\end{aligned}
\end{equation}
for all $j\in \bfI$. The corollary follows from this observation, coupled with the fact that $R_i^V(u)$ and $\bar{T}_i(u)$ restrict to the identity on the highest weight space $V_\lambda$. Indeed, this can be seen by the following argument, given in the proof of \cite[Prop.~5.7]{hernandez-zhang}: If $\mu\in \h^\ast$ is a $\g$-weight of $V$, then $V_\mu$ is spanned by vectors of the form $\Omega_{j_1,\ldots,j_p}^{n_1,\ldots,n_p}=x_{j_1,n_1}^- \cdots x_{j_p,n_p}^- \Omega^+$, where each index $j\in \bfI$ appears $(\lambda-\mu)(\varpi_j^\vee)$ times as a subscript. The relations of \eqref{Ti-xj} then give
\begin{equation*}
\bar{T}_i(u)\Omega_{j_1,\ldots,j_p}^{n_1,\ldots,n_p}=\prod_{\ell=1}^p(u^{\delta_{i,j_\ell}}x_{j_\ell,n_\ell}^--\delta_{i,j_\ell}x_{i,n_\ell+1}^-)\Omega^+=(-1)^{(\lambda-\mu)(\varpi_i^\vee)}R_i^V(u)\Omega_{j_1,\ldots,j_p}^{n_1,\ldots,n_p}.
\end{equation*}
Hence, the equality $\bar{T}_i(u)=C_i\cdot R_i^V(u)$ holds on each $\g$-weight space $V_\mu$, and thus on all of $V$. \qedhere
\end{pf}
\begin{rem}\label{R:GKLO}
By Proposition 5.8 of \cite{hernandez-zhang}, $R_i^V(u)$ satisfies the additive difference equation
\begin{equation*}
R_i^V(u+\hbar d_i)=\bar{\mathscr{A}_i}(u)R_i^V(u)
\end{equation*} 
where $\bar{\mathscr{A}}_i(u)\in \End(V)(u)$ is the evaluation of the so-called GKLO series \cite{GKLO} $\mathscr{A}_i(u)$ for $\Yhg$ on $V$, normalized so as to act as the identity operator on $V_\lambda$;  see \cite[Lemma 2.1]{GKLO} and \cite[(2.20)]{hernandez-zhang} with $\mathbf{r}_i(u)=1$. By \eqref{eq:Ti} and the above corollary, we thus have 
\begin{equation*}
\bar{A}_i(u)=\bar{\mathscr{A}}_i(u)\bar{\mathscr{A}}_i(u+\hbar d_i)\cdots \bar{\mathscr{A}}_i(u+ (\tfrac{2\kappa}{d_i}-1)\hbar d_i),
\end{equation*}
where $\bar{A}_i(u)=g_i^+(u)^{-1}A_i(u)$ with $A_i(u)$ as in \eqref{eq:Ai} and $g_i^+(u)$ its eigenvalue on $V_\lambda$.
As explained to us by A.~Tsymbaliuk, this relation holds even with $\bar{A}_i(u)$ and $\bar{\mathscr{A}}_i(u)$ (which are $\End(V)$-valued rational functions of $u$) replaced by the elements $A_i(u)$ and $\mathscr{A}_i(u)$ of $Y_\hbar^0(\g)[\![u^{-1}]\!]$, respectively.
\end{rem}
\section{Poles of simple modules}\label{sec:Combinatorial}

Part \eqref{Baxter:4} of Theorem \ref{thm:Baxter} implies that the $i$-th set of poles $\sigma_i(V)$ of an arbitrary finite-dimensional irreducible representation $V$ of $\Yhg$ can be obtained explicitly by computing the specialized Baxter polynomial $\mathcal{Q}_{i,V}^\g(u)$ associated to the lowest weight space of $V$. In this section, we carry out this computation by directly solving the difference equation defining these polynomials. The end result is Theorem \ref{thm:Comb}, which gives a combinatorial description of $\mathcal{Q}_{i,V}^\g(u)$, and thus $\spec_i(V)$,  in terms of the Drinfeld polynomials of $V$ and the underlying Cartan data of $\g$.

After proving Theorem \ref{thm:Comb}, we spell out a few corollaries and give a detailed example, with emphasis on the the fundamental representations $V=L_{\varpi_j}$; see Sections \ref{ssec:vij}--\ref{ssec:ysln}. We conclude in Section \ref{ssec:Q-props} by translating formulas from \cite{frenkel-hernandez} for the $U_q(L\g)$-analogues of the eigenvalues of $\bar{T}_i(u)$ to the setting of $\Yhg$.

\subsection{Reduction to fundamental modules}

Henceforth, we fix $V\cong L(\ul{P})$ to be a finite-dimensional irreducible $\Yhg$-module. Let $\lambda=\sum_{i\in \bfI}\deg(P_i)\varpi_i\in \h^\ast$, so that $V_\lambda$ is the highest weight space of $V$ (see Section \ref{ssec:dp}). 

Let $\bar{A}_i(u)=g_i^+(u)^{-1}A_i(u)$,  where $g_i^+(u)$ is the eigenvalue of $A_i(u)$ on a $V_\lambda$, as in Remark \ref{R:GKLO}. This operator admits a unique eigenvalue on each generalized weight space $V[\ul{\mu}]$ of $V$, given explicitly by
\begin{equation*}
\upnu_{i,V[\ul{\mu}]}^\g(u)=\prod_{\substack{j\in\bfI\\r\in \Z}}\bar{\mu}_j\!\left(u+\kappa\hbar+\frac{\hbar r}{2}\right)^{-c_{ij}^{(r)}},
\end{equation*}
where $\ul{\mu}=(\mu_i(u))_{i\in \bfI}$ and $\bar{\mu}_j(u)=\mu_j(u)\cdot \frac{P_j(u)}{P_j(u+\hbar d_j)}$.
By definition of $T_i(u)$ and Theorem \ref{thm:Baxter}, the eigenvalue $\mathcal{Q}_{i,V[\ul{\mu}]}^\g(u)$ of $\overline{T}_i(u)$ on $V[\underline{\mu}]$ is a monic polynomial in $u$, uniquely characterized by the fact that is solves the difference equation
\begin{equation}\label{def:QiV^mu}
\mathcal{Q}_{i,V[\ul{\mu}]}^\g(u+2\kappa\hbar)= \upnu_{i,V[\ul{\mu}]}^\g(u) \mathcal{Q}_{i,V[\ul{\mu}]}^\g(u).
\end{equation}

Though we shall return to the setting of a general weight $\ul{\mu}$ in Section \ref{ssec:Q-props}, to realize our main goal of computing $\mathcal{Q}_{i,V}^\g(u)$ (and thus $\sigma_i(V)$), we now narrow our focus to the case where $V[\ul{\mu}]$ is the lowest weight space $V_{w_0(\lambda)}$ of $V$, in which case we will write $\upnu_{i,V}^{\g}(u)$ for $\upnu_{i,V[\ul{\mu}]}^{\g}(u)$. The following simple lemma  reduces our task to the case where $V$ is a fundamental representation. 
\begin{lem}\label{L:Q-prod-fac}
Let $V$ be the \fd \irr representation
of $\Yhg$ with Drinfeld polynomials $\ul{P}=(P_j(u))_{j\in\bfI}$, as above. Then 
\begin{equation*}
\mathcal{Q}^\g_{i,V}(u) = \!\prod_{\begin{subarray}{c}
j\in\bfI \\ z\in \zeroes(P_j(u))\end{subarray}}
\hspace{-1em}\mathcal{Q}^\g_{ij}(u-z)^{m_z} \quad \forall\quad i\in \bfI,
\end{equation*}
where $\mathcal{Q}^\g_{ij}(u)=\mathcal{Q}^\g_{i,L_{\varpi_j}}(u)$ and $m_z$ is the multiplicity of $z$ as a root of $P_j(u)$. 
\end{lem}
\begin{pf}
By \eqref{def:QiV^mu}, it is sufficient to establish instead the relation obtained from the claimed identity by replacing $\mathcal{Q}^\g_{i,V}(u)$ by $\upnu^\g_{i,V}(u)$ and each polynomial $\mathcal{Q}_{ij}^\g(u)$ by $\upnu^\g_{i,L_{\varpi_j}}(u)$. 
%
%
This relation follows from the fact that, by Proposition \ref{pr:CPthm}, the components of the lowest weight of $V$ are given by $\mu_j(u)=P_{j^\ast}(u-\hbar\kappa+\hbar d_j)^{-1}P_{j^\ast}(u-\hbar\kappa)$, and hence we have 
\begin{equation}\label{upnu-fac}
\prod_{j\in\bfI}\bar{\mu}_j(u)^{-c_{ij}^{(r)}}=\prod_{j\in \bfI} \mu_j^+(u)^{c_{ij}^{(r)}}\mu_j^+(u-\hbar\kappa)^{c_{ij^\ast}^{(r)}}, \quad  
\end{equation}
where $\mu_j^+(u):=P_j(u)^{-1}P_j(u+\hbar d_j)$ for each $j\in \bfI$. \qedhere 
\end{pf}
\subsection{Poles of simple modules}\label{ssec:poles-compute}
To state and prove the main result of this section, 
we define the inverse, weighted $q$-Cartan matrix $\mathbf{E}(q)=(v_{ij}(q))_{i,j\in \bfI}$ by the formula 
\begin{equation*}
\mathbf{E}(q)=([a_{ij}]_{q^{d_i}})_{i,j\in \bfI}^{-1}.
\end{equation*}
Equivalently, 
$
\mathbf{E}(q)= \mathbf{B}(q)^{-1}\operatorname{Diag}([d_i]_q:i\in\bfI),
$
where $\mathbf{B}(q)=([d_ia_{ij}]_q)_{i,j\in \bfI}$ is the symmetrized $q$-Cartan matrix of $\g$, as in Section \ref{ssec:qC-any}. 
The entries $v_{ij}(q)$ of $\mathbf{E}(q)$ are given in terms of the Laurent polynomials $c_{ij}(q)$ from Section \ref{ssec:qC-any} by the formula 
\[
v_{ij}(q) = \frac{[d_j]_q}{[2\kappa]_q}c_{ij}(q)=\big[2\kappa/d_j\big]_{q^{d_j}}^{-1}c_{ij}(q) \quad \forall \quad i,j\in \bfI
\]
and are viewed as Laurent series in $q$ with coefficients in $\Z$. 
In fact, it is not difficult to deduce from the definition of $\mathbf{E}(q)$  that $v_{ij}(q)$ belongs to $q^{d_i}\Z[\![q]\!]$; see \cite[Lemma 3.3]{fujita-oh}. We will denote the coefficient of $q^r$ in $v_{ij}(q)$ by $v_{ij}^{(r)}$, so that
\[
v_{ij}(q) = \sum_{r\geq d_i} v_{ij}^{(r)} q^r \in q^{d_i}\Pseries{\Z}{q}.
\]
With these preliminaries at our disposal, we are now in a position to present a uniform formula for $\mathcal{Q}_{i,V}^\g(u)$ and $\sigma_i(V)$. 
\begin{thm}\label{thm:Comb}
Let  $V$ be a finite-dimensional irreducible $\Yhg$-module with tuple of Drinfeld polynomials 
$\ul{P}=(P_j(u))_{j\in \bfI}$. Then 
\begin{equation*}
\mathcal{Q}_{i,V}^\g(u)
= \prod_{j\in \bfI} \prod_{s=d_i}^{2\kappa-d_i}P_j\!\left(u-(s-d_j)\frac{\hbar}{2}\right)^{\!v_{ij}^{(s)}}\! \quad \forall\; i\in \bfI.
\end{equation*}
Consequently, the $i$-th set of poles $\sigma_i(V)$ is given explicitly by
\begin{gather*}
\sigma_i(V)
=\bigcup_{j\in \bfI} \left(\zeroes(P_j(u))+\sigma_i(L_{\varpi_j})\right), \\
\sigma_i(L_{\varpi_j})=\left\{\frac{\hbar}{2}(s-d_j): d_i\leq s\leq 2\kappa - d_i \; \text{ and }\; v_{ij}^{(s)}>0\right\}.
\end{gather*}
\end{thm}
\begin{rem}\label{R:vij+}
Note that if the claimed formula for $\mathcal{Q}_{i,V}^\g(u)$ holds,  we must have $v_{ij}^{(s)}\geq 0$ for all $d_i\leq s\leq 2\kappa-d_i$. Indeed, this would follow by taking $V=L_{\varpi_j}$ and invoking the polynomiality of $\mathcal{Q}_{i,V}^\g(u)$. Though this property has already been established in \cite[Cor.~3.10]{fujita-oh}, we shall not need it in the proof of the theorem and will instead recover it as a byproduct; see Corollary \ref{C:vij}. 
\end{rem}
\begin{pf}
By Part \eqref{Baxter:4} of Theorem \ref{thm:Baxter} and Lemma \ref{L:Q-prod-fac}, it is sufficient to establish that the monic polynomial $\mathcal{Q}_{ij}^\g(u)=\mathcal{Q}_{i,L_{\varpi_j}}^\g(u)$ is given as in the statement of the theorem for each $i,j\in \bfI$. 
By \eqref{def:QiV^mu}, $\mathcal{Q}_{ij}^\g(u)$ is uniquely determined by the relation
\begin{equation}\label{Qij-def}
\upnu_{ij}^\g(u)=\frac{\mathcal{Q}_{ij}^\g(u+2\kappa \hbar)}{\mathcal{Q}_{ij}^\g(u)},
\end{equation}
where, by \eqref{upnu-fac}, $\upnu_{ij}^\g(u):=\upnu_{i,L_{\varpi_j}}^\g(u)$ is given explicitly by 
 \begin{equation*}
  \upnu_{ij}^\g(u)=\prod_{r\in\Z} \lp\frac{u+\hbar(\kappa+d_j)+\frac{r\hbar}{2}}
{u+\hbar\kappa + \frac{r\hbar}{2}}\rp^{c_{ij}^{(r)}}\!\!
\lp\frac{u+\hbar d_j+\frac{r\hbar}{2}}{u+\frac{r\hbar}{2}}\rp^{c_{ij^\ast}^{(r)}}.
 \end{equation*}
As the poles and zeroes of $\upnu_{ij}^\g(u)$ are contained in $\frac{\hbar}{2}\Z$ and $2\kappa\in \Z$, the relation \eqref{Qij-def} implies that  $\zeroes(\mathcal{Q}_{ij}^\g(u))\subset \frac{\hbar}{2}\Z$, and thus that $\mathcal{Q}_{ij}^\g(u)$ admits an expansion of the form
\begin{equation}\label{Q-fac-m_r}
\mathcal{Q}_{ij}^\g(u)=\prod_{r\in \Z} \left(u-\frac{r\hbar}{2}\right)^{\!m_r}
\end{equation}
with $m_r\in \N$ equal to zero for all but finitely many $r\in \Z$.   Using the shift of argument $\tau:f(u)\mapsto f(u+\hbar/2)$, we can therefore write the logarithmic derivative of the ratio $\mathbf{Q}^\g_{ij}(u):=\mathcal{Q}_{ij}^\g(u)^{-1}\mathcal{Q}_{ij}^\g(u+2\kappa\hbar)$ as 
\begin{equation*}
\frac{\partial_u \mathbf{Q}^\g_{ij}(u)}{\mathbf{Q}^\g_{ij}(u)}
=\sum_{r\in \Z} m_r\left(\frac{1}{u-\frac{\hbar r}{2}+2\kappa\hbar}-\frac{1}{u-\frac{\hbar r}{2}}\right)=(\tau^{4\kappa}-1)\left(\sum_{r\in \Z} m_r \tau^{-r}\right) \cdot \frac{1}{u}.
\end{equation*}
One the other hand, taking the logarithmic derivative of $\upnu_{ij}^\g(u)$ outputs the formula 
\begin{equation*}
\frac{\partial_u \mathbf{Q}_{ij}^\g(u)}{\mathbf{Q}^\g_{ij}(u)}
=(\tau^{2d_j}-1)\left( \tau^{2\kappa}c_{ij}(\tau)+c_{ij^\ast}(\tau)\right)\cdot\frac{1}{u}. 
\end{equation*}
We now employ the elementary fact that if $P(q)\in \Z[q,q^{-1}]$ satisfies $P(\tau)\cdot \frac{1}{u}=0$, then $P(q)=0$. This allows us to conclude from the above two identities that one has the equality of Laurent polynomials
\begin{equation}\label{m_r->c(q)}
({q^{4\kappa}-1})\sum_{r\in \Z} m_r q^{-r} =(q^{2d_j}-1)(c_{ij^\ast}(q)+q^{2\kappa} c_{ij}(q)).
\end{equation}
%
%
To complete the proof of the theorem, we shall rewrite this equality in two different ways in terms of the entries of the matrix $\mathbf{E}(q)$. Firstly, since $v_{ij}(q)= \frac{[d_j]_q}{[2\kappa]_q}c_{ij}(q)$, we obtain immediately from \eqref{m_r->c(q)} that
\begin{equation}\label{eq:m-v1}
\sum_{r\in \Z} m_{r-d_j} q^{-r}= v_{ij}(q)+q^{-2\kappa}v_{ij^\ast}(q).
\end{equation}
As $v_{ik}(q)\in q^{d_i}\Z[\![q]\!]$ for each $k\in \bfI$, this gives $m_{r-d_j}=0$ for all $r>2\kappa-d_i$ and $m_{r-d_j}=v_{ij^\ast}^{(2\kappa-r)}$ for all $-d_i< r\leq 2\kappa-d_i$. 
Next, since $\mathbf{C}(q)=\mathbf{C}(q^{-1})$, the relation \eqref{m_r->c(q)} may be expressed equivalently as 
\begin{equation}\label{eq:m-v2}
\sum_{r\in \Z} m_{r-d_j} q^r
=v_{ij}(q)+q^{2\kappa}v_{ij^\ast}(q).
\end{equation}
Using again that $v_{ik}(q)\in q^{d_i}\Z[\![q]\!]$ for each $k$, we deduce that $m_{r-d_j}=0$ for all $r<d_i$ and $m_{r-d_j}=v_{ij}^{(r)}$ for $d_i\leq r<2\kappa+d_i$. Combining this with the previous computation, we obtain $v_{ij}^{(2\kappa+r)}=0$ for all $|r|<d_i$ and
\begin{equation}\label{mrdj}
m_{r-d_j}= 
\begin{cases}
 v_{ij}^{(r)}=v_{ij^\ast}^{(2\kappa-r)} \; &\text{ if }\quad d_i\leq r\leq 2\kappa-d_i,\\
 0 \; &\text{ if }\quad r<d_i \; \text{ or }\; r>2\kappa-d_i. 
\end{cases} 
\end{equation}
Reinserting this into \eqref{Q-fac-m_r} yields the formula 
\begin{equation}\label{def:Qij}
\mathcal{Q}_{ij}^\g(u)
= \prod_{s=d_i}^{2\kappa-d_i}\left(u-(s-d_j)\frac{\hbar}{2}\right)^{\!v_{ij}^{(s)}}
\end{equation}
for all $i,j\in \bfI$, which is precisely the claimed expression for $\mathcal{Q}_{ij}^\g(u)$. \qedhere 
\end{pf}

\subsection{Properties of $v_{ij}(q)$}\label{ssec:vij}
In what follows, we set $v_{ij}^{(r)}=0$ if $r<d_i$. 
As an application of the proof of Theorem \ref{thm:Comb}, we recover the statement of
\cite[Cor.~3.10]{fujita-oh} (see also \cite[Lemma 3.7]{fujita2020} and \cite[\S2]{hernandez-leclerc2}): 
\begin{cor}\label{C:vij}
The coefficients $v_{ij}^{(r)}$ of $v_{ij}(q)$ have the following properties:
\begin{enumerate}[font=\upshape]
\item\label{vij:1} If $d_j\geq d_i$, then for each $r\in \Z$ we have 
\begin{equation*}
v_{ij}^{(r)}
=
\sum_{b=0}^{d_j/d_i -1} v_{ji}^{(r-(d_j/d_i) +1+2b)}.
\end{equation*}
\item\label{vij:2} For each $r\geq 0$, we have 
\begin{equation*}
v_{ij}^{(r+4\kappa)}=v_{ij}^{(r)} \quad \text{ and }\quad v_{ij}^{(r+2\kappa)}=-v_{ij^\ast}^{(r)}.
\end{equation*}
\item \label{vij:3}For each $0\leq r\leq 2\kappa$ and $0\leq s\leq 4\kappa$, we have 
\begin{equation*}
v_{ij}^{(2\kappa-r)}=v_{ij^\ast}^{(r)} \aand v_{ij}^{(4\kappa-s)}=-v_{ij}^{(s)}.
\end{equation*}
\item\label{vij:4} If $|r-2\kappa b|<d_i$ for some $b\geq 0$, then $v_{ij}^{(r)}=0$. 
\item \label{vij:5}For each $0\leq r\leq 2\kappa$ and $2\kappa\leq s\leq 4\kappa$, we have 
\begin{equation*}
v_{ij}^{(r)}\geq 0 \aand v_{ij}^{(s)}\leq 0. 
\end{equation*}
\end{enumerate}
\end{cor}
\begin{pf}
Part \eqref{vij:1} is the content of \cite[Lemma~3.3 (2)]{fujita-oh}. It is a consequence of the fact that $[d_j]_q^{-1} v_{ij}(q)=[2\kappa]_q^{-1}c_{ij}(q)$ is symmetric in $i$ and $j$, and hence one has
\begin{equation*}
v_{ij}(q)=[d_j/d_i]_{q^{d_i}}v_{ji}(q)
\end{equation*}
provided  $d_j\geq d_i$ (in which case $d_j/d_i$ is a positive integer). Part \eqref{vij:2} follows from the identities \eqref{eq:m-v1} and \eqref{eq:m-v2} together with the fact that $m_{r-d_j}=0$ for all $r<d_j$ and $r>2\kappa -d_j$ (see \eqref{mrdj}), which give 
\begin{equation*}
v_{ij}^{(r)}=-v_{ij^\ast}^{(r+2\kappa)} \aand v_{ij^\ast}^{(r)}=-v_{ij}^{(r+2\kappa)} \quad \forall \quad r> -d_i.
\end{equation*}
Similarly, \eqref{mrdj} implies that $v_{ij}^{(r)}\geq 0$ for all $d_i\leq r\leq 2\kappa-d_i$. Since we also have $v_{ij}^{(2\kappa+r)}=0$ for all $|r|<d_i$ (see above \eqref{mrdj}), this observation and the   second equality of Part \eqref{vij:2} yield Parts \eqref{vij:4} and \eqref{vij:5} of the corollary. Finally, Part \eqref{vij:3} follows from \eqref{mrdj}
and Parts \eqref{vij:2} and  \eqref{vij:4} of the corollary. \qedhere 
\end{pf}
\begin{rem}
Here we note that, as illustrated in \cite[Cor.~3.10 (5)]{fujita-oh}, Part \eqref{vij:4} may be strengthened to the assertion that $v_{ij}^{(r)}=0$ whenever $|r-2\kappa b|\leq d_i-\delta_{ij}$ for some $b\geq 0$. This follows from Lemma 3.3 of \cite{fujita-oh}, where it was shown using the definition of $v_{ij}(q)$ that $v_{ij}^{(d_i)}=\delta_{ij}$.  
\end{rem}

\subsection{Symmetries of $\mathcal{Q}_{ij}^\g(u)$ and $\sigma_i(L_{\varpi_j})$}\label{ssec:Qij}

For each $i\in \bfI$ and a positive integer $\ell$, $L_{\ell\varpi_i}$ denotes the unique, up to isomorphism, finite-dimensional irreducible $\Yhg$-module with Drinfeld polynomials $(P_j(u))_{j\in \bfI}$ given by $P_j(u)=1$ if $j\neq i$, and 
\begin{equation*}
P_i(u)=u(u+\hbar d_i)\cdots (u+(\ell-1)\hbar d_i).
\end{equation*} 
Note that the $\g$-weight $\lambda\in \h^\ast$ of any highest weight vector for $L_{\ell\varpi_i}$ is given by $\lambda=\deg P_i(u)\varpi_i=\ell \varpi_i$, which justifies our choice of notation. In general, a $\Yhg$-module which takes the form $L_{\ell\varpi_i}(a)$ for some $a\in \C$ is called a \textit{Kirillov--Reshetikhin module}.

The corollary below is a simple consequence of the description of $\mathcal{Q}_{ij}^\g(u)$ provided by Theorem \ref{thm:Comb} (see \eqref{def:Qij}), the relations of Corollary \ref{C:vij}, and that $\mathbf{E}(q)=(v_{ij}(q))_{i,j\in \bfI}$ is the inverse of the weighted $q$-Cartan matrix $([a_{ij}]_{q^{d_i}})_{i,j\in \bfI}$, which is invariant under diagram automorphisms.
\begin{cor}\label{C:Qij-symm}
The polynomials $\mathcal{Q}_{ij}^\g(u)$ and sets $\sigma_i(L_{\varpi_j})$ have the following properties:
\begin{enumerate}[font=\upshape]
\item\label{Qij-symm:1} For each diagram automorphism $\omega$ of $\g$, one has 
\begin{equation*}
\mathcal{Q}_{ij}^\g(u)=\mathcal{Q}_{\omega(i)\omega(j)}^\g(u)\quad\text{ and }\quad 
\sigma_i(L_{\varpi_j})=\sigma_{\omega(i)}(L_{\varpi_{\omega(j)}}).
\end{equation*}
\item\label{Qij-symm:2} $\mathcal{Q}_{ij}^\g(u)$ and $\sigma_i(L_{\varpi_j})$ admit the symmetries
\begin{gather*}
\mathcal{Q}^\g_{ij^*}(u+\hbar\kappa-\hbar d_j)
= (-1)^{\deg(\mathcal{Q}^\g_{ij}(u))} \mathcal{Q}^\g_{ij}(-u),\\
\sigma_i(L_{\varpi_j})=-\sigma_i(L_{\varpi_{j^\ast}})+\hbar\{\kappa-d_j\}.
\end{gather*}
\item\label{Qij-symm:3} Suppose $d_j\geq d_i$ and set $r_{ij}:=d_j/d_i\in \Z_{\geq 1}$. Then 
\begin{gather*}
\mathcal{Q}_{ij}^\g(u)=\mathcal{Q}_{j,L_{r_{ij}\varpi_i}}^\g(u) \quad \text{ and }\quad
\sigma_i(L_{\varpi_j})=\sigma_j(L_{r_{ij}\varpi_i}).
\end{gather*}
\end{enumerate}
\end{cor}
\begin{rem}\label{R:Qij->poles}
The relations in Parts \eqref{Qij-symm:1} and \eqref{Qij-symm:2} of Corollary \ref{C:Qij-symm} for the sets $\sigma_i(L_{\varpi_j})$ also follow easily from the definition of the $i$-th set of poles $\sigma_i(V)$. Indeed, as recalled in Section \ref{ssec:DDauto}, $\omega$ determines an algebra automorphism of $\Yhg$  which by definition satisfies $\sigma_i(L_{\varpi_j}^{\omega})=\sigma_{\omega(i)}(L_{\varpi_j})$ and $L_{\varpi_j}^{\omega}\cong L_{\varpi_{\omega(j)}}$, from which we obtain $\sigma_i(L_{\varpi_j})=\sigma_{\omega(i)}(L_{\varpi_{\omega(j)}})$. Similarly, the second relation follows by using the Cartan involution $\varphi$ of $\Yhg$ from Section \ref{ssec:CPauto} and Proposition \ref{pr:CPthm}. Indeed, the definition of $\varphi$ immediately implies that
\begin{equation*}
\sigma_i(L_{\varpi_j}^\varphi)=-\sigma_i(L_{\varpi_j}) \quad \forall\quad i,j\in \bfI. 
\end{equation*}
On the other hand, by Proposition \ref{pr:CPthm}, we have $L_{\varpi_j}^\varphi\cong L_{\varpi_{j^\ast}}(d_j\hbar-\kappa\hbar)$, and hence $\sigma_i(L_{\varpi_j}^\varphi)=\sigma_i(L_{\varpi_{j^\ast}})+\hbar\{d_j-\kappa\}$, which recovers the second relation of Part \eqref{Qij-symm:2}.

Finally, we note that, by \eqref{Qij-def},  Part \eqref{Qij-symm:3} is equivalent to the two identities 
\begin{gather*}
\mathcal{Q}_{ij}^\g(u)=\mathcal{Q}_{ji}^\g(u)\mathcal{Q}_{ji}^\g(u+\hbar)\cdots \mathcal{Q}_{ji}^\g(u+\dot{r}_{ij}\hbar),\\
\sigma_i(L_{\varpi_j})=\sigma_j(L_{\varpi_i})-\hbar\{b\in \Z: 0\leq b\leq \dot{r}_{ij}\}.
\end{gather*}
where $i,j\in \bfI$ are any two indices satisfying $d_j\geq d_i$ and $\dot{r}_{ij}=r_{ij}-1$. Here we have used the fact that $d_j>d_i$ implies that $d_i=1$. 
\end{rem}

\subsection{Example: \texorpdfstring{$\Yhsl{n}$}{Y(sl\_n)}}\label{ssec:ysln}

In this subsection, we restrict our attention to $\g=\sl_n$, with the goal of providing a detailed example of Theorem \ref{thm:Comb} and the above results. Fix $n\geq 2$, and take
 $\bfI = \{1,\ldots, n-1\}$. The entries of the Cartan matrix
$\bfA = (a_{ij})_{i,j\in\bfI}$ are then given by
\[
a_{ij} = \left\{ \begin{array}{cl} 2 & \text{if } i=j \\
-1 & \text{if } |i-j|=1 \\ 0 & \text{otherwise} \end{array}\right.
\]
Moreover, we have $2\kappa=n$ and the diagram automorphism $i\mapsto i^*$ induced by the longest element $w_0$ of the Weyl group
becomes $i^* = n-i$; see \S \ref{ssec:CPthm}.

For each $i,j\in \bfI$, let $\bfJ_{ij}$ denote the $\Z$-valued interval
\begin{equation*}
\bfJ_{ij}=[i+j+1-n,i]\cap [1,j].
\end{equation*}
The following corollary describes the monic polynomial $\mathcal{Q}_{ij}^{\sl_n}(u)$ and the $i$-th set of poles $\sigma_i(L_{\varpi_j})$ in terms of this interval. 
\begin{cor}\label{C:sln-comb}

For each $i,j\in\bfI$, the polynomial $\mathcal{Q}_{ij}^{\sl_n}(u)$ is given by 
\begin{equation*}
\mathcal{Q}_{ij}^{\sl_n}(u)=\prod_{b\in \bfJ_{ij}}\left(u-\hbar\left(\frac{i+j}{2}-b\right)\right).
\end{equation*}
Consequently, the $i$-th set of poles $\sigma_i(L_{\varpi_j})$ of $L_{\varpi_j}$ is given by
\begin{equation*}
\sigma_i(L_{\varpi_j})=\hbar \left\{ \frac{i+j}{2}-b: b\in [i+j+1-n,i]\cap [1,j]\right\}.
\end{equation*}
\end{cor}
\begin{pf}
By Theorem \ref{thm:Comb}, the identities \eqref{eq:m-v1} and \eqref{mrdj}, and the equality $v_{ij}(q)=[n]_q^{-1}c_{ij}(q)$, it is sufficient to establish the identity 
\begin{equation}\label{pij-sln}
\sum_{b\in \bfJ_{ij}}q^{-(i+j)+2b}=\frac{q^2-1}{q^{2n}-1}(c_{ij^\ast}(q)+q^nc_{ij}(q))
\end{equation}
By the explicit formulas given in \cite[\S A.3]{sachin-valerio-III}, the entry $c_{ij}(q)$ of $\mathbf{C}(q)$ is given by 
\begin{equation*}
c_{ij}(q)=
\begin{cases}
[n-i]_q[j]_q & \text{ if } i\geq j\\
[i]_q[n-j]_q & \text{ if } i< j
\end{cases}
\end{equation*}
Since both sides of the equality \eqref{pij-sln} are symmetric in $i$ and $j$, we may assume without loss of generality that $i\geq j$. Under this assumption, the above formula implies that 
\begin{equation*}
\frac{q^2-1}{q^{2n}-1}(c_{ij^\ast}(q)+q^nc_{ij}(q))=
\begin{cases}
q^{j-n+1}[n-i]_q & \text{ if }\; i+j\geq n\\
q^{-i+1}[j]_q & \text{ if }\; i+j< n
\end{cases}
\end{equation*}
Since $[m]_q=\sum_{k=0}^{m-1} q^{m-1-2k}$ for any $m\geq 0$, this can be rewritten as
\begin{equation*}
\frac{q^2-1}{q^{2n}-1}(c_{ij^\ast}(q)+q^nc_{ij}(q))=\sum_{k=0}^{m_{j,n-i}-1}q^{j-i-2k}=\sum_{b\in \bfJ_{ij}} q^{-(i+j)+2b},
\end{equation*}
where $m_{j,n-i}=\min\{j,n-i\}$ and to obtain the second equality we have made the substitution $b=j-k$.\qedhere
\end{pf}

The description of the $i$-th set of poles $\sigma_i(L_{\varpi_j})$ provided by the above corollary can also be easily deduced from the following explicit description of the $j$-th fundamental representation $L_{\varpi_j}$ of $\Yhsl{n}$.

Fix $j\in \bfI$. Let the standard basis of $\C^n$ be denoted by $\{\bra{1},\ldots,
\bra{n}\}$ and consider the subspace 
$V\subset \lp \C^n\rp^{\otimes j}$ defined by 
\[
V = \text{Span of } \{\bra{p_1}\otimes\cdots\otimes\bra{p_j} : 
1\leq p_1<\ldots<p_j\leq n\}.
\]
For $\fp = (p_1,\ldots,p_j)$, we will use the notation
$\bra{\fp} = \bra{p_1}\otimes\cdots\otimes\bra{p_j}$. 
Furthermore, for each $i\in \bfI$ and $1\leq \ell \leq j$, we define $b_{i,\ell}\in \C$ by 
\begin{equation*}
b_{i,\ell} = \frac{\hbar}{2}
(j+i-2\ell).
\end{equation*}
Then the fundamental representation $L_{\varpi_j}$ of $\Yhsl{n}$ may be realized on the space $V$, with action determined by the $\End(V)$-valued rational functions $\{\xi_i(u),x_i^{\pm}(u)\}_{i\in \bfI}$ defined on each tensor $\bra{\fp}$ as follows:
\begin{enumerate}
\item\label{ysln-fun:1} $\ds \xi_i(u)\bra{\fp} = \bra{\fp}$, if either $\{p_1,\ldots,p_j\}\cap
\{i,i+1\} = \emptyset$, or $\{i,i+1\} \subset \{p_1,\ldots,p_j\}$.
\begin{itemize}
\item If $i\in\{p_1,\ldots,p_j\}$ and $i+1\not\in\{p_1,\ldots,p_j\}$,
then: 
\[
\xi_i(u)\bra{\fp} = \frac{u+\hbar-b_{i,k}}{u-b_{i,k}}\bra{\fp},
\text{ where } p_k = i.
\]

\item If $i+1\in\{p_1,\ldots,p_j\}$ and $i\not\in\{p_1,\ldots,p_j\}$,
then: 
\[
\xi_i(u)\bra{\fp} = \frac{u-\hbar-b_{i,k}}{u-b_{i,k}}\bra{\fp},
\text{ where } p_k = i+1.
\]
\end{itemize}

\item\label{ysln-fun:2} $\ds x_i^+(u)\bra{\fp} = 0$, if either $i+1\not\in\{p_1,\ldots,p_j\}$,
or $\{i,i+1\}\subset\{p_1,\ldots,p_j\}$. If $\bra{\fp}$ is such that
$i+1\in\{p_1,\ldots,p_j\}$ (say $p_k=i+1$) and $i\not\in\{p_1,\ldots,p_j\}$,
then:
\[
x_i^+(u)\bra{\fp} = \frac{\hbar}{u-b_{i,k}} 
\bra{p_1,\ldots,p_{k-1},i,p_{k+1},\ldots,p_j}.
\]

\item\label{ysln-fun:3}  $\ds x_i^-(u)\bra{\fp} = 0$, if either $i\not\in\{p_1,\ldots,p_j\}$,
or $\{i,i+1\}\subset\{p_1,\ldots,p_j\}$. If $\bra{\fp}$ is such that
$i\in\{p_1,\ldots,p_j\}$ (say $p_k=i$) and $i+1\not\in\{p_1,\ldots,p_j\}$,
then:
\[
x_i^-(u)\bra{\fp} = \frac{\hbar}{u-b_{i,k}} 
\bra{p_1,\ldots,p_{k-1},i+1,p_{k+1},\ldots,p_j}.
\]
\end{enumerate}
Note that these formulas imply that, when viewed as an $\sl_n$-module via the inclusion $\sl_n \subset
\Yhsl{n}$, $V$ is isomorphic to the exterior product $\ds \wedge^j\C^n$, with
the canonical isomorphism given by
\[
\bra{p_1,\ldots,p_j} \mapsto \bra{p_1}\wedge\cdots\wedge\bra{p_j}.
\]
In particular, the highest weight vector of $L_{\varpi_j}\cong V$ is $\Omega=\bra{1,2,\ldots,j}$. 
 
Given this realization of $L_{\varpi_j}$, Proposition \ref{pr:sigma-i} implies that we compute $\sigma_i(L_{\varpi_j})$ by locating the poles of the single operator $x_i^+(u)$. By definition, this operator has poles at $b_{i,k}=\frac{\hbar}{2}(j+i-2k)$ for each
$1\leq k\leq j$ such that there exists $\fp$ with $p_k = i+1$
and $p_{k-1}<i$. For such
an increasing sequence to exist, it is necessary and sufficient
that $k\leq i$ and $j-k\leq n-i-1$, \ie $j+i+1-n\leq k\leq i$.
This gives us
\begin{equation*}
\spec_i(L_{\varpi_j}) = \hbar\left\{\frac{i+j}{2}-k : k\in [i+j+1-n,i]\cap [1,j]
\right\},
\end{equation*}
which recovers the description of $\spec_i(L_{\varpi_j})$ provided by Corollary \ref{C:sln-comb}.
\begin{rem}
Let us provide some context as to how the description of $L_{\varpi_j}$ given above was computed. For a fixed $b\in \C$, consider the $\Yhsl{n}$-action
on $\C^n$ given by the formulae 
\begin{gather*}
x^+_i(u)\bra{j} = \delta_{i+1,j} \frac{\hbar}{u-b_i} \bra{i}, \quad x^-_i(u)\bra{j} = \delta_{i,j} \frac{\hbar}{u-b_i} \bra{i+1}, \\
\xi_i(u)\bra{j} = \bra{j}+\frac{\hbar}{u-b_i}(\delta_{i,j}-\delta_{i+1,j})\bra{j},
\end{gather*}
where $b_i = b+\frac{\hbar}{2}(i-1)$. This representation,
denoted by $\C^n(b)$, is the fundamental
representation $L_{\varpi_1}(b)$.
The realization of $L_{\varpi_j}$ on $V\subset (\C^n)^{\otimes j}$ described above is nothing but the cyclic subrepresentation
generated by
$\Omega = \bra{1,2,\ldots,j}$ in the (Drinfeld)
tensor product 
\[
\C^n\lp \tfrac{\hbar}{2}(j-1)\rp 
\otimes_D \C^n\lp \tfrac{\hbar}{2}(j-3)\rp \otimes_D \cdots\otimes_D
\C^n\lp -\tfrac{\hbar}{2}(j-1)\rp,
\]
as defined in \cite[\S 4.5]{sachin-valerio-III}.
It is worth pointing out that the Drinfeld tensor product $V\otimes_D W$
of two \fd representations of $\Yhg$ is defined in
\cite[\S 4.5]{sachin-valerio-III} under the assumption that $\spec(V)
\cap \spec(W) = \emptyset$. However, for this definition and the
proof of \cite[Thm.~4.6]{sachin-valerio-III}, it is enough to
assume the weaker condition that $\spec_i(V)\cap\spec_i(W)=\emptyset$,
for every $i\in\bfI$. Otherwise, the tensor product written above would not be defined.

Alternatively, one can achieve the above description of $L_{\varpi_j}$ by using the evaluation morphism $\Yhsl{n}\twoheadrightarrow U(\sl_n)$, which is described explicitly in terms of the generating series $\xi_i(u),x_i^\pm(u)\subset\Yhsl{n}[\![u^{-1}]\!]$ in Theorem 5.1 of \cite{andrea-sachin}, to extend the $U(\sl_n)$-action on $V\cong \wedge^j\C^n$ to an $\Yhsl{n}$-action. The resulting representation necessarily coincides with $L_{\varpi_j}(a)$ for some $a\in \C$. However, it is perhaps a less trivial exercise to verify that this procedure recovers the  operators defined in \eqref{ysln-fun:1}--\eqref{ysln-fun:3} above.
\end{rem}

\subsection{The Frenkel--Hernandez formula}\label{ssec:Q-props}

In the context of quantum loop algebras, the specialized Baxter polynomial arising from the $U_q(L\g)$-analogue of the eigenvalue of $\overline{T}_i(u)$ on any generalized eigenspace $V[\ul{\mu}]$ of $V$ has been computed in Proposition 5.8 of \cite{frenkel-hernandez}; see also \cite[\S12.7]{hernandez-qshifted} for calculations specific to the $U_q(L\g)$-analogue of $\mathcal{Q}_{i,V}^\g(u)$. This result gives a description of each eigenvalue of $\overline{T}_i(u)$ in terms of the data arising from a prescribed factorization of the term in the $q$-character of $V$ corresponding to $\ul{\mu}$. This section is devoted to translating these formulas to the Yangian setting, which, given the identification from Corollary \ref{C:Ti-vs-Ri}, recovers a special case of  Proposition 5.8 (ii) from the recent article \cite{hernandez-zhang}.

Let $V=L(\underline{P})$, and let $\underline{\mu}=(\mu_j(u))_{j\in \bfI}$ be a weight of $V$. Then, by \cite[Cor.~7]{Zhang20}, there are $p\in \Z_{\geq 0}$, $i_1,\ldots,i_p\in \bfI$ and $b_1,\ldots,b_p\in \C$ such that
\begin{equation}\label{mu_j-fac}
\mu_j(u)=\prod_{k=1}^p \left(\frac{u-b_k-\hbar d_{j,i_k}}{u-b_k+\hbar d_{j,i_k}}\right)\frac{P_j(u+\hbar d_j)}{P_j(u)},
\end{equation}
where $d_{ji}=d_j a_{ji}/2$. In the same breath, there are $q\in \Z_{\geq 0}$, $j_1,\ldots,j_p\in \bfI$ and $c_1,\ldots,c_p\in \C$ such that
\begin{equation}\label{mu_j-fac'}
\mu_j(u)=\prod_{\ell=1}^q \left(\frac{u-c_\ell+\hbar d_{j,j_\ell}}{u-c_\ell-\hbar d_{j,j_\ell}}\right)\frac{P_{j^\ast}(u-\hbar\kappa)}{P_{j^\ast}(u-\hbar \kappa+ \hbar d_j)}.
\end{equation}
Indeed, the existence of such a decomposition follows from \eqref{mu_j-fac}, Proposition \ref{pr:CPthm}, and the fact that $\underline{\mu}^\varphi=(\mu_j(-u))_{j\in \bfI}$ is a weight of $V^\varphi$. 

The $U_q(L\g)$-analogue of the result below is established in Proposition 5.8 in \cite{frenkel-hernandez}; see also \cite[Prop.~9.8]{hernandez-qshifted}. Given Corollary \ref{C:Ti-vs-Ri}, it can be seen as a consequence of Proposition 5.8 (ii) in \cite{hernandez-zhang}. 
\begin{prop}\label{P:FH-Q}
Let $V$ and $\underline{\mu}$ be as above. Then the eigenvalue $\mathcal{Q}_{i,V[\underline{\mu}]}^{\g}(u)$ of $\overline{T}_i(u)$ on $V[\underline{\mu}]$ satisfies 
\begin{equation*}
\mathcal{Q}_{i,V[\underline{\mu}]}^{\g}(u)=\prod_{k:i_k=i}(u-b_k) 
\quad \text{ and }\quad  
\mathcal{Q}_{i,V}^\g(u)/\mathcal{Q}_{i,V[\underline{\mu}]}^{\g}(u)=\prod_{\ell:j_\ell=i}(u-c_\ell).
\end{equation*}
In particular, the components of $\underline{\mu}$ satisfy
\begin{equation*}
\mu_j(u)=\prod_{i\in \bfI}\frac{\mathcal{Q}_{i,V[\underline{\mu}]}^{\g}(u-\hbar d_{ji})}{\mathcal{Q}_{i,V[\underline{\mu}]}^{\g}(u+\hbar d_{ji})}\frac{P_j(u+\hbar d_j)}{P_j(u)} \quad \forall\quad j\in \bfI.
\end{equation*}
\end{prop}
\begin{pf}
Note that the third relation is an immediate consequence of the first and the decomposition \eqref{mu_j-fac}.
The proofs of the first two relations can be reduced to the same type of argument as given in \cite[Prop.~5.8]{frenkel-hernandez} by applying the techniques used to prove Theorem \ref{thm:Comb}. 
Indeed, using \eqref{def:QiV^mu} together with the factorizations \eqref{mu_j-fac} and \eqref{mu_j-fac'}, it is easy to reduce the proofs of these identities to the following claim: If $\mathcal{Q}_{i,k}^b(u)\in \C[u]$ is monic and satisfies
\begin{equation} \label{u-b:solve}
\prod_{\substack{j\in \bfI\\r\in \Z}}\left(\frac{u-b + \hbar(\kappa+ d_{j,k})+\frac{\hbar r}{2}}{u-b+\hbar (\kappa-d_{j,k})+\frac{\hbar r}{2}}\right)^{c_{ij}^{(r)}}
=
\frac{\mathcal{Q}_{i,k}^b(u+2\kappa\hbar)}{\mathcal{Q}_{i,k}^b(u)},
\end{equation}
then $\mathcal{Q}_{i,k}^b(u)=(u-b)^{\delta_{i,k}}$.

As in the proof of Theorem \ref{thm:Comb} (see \eqref{Q-fac-m_r}), this defining equation implies that $\mathcal{Q}_{i,k}^b(u)=\prod_{r\in \Z}(u-b-\frac{\hbar r}{2})^{m_r}$ for some $m_r\in \N$. Taking the logarithmic derivative of both sides of the above identity then recovers
\begin{equation*}
\sum_{j\in \bfI} c_{ij}(\tau) \tau^{2\kappa} (\tau^{d_j a_{jk}}-\tau^{-d_j a_{jk}}) \cdot \frac{1}{u-b}
=
(\tau^{4\kappa}-1)\left(\sum_{r\in \Z}m_r \tau^{-r}\right)\cdot \frac{1}{u-b},
\end{equation*}
where we recall that $\tau$ is the shift operator $f(u)\mapsto f(u+\hbar/2)$.
By \eqref{eq:BC=l} and the same reasoning as given in the proof of Theorem \ref{thm:Comb}, we may conclude that the integers $m_r$ satisfy
\begin{equation*}
\delta_{i,k}(q^{2\kappa}-q^{-2\kappa})=\sum_{j\in \bfI} c_{ij}(q) (q^{d_j a_{jk}}-q^{-d_j a_{jk}}) = (q^{2\kappa}-q^{-2\kappa})\sum_{r\in \Z}m_r q^{-r},
\end{equation*}
and are thus given by  $m_r=\delta_{i,k}\delta_{r,0}$, as desired. \qedhere
\end{pf}
\begin{rem}
In the language of $q$-characters, the factorizations \eqref{mu_j-fac} and \eqref{mu_j-fac'} can be expressed in the more familiar forms
\begin{equation*}
e(\underline{\mu})=m_+ A_{i_1,b_1}^{-1}A_{i_2,b_2}^{-1}\cdots A_{i_p,b_p}^{-1} \quad \text{ and }\quad e(\underline{\mu})=m_- A_{j_1,c_1}A_{j_2,c_2}\cdots A_{j_q,c_q},
\end{equation*}
where $m_\pm=e(\ul{\mu}^\pm)$ with $\ul{\mu}^+$ and $\ul{\mu}^-$ the highest and lowest weights of $V$, respectively, and $A_{i,b}$ is the generalized simple root monomial 
\begin{equation*}
A_{i,b}= e((\Psi_{\hbar d_{ji}}(u-b))_{j\in \bfI}), \quad \text{ with }\quad \Psi_{x}(u)=
\frac{u+x}{u-x}.
\end{equation*}
 For the quantum loop algebra $U_q(L\g)$, the existence of the above expansions for $e(\ul{\mu})$ was first established in Theorem 4.1 of \cite{frenkel-mukhin}. 

In the case where $e(\underline{\mu})$ is itself $m_-$, Proposition \ref{P:FH-Q} implies that the multiset of roots of $\mathcal{Q}_{i,V}^\g(u)$ consists of those $b_j$
appearing in the above factorization for which $i_j=i$. The proof of Theorem \ref{thm:Comb} gives one self-contained method for computing these $b_j$ explicitly. We emphasize that it does not, however, depend on Proposition \ref{P:FH-Q} or require any information about the factorizations \eqref{mu_j-fac} and \eqref{mu_j-fac'}. 
%

\end{rem}

\section{Poles and characters}\label{sec:char}

Our primary goal in this section is to prove Theorem \ref{thm:diagonal-poles} which asserts that, for any $V\in \Ryang$, the $i$-th set of poles $\sigma_i(V)$ is equal to the set of poles of the eigenvalues of the operator $\xi_i(u)$, and is therefore encoded by $\chi_q(V)$.

\subsection{Poles and $q$-characters}\label{ssec:diagonal-poles}

For each $V\in \Ryang$ and $i\in \bfI$, we define the $i$-th set of poles $\sigma_i(\chi_q(V))\subset \C$ of the $q$-character $\chi_q(V)$ by
\begin{equation*}
\sigma_i(\chi_q(V)):=\bigcup_{\ul{\mu}:V[\ul{\mu}]\neq 0}\!\sigma(\mu_i(u)),
 \end{equation*}
 where  $\sigma(\mu_i(u))\subset \C$ denotes the set of poles of $\mu_i(u)$.  Equivalently, $\sigma_i(\chi_q(V))$ is the set of poles $\sigma(\xi_{i,S}(u);V)\subset \C$ of the diagonal part $\xi_{i,S}(u)$ of $\xi_i(u)$, defined by 
 \begin{equation*}
 \xi_{i,S}(u):=\sum_{\ul{\mu}:V[\ul{\mu}]\neq 0} \mu_i(u)\id_{\ul{\mu}}\in \End(V)(u),
 \end{equation*}
 where $\id_{\ul{\mu}} : V\to V[\ul{\mu}]$ is the natural projection operator associated to the decomposition $V=\bigoplus_{\ul{\mu}}V[\ul{\mu}]$. As usual, we omit all subscripts $i$  when $\g=\sl_2$.
\begin{thm}\label{thm:diagonal-poles}
Let $V$ be a finite-dimensional $\Yhg$-module. Then
\begin{equation*}
\sigma_i(V)=\sigma_i(\chi_q(V))=\sigma(\chi_q(\varphi_i^\ast(V))) \quad \forall \quad i\in \bfI.
\end{equation*}
\end{thm}
\begin{pf}  
Note that it is sufficient to prove this statement
for $\g=\sl_2$; see \eqref{sigma-sl2}. Thus, throughout the remainder of our proof, which occupies \S\ref{ssec:extns}--\ref{ssec:irr-Ysl2}, we restrict our attention to $\Yhsl{2}$.

We prove the theorem by induction on the length of a composition
series of $V$. We carry out the induction step in Proposition \ref{pr:extns} below. Afterwards, we verify the base case, \ie when $V$ is irreducible, in Section \ref{ssec:irr-Ysl2}; see Proposition \ref{P:eigen}. 
\end{pf}

\subsection{Extensions}\label{ssec:extns}
Let $V_1,V_2$ and $V_3$ be three \fd representations of $\Yhsl{2}$. Assume
that $\spec(V_{\ell}) = \spec(\xi_S(u);V_{\ell})$ for $\ell=1,2$,
and that we have a short exact sequence of $\Yhsl{2}$ representations:
\[
0\to V_1\to V_3\to V_2\to 0\ .
\]
\begin{prop}\label{pr:extns}
Under the hypotheses stated above, we have
\[
\spec(V_3) = \spec(\xi_S(u);V_3) =  \spec(\xi_S(u);V_1)\cup \spec(\xi_S(u);V_2)
= \spec(V_1)\cup \spec(V_2).
\]
\end{prop}
\begin{pf}
Note that the equality $\spec(\xi_S(u);V_3) = \spec(\xi_S(u);V_1)\cup
\spec(\xi_S(u);V_2)$ is obvious, since $\xi_S(u)$ is semisimple and
\[
\left. \xi_S(u)\right|_{V_3} = 
\left. \xi_S(u)\right|_{V_1}\oplus
\left. \xi_S(u)\right|_{V_2}.
\]
The last equality $\spec(\xi_S(u);V_1)\cup \spec(\xi_S;V_2)
= \spec(V_1)\cup \spec(V_2)$ holds by assumption. 
Moreover, since $V_3$ is an extension of $V_2$ by $V_1$, it
is clear that $\spec(V_1)\cup\spec(V_2)\subset \spec(V_3)$.
Combining this with Proposition \ref{pr:sigma-i}, we find that 
$\spec(V_3) = \spec(\xi_S(u);V_3)$ will hold provided 
$\spec(x^+(u);V_3) \subset \spec(V_1)\cup \spec(V_2)$.

For the purposes of this proof, we write $V_3 = V_1\oplus V_2$ as an
$\sl_2$-representation. Every element $y\in\Yhsl{2}$ can be
written in the block form
\[
y = \left[\begin{array}{cc} y_{11} & y_{12} \\ 0 & y_{22}
\end{array}\right],
\]
where $y_{ij} : V_j \to V_i$, for each $i,j\in\{1,2\}$.

Now suppose, towards a contradiction, that there
is $z\in\spec(x^+(u);V_3)$ such that $z\not\in \spec(V_1)
\cup \spec(V_2)$. Let $N = \order_z(x^+(u);V_3) = \order_z(\xi(u);V_3)$
(see Proposition \ref{pr:sigma-i}). Define
elements $X^+$ and $H$ of $\End(V_3)$ by
\[
X^+ = \lim_{u\to z} (u-z)^N x^+(u)
\aand
H = \lim_{u\to z} (u-z)^N \xi(u).
\]
Note that, since $z\not\in \spec(V_1)\cup\spec(V_2)$, the diagonal
blocks of $X^+$ and $H$ are zero. We consider the $(1,2)$ component
of the relation
\ref{sY5}:
\begin{multline*}
x^+(u)_{11}x^-(v)_{12} + x^+(u)_{12}x^-(v)_{22}
- x^-(v)_{11}x^+(u)_{12} - x^-(v)_{12}x^+(u)_{22} \\
= \frac{\hbar}{u-v} (\xi(v)_{12}-\xi(u)_{12}).
\end{multline*}
Multiplying both sides by $(u-z)^N$ and letting $u\to z$, we obtain the identity
\[
X^+x^-(v)_{22} - x^-(v)_{11}X^+ = \frac{H}{v-z}.
\]
Since the \lhs of this equation is regular at $v=z$, so must be
the right--hand side, which proves that $H=0$. This is a contradiction
to Proposition \ref{pr:sigma-i}.
\end{pf}

\subsection{Finite-dimensional irreducible representations of $\Yhsl{2}$}
\label{ssec:irr-Ysl2}
To complete the proof of Theorem \ref{thm:diagonal-poles}, we are left to establish that it holds for $\g=\sl_2$ with $V$ taken to be a finite-dimensional irreducible module. By Theorems \ref{thm:dp} and \ref{thm:Comb}, this amounts to establishing that 
\begin{equation}\label{diag-poles-sl2}
\mathsf{Z}(P(u))\subset \spec(\xi_S(u);L(P))
\end{equation}
for any monic polynomial $P(u)\in \C[u]$, where $L(P)$ is the finite-dimensional irreducible $\Yhsl{2}$-module with Drinfeld polynomial $P(u)$. Here we have used that for $\g=\sl_2$ one has $\kappa=1$ and $\mathbf{C}(q)=1$, from which it follows by Theorem \ref{thm:Comb} that 
\begin{equation}\label{eq:Baxter-sl2}
\mathcal{Q}_{L(P)}^{\sl_2}(u)=P(u)\aand\sigma(L(P))=\zeroes(P(u)).
\end{equation}

We shall give a combinatorial proof of \eqref{diag-poles-sl2} in Proposition \ref{P:eigen} which draws inspiration from the work of Young \cite[\S6]{young-15} and is based on properties of the lowering operators which featured prominently in the proof of Theorem \ref{thm:Baxter}. To this end, recall from \eqref{eq:partial-fractions} that, given a finite-dimensional $\Yhsl{2}$-module $V$, the $\End(V)$-valued rational function $x^-(u)$ admits a partial fraction decomposition 
\begin{equation*}
x^-(u)=\sum_{\substack{b\in \sigma(V)\\n\in \Z_{\geq 0}}} \frac{X_{b,n}^-}{(u-b)^{n+1}}. 
\end{equation*}
We shall need the following commutation relations for the operators $X_{b,n}^-$, which are analogous to those provided by Lemma \ref{lem:comm-xi-pf}.
\begin{lem}\label{L:comm}
Let $a,b\in \sigma(V)$ with $a\neq b-\hbar$. Then, for each $m,n\in \N$, one has 
\begin{equation*}
X_{a,m}^-X_{b,n}^-=\sum_{p,q\geq 0}   \partial_{u}^{(q)}\!\left.\left(\partial_v^{(p)}\!\left.\left( \frac{u-v-\hbar}{u-v+\hbar}\right)\right|_{v=b}\right)\right|_{u=a}\!\cdot X_{b,n+p}^- X_{a,m+q}^-.
\end{equation*}
\end{lem}
\begin{pf}
By \ref{sY4}, the rational function $x^-(u)$ satisfies the relation 
\begin{equation*}
x^-(u)x^-(v)=\frac{u-v-\hbar}{u-v+\hbar} x^-(v)x^-(u)+\frac{\hbar}{u-v+\hbar}[x^-_0, x^-(v)+x^-(u)].
\end{equation*}
Multiplying by $(v-b)^n$ and integrating over a small contour enclosing $b$ yields 
\begin{align*}
x^-(u)X_{b,n}^-=\sum_{p\geq 0}& \partial_v^{(p)}\!\left.\left( \frac{u-v-\hbar}{u-v+\hbar}\right)\right|_{v=b} X_{b,n+p}^- x^-(u)\\
 & +\sum_{p\geq 0} \partial_v^{(p)}\!\left.\left( \frac{1}{u-v+\hbar}\right)\right|_{v=b} [x_0^-, X_{b,n+p}^-].
\end{align*}
Note that each of the rational coefficients appearing on the right-hand side has a single pole located at $u=b-\hbar$, which does not coincide with $a$ by assumption. Hence, multiplying by $(u-a)^m$ and integrating over a small contour enclosing $a$ outputs the relation from the assertion of the lemma. \qedhere
\end{pf}

For the remainder of this subsection  we fix a monic polynomial $P(u)$ with set of roots $\mathsf{Z}:=\zeroes(P(u))$, together with a root $b\in \mathsf{Z}$. In addition, we let $\Omega^+$ and $\Omega^-$ be highest and lowest weight vectors of $L(P)$, respectively,  and we define $\mathsf{Z}_b^+\subset \mathsf{Z}$ to be the subset
\begin{equation*}
\mathsf{Z}_b^+=\mathsf{Z}\cap (b+\hbar \Z_{\geq 0}).
\end{equation*}
We view $\mathsf{Z}_b^+$ as an ordered set, with ordering inherited from $\Z_{\geq 0}$. We say that a  monomial $X=X_{a_1,n_1}^- X_{a_2,n_2}^-\cdots X_{a_m,n_m}^-$ in $\{X_{a,n}^-\}_{a\in \mathsf{Z}_b^+,n\in \Z_{\geq 0}}$ is $\mathsf{Z}_b^+$-\textit{complete} if
\begin{enumerate}
\item\label{Zb:1} The vector $\Omega_X:=X\cdot \Omega^+$ is nonzero. 
\item\label{Zb:2} $a_1\preceq a_2\preceq \cdots \preceq a_m$, with respect to the ordering on $\mathsf{Z}_b^+$. 
\item\label{Zb:3} For each $\beta\in \mathsf{Z}_b^+$, the number $|\{j:a_j=\beta\}|$ is the multiplicity $m(\beta)$ of $\beta$ as a root of $P(u)$.  
\end{enumerate}
Given a complete monomial $X$ as above, we set $\|X\|:=\sum_{j=1}^m n_j$. We say that such a monomial is \textit{maximal} if there is no $\mathsf{Z}_b^+$-complete monomial $Y$ with $\|Y\|>\|X\|$. 
The next proposition uses this formalism to prove that $b$ necessarily occurs as a pole of an eigenvalue of $\xi(u)$, thus completing the proof of  \eqref{diag-poles-sl2} and Theorem \ref{thm:diagonal-poles}. 
\begin{prop}\label{P:eigen}
There exists a $\mathsf{Z}_b^+$-complete maximal monomial $X$. Moreover $\Omega_X$ is an eigenvector for $\xi(u)$ with eigenvalue $\mu_X(u)$ determined by 
\begin{equation*}
P(u)\mu_X(u)=
\prod_{\beta\in \mathsf{Z}_b^+}(u-\beta-\hbar)^{m(\beta)}\prod_{\alpha\in \mathsf{Z}\setminus \mathsf{Z}_b^+}(u-\alpha+\hbar)^{m(\alpha)}.
\end{equation*}
In particular, $\mu_X(u)$ has a pole at $u=b$. 
\end{prop}
\begin{proof}[{\sc Proof}]
Since $L(P)$ is irreducible, we may fix a chain
$(X_{c_1,m_1}^-,\ldots,X_{c_p,m_p}^-)$ of lowering operators satisfying the conditions
\begin{itemize}
\item $X_{c_1,m_1}^-\cdots X_{c_p,m_p}^- \Omega^+ \in\C\Omega^-$.
\item For each $1\leq j\leq p$, $m_j$ is the largest non-negative
integer such that 
\[
\Omega_j:= X_{c_j,m_j}^-\cdots X_{c_p,m_p}^-\Omega^+
\neq 0.
\]
\end{itemize}
In particular, $\Omega_1$ is a non-zero scalar
multiple of $\Omega^-$ and, by convention,  $\Omega_{p+1}=\Omega^+$. 

Observe that, given such a chain, one has $P(u)=\prod_{j=1}^p(u-c_j)$. Indeed, this follows from the relation \eqref{eq:Ti-xk}, which together with the maximality of each $m_j$ implies that 
\begin{equation*}
\overline{T}(u)\Omega^-=\mathcal{Q}_{L(P)}^{\sl_2}(u)\Omega^-=(u-c_1)(u-c_2)\cdots (u-c_p)\Omega^-.
\end{equation*}

Next, since $\Omega^-$ is proportional to $\Omega_1$, Lemma \ref{L:comm} implies that it may be written as a linear combination of vectors of the form $M_{b}X_b\Omega^+$, where $X_b$ is a monomial in $\{X_{a,n}^-\}_{a\in \mathsf{Z}_b^+,n\in \Z_{\geq 0}}$  satisfying the defining conditions \eqref{Zb:2} and \eqref{Zb:3} of a  $\mathsf{Z}_b^+$-complete monomial, and $M_{b}$ is a monomial in $\{X_{c,k}^-\}_{c\in\mathsf{Z}\setminus \mathsf{Z}_b^+, k\in \Z_{\geq 0}}$. Since $\Omega^-$ is nonzero, this proves the existence of a $\mathsf{Z}_b^+$-complete monomial. As there are clearly only finitely many such monomials, there is a $\mathsf{Z}_b^+$-complete monomial $X$ with the property that $\|X\|$ is maximal. 
Moreover, it follows from the maximality of $X$ and the commutation relation \eqref{eq:xi-pf1} of Lemma \ref{lem:comm-xi-pf} that 
%
 $\Omega_X=X\cdot \Omega^+$ is a $\xi(u)$ weight vector of weight
\begin{equation*}
\mu_X(u)=\frac{P(u+\hbar)}{P(u)}\prod_{\beta\in \mathsf{Z}_b^+}\left(\frac{u-\beta-\hbar}{u-\beta+\hbar}\right)^{m(\beta)}. \qedhere
\end{equation*}
\end{proof}

\subsection{Poles and composition series}
\label{ssec:cors}
The following properties of $\spec_i(V)$ follow directly
from the equality \eqref{sigma-sl2},
Theorem \ref{thm:diagonal-poles}, Proposition
\ref{pr:extns} and that, by Theorem \ref{thm:Comb}, the $\Yhsl{2}$-module $L(P)$ has set of poles $\sigma(L(P))=\zeroes(P(u))$; see \eqref{eq:Baxter-sl2}. 
\begin{cor}\label{cor:sigma}
Let $V$ be a \fd representation of $\Yhg$, and fix $i\in \bfI$. 
\begin{enumerate}[font=\upshape]
\item\label{c-sigma:2}  If
$\{V_1,\ldots,V_N\}$ are the simple factors in a composition series
of $V$, then
\[
\spec_i(V) = \bigcup_{k=1}^N \spec_i(V_k).
\]

\item \label{c-sigma:3} If $\{W_1,\ldots,W_M\}$ are the simple factors in a composition series for the $Y_{d_i\hbar}(\sl_2)$-module $\varphi_i^\ast(V)$, then 
\begin{equation*}
\spec_i(V) = \bigcup_{k=1}^M \spec(W_k)=\bigcup_{k=1}^M\zeroes(P_k(u)),
\end{equation*}
where $P_k(u)$ is the Drinfeld polynomial associated to $W_k$: $W_k\cong L(P_k)$.
\end{enumerate}
\end{cor}

\subsection{Poles and $i$-dominant weights}\label{ssec:i-dom}

For a fixed $i\in \bfI$, we shall say that $\ul{\mu}\in \mathcal{L}^\bfI$ is an $i$-\textit{dominant weight} of $V$ if
$V[\ul{\mu}]\neq 0$ and there is a (necessarily unique) monic polynomial $P_{\mu_i}(u)$ such that 
\begin{equation*}\label{idom-P}
\mu_i(u)=\frac{P_{\mu_i}(u+\hbar d_i)}{P_{\mu_i}(u)}.
\end{equation*}
 Let $\mathsf{w}_i^+(V)\subset \mathcal{L}^\bfI$ denote the set of $i$-dominant weights of $V$.
\begin{prop}\label{P:q-Poles} Let $V$ be a finite-dimensional $\Yhg$-module. Then, for each $i\in \bfI$, we have 
\begin{equation*}
\sigma_i(V)=\bigcup_{\ul{\mu}\in \mathsf{w}_i^+(V)}\hspace{-.5em}\zeroes(P_{\mu_i}(u)).
\end{equation*}
\end{prop}
\begin{pf}
By Part \eqref{c-sigma:3} of Corollary \ref{cor:sigma}, it suffices to prove the proposition in the case where $\g=\sl_2$ and $V=L(P)$ for a fixed monic polynomial $P(u)\in \C[u]$. As observed in \eqref{eq:Baxter-sl2}, Theorem \ref{thm:Comb} yields
\begin{equation*}
\sigma(L(P))=\zeroes(P(u))\subset \bigcup_{\mu\in \mathsf{w}^+_P}\hspace{-.5em}\zeroes(P_{\mu}(u)),
\end{equation*}
where we have written $\mathsf{w}^+_P$ for the set of dominant weights of $L(P)$. 
 We are thus left to show that if $\mu(u)$ is a dominant weight of $L(P)$ then $\zeroes(P_\mu(u))\subset\zeroes(P(u))$.
To establish this we will employ Proposition \ref{P:FH-Q}, which implies that any weight $\mu(u)$ of $V=L(P)$ takes the form
\begin{equation*}
\mu(u)
=
\frac{\mathcal{Q}_{V[\mu]}^{\sl_2}(u-\hbar)\mathcal{P}_{V[\mu]}^{\sl_2}(u+\hbar)}{P(u)}=\frac{\mathcal{Q}_{V[\mu]}^{\sl_2}(u-\hbar)}{\mathcal{Q}_{V[\mu]}^{\sl_2}(u)}\cdot\frac{\mathcal{P}_{V[\mu]}^{\sl_2}(u+\hbar)}{\mathcal{P}_{V[\mu]}^{\sl_2}(u)},
\end{equation*}
where we have set $\mathcal{P}_{V[\mu]}^{\sl_2}(u)=P(u)/\mathcal{Q}_{V[\mu]}^{\sl_2}(u)=\mathcal{Q}_V^{\sl_2}(u)/\mathcal{Q}_{V[\mu]}^{\sl_2}(u)$, which is a monic polynomial by Proposition \ref{P:FH-Q}.  Consequently, if $\mu(u)$ is a dominant weight with associated polynomial $P_\mu(u)$, then $P_\mu(u)\mathcal{Q}_{V[\mu]}^{\sl_2}(u-\hbar)=\mathcal{P}_{V[\mu]}^{\sl_2}(u)$. In particular, $P_\mu(u)$ divides $P(u)$ and so $\zeroes(P_\mu(u))\subset \zeroes(P(u))=\sigma(L(P))$. \qedhere
\end{pf}

\subsection{Poles and tensor products}\label{ssec:strong}
As another application of Theorem \ref{thm:diagonal-poles}, we deduce that the operation of taking the $i$-th set of poles is compatible with tensor products. More precisely, we have the following result:
\begin{cor}\label{cor:tensor}
Let $V$ and $W$ be \fd representations of $\Yhg$. Then, for each index $i\in \bfI$, one has 
\begin{equation*}
\spec_i(V\otimes W) \subset \spec_i(V)\cup \spec_i(W). 
\end{equation*}
\end{cor}
\begin{pf}
This is an immediate consequence of Proposition \ref{pr:delta-xi-tr} (or, alternatively, the equality $\chi_q(V\otimes W)=\chi_q(V)\chi_q(W)$ established in \cite[Thm.~2]{knight}; see \S\ref{ssec:q-char}) and Theorem \ref{thm:diagonal-poles}. 
\end{pf}

We now apply this corollary together with Theorem \ref{thm:Comb} to prove a strong tensor property for an arbitrary finite-dimensional irreducible $\Yhg$-module. 
Given an $\bfI$-tuple of Drinfeld polynomials $\ul{P}=(P_j(u))_{j\in \bfI}$, let us expand 
each polynomial $P_j(u)$ explicitly as 
\begin{equation*}
P_j(u) = \prod_{\ell=1}^{n_j} \lp u-a^{(j)}_\ell\rp. 
\end{equation*}
\begin{prop}\label{prop:STP}
Let $V$ be the \fd \irr representation
of $\Yhg$ with Drinfeld polynomials $\ul{P}=(P_j(u))_{j\in\bfI}$, as above.
Consider the tensor product
\begin{equation*}
\mathcal{V} = \bigotimes_{j\in\bfI} \lp \bigotimes_{\ell=1}^{n_j}
L_{\varpi_j}\lp a^{(j)}_\ell\rp \rp,
\end{equation*}
taken with respect to any fixed order. 
Then, for each $i\in \bfI$, we have
\begin{equation*}
\sigma_i(V) = \bigcup_{j\in\bfI}
\lp \zeroes(P_j(u)) + \sigma_i(L_{\varpi_j})\rp=\sigma_i(\mathcal{V}).
\end{equation*}
\end{prop}

\begin{pf}
Let $\Omega^+\in\mathcal{V}$ be the tensor product of the highest-weight vectors
of each tensor factor in $\mathcal{V}$
Then
$V$ is a quotient of the subrepresentation of $\V$ generated by $\Omega^+$; see \cite[Prop.~2.15]{chari-pressley-singularities}, for instance. 
Hence, by Corollary \ref{cor:tensor}, we have
\begin{equation*}
\sigma_i(V)\subset \sigma_i(\V)\subset \bigcup_{j\in \bfI}\bigcup_{\ell=1}^{n_j}\sigma_i(L_{\varpi_j}(a_\ell^{(j)}))=\bigcup_{j\in \bfI}\left( \zeroes(P_j(u))+\sigma_i(L_{\varpi_j})\right) \quad \forall\; i\in \bfI,
\end{equation*}
where we have used that $\sigma_i(V(a))=\sigma_i(V)+\{a\}$ in the last equality.
The proposition now follows from the Theorem \ref{thm:Comb} which, in particular, established that the first and last sets appearing above coincide. 
 \qedhere
\end{pf}
%
%

\section{Poles, cyclicity conditions and \texorpdfstring{$R$}{R}-matrices}\label{sec:cyc-R}

In this section, we apply the description of the sets $\{\sigma_i(V)\}_{i\in \bfI}$ provided by Proposition \ref{P:q-Poles} to illustrate that they provide a control over the cyclicity and simplicity of any tensor product $V\otimes W$ of finite-dimensional, irreducible $\Yhg$-modules; see Theorem \ref{T:cyclic} and Corollary \ref{C:irr-cond}. Afterwards, we reinterpret some of our results in terms of the poles of the normalized rational $R$-matrices associated to tensor products of the form $V\otimes L_{\varpi_i}(s)$ for $i\in \bfI$ and $V$ a finite-dimensional irreducible $\Yhg$-module; see Sections \ref{ssec:R-matrix}--\ref{ssec:conjecture}.

\subsection{Poles and cyclicity of tensor products}\label{ssec:path}
 We begin with a technical lemma which has appeared in various forms in the literature \cite{chari-braid,guay-tan,tan-braid}.

Let $V$ be a finite-dimensional irreducible $\Yhg$ module with highest weight vector $v$. Let $\lambda\in \h^\ast$ be the $\g$-weight of $v$, so $\lambda=\sum_{i\in \bfI}\deg(P_i)\varpi_i$, where $(P_i(u))_{i\in \bfI}$ is the Drinfeld tuple associated to $V$, and $V_\lambda=\C v$. Let 
\[
w_0=s_{i_1}s_{i_2}\cdots s_{i_{p}}
\]
 be a reduced expression for the longest element of the Weyl group of $\g$, where each $s_k$ is a simple reflection. For each $0\leq j\leq p$, introduce the Weyl group element $w_j:=s_{i_{j+1}}s_{i_{j+2}}\cdots s_{i_p}$ and define weight vectors $v_{w_j(\lambda)}\in V_{w_j(\lambda)}$ recursively by 
\begin{equation*}
v_{w_j(\lambda)}:=(x_{i_{j+1}}^-)^{m_{j+1}}v_{w_{j+1}(\lambda)},
\end{equation*}
where $w_p=\id$, $v_{w_p(\lambda)}=v$, and  $m_j=d_{i_j}^{-1}(\alpha_{i_j},w_j(\lambda))$ for all $0\leq j\leq p$.

The following is a restatement of \cite[Lemma 4.3]{tan-braid}, where we employ the terminology from Section \ref{ssec:i-dom}; see also \cite[Lemma~5.1]{guay-tan}. 
\begin{lem}\label{L:path}
Let $V$ be a finite-dimensional irreducible $\Yhg$-module
with highest weight vector $v\in V_\lambda$ as above, and fix $0<j\leq p$. Then:
\begin{enumerate}[font=\upshape]

\item\label{path:1} The vector $v_{w_j(\lambda)}$ is a highest weight vector in the  $Y_{d_{i_j}\hbar}(\sl_2)$-module 
\begin{equation*}
Y_{d_{i_j}\hbar}(\sl_2)v_{w_j(\lambda)}=Y_\hbar(\g_{i_j})v_{w_j(\lambda)}\subset \varphi_{i_j}^\ast(V).
\end{equation*}
\item\label{path:2} There is an $i_j$-dominant weight $\ul{\mu}^{w_j}=(\mu_i^{w_j}(u))_{i\in \bfI}\in \mathsf{w}_{i_j}^+(V)$ such that 
\begin{equation*}
\xi_i(u)v_{w_j(\lambda)}=\mu_i^{w_j}(u)v_{w_j(\lambda)} \quad \forall\quad i\in \bfI.
\end{equation*}
\end{enumerate}
\end{lem}
Here we note that, since $V_{w_j(\lambda)}$ is one-dimensional and preserved by all $\xi_i(u)$, the second assertion of the lemma is an immediate consequence of the first.

\subsection{}\label{ssec:cyclic}
With the above lemma at our disposal, we are now prepared to prove the following theorem. 
\begin{thm}\label{T:cyclic} 
Let $V$ and $W$ be finite-dimensional irreducible $\Yhg$-modules with Drinfeld polynomials 
 $\ul{P}=(P_i(u))_{i\in \bfI}$ and $\ul{Q}=(Q_i(u))_{i\in \bfI}$, respectively, and highest weight vectors $v\in V$ and $w\in W$. Suppose that 
\begin{equation*}
\zeroes(Q_i(u+\hbar d_i))\subset \C\setminus \spec_i(V) \quad \forall\quad i\in \bfI. 
\end{equation*}
Then $V\otimes W$ is a highest weight module with highest weight vector $\Omega=v\otimes w$.
\end{thm}
\begin{pf} 
The theorem follows from Proposition \ref{P:q-Poles} and an argument due to Chari which was first given for quantum loop algebras in \cite[Thm.~4.4]{chari-braid}, and adapted to the Yangian setting by Guay and Tan in \cite{guay-tan,tan-braid}. In order to illustrate exactly in what context the poles $\spec_i(V)$ arise in this argument, we shall review some of its key ingredients, following \cite[Thm.~4.8]{tan-braid} and \cite[Thm.~5.2]{guay-tan}. In what follows, we employ freely the notation from \S\ref{ssec:path}.

It is easy to see that the vector $v_{w_0(\lambda)}\in V_{w_0(\lambda)}$ defined above Lemma \ref{L:path} is the unique, up to scalar multiplication, lowest weight vector of the $\Yhg$-module $V$.  By \cite[Lemma 2.13]{guay-tan}, the vector $v_{w_0(\lambda)}\otimes w$ generates $V\otimes W$.  In particular,  $V\otimes W$ will be a highest weight module with highest weight vector $v\otimes w$ provided 
\begin{equation}\label{v-low-high}
v_{w_0(\lambda)}\otimes w\in \Yhg(v\otimes w). 
\end{equation}
Let $P_{w_j}(u)$ denote the Drinfeld polynomial of the irreducible quotient of the highest weight module $Y_{d_{i_j}\hbar}(\sl_2)v_{w_j(\lambda)}$ from Part \eqref{path:1} of Lemma \ref{L:path}.
It was proven in \cite[Thm.~4.8]{tan-braid} that, for each fixed $1\leq j\leq p$, one has 
\begin{equation*}
v_{w_{j-1}(\lambda)}\otimes w\subset Y_{\hbar}(\g_{i_j})(v_{w_j(\lambda)}\otimes w)\subset \Yhg(v_{w_j(\lambda)}\otimes w)
\end{equation*}
provided $\hbar d_{i_j}\notin \zeroes(Q_{i_j}(u))-\zeroes(P_{w_j}(u))$. The proof of this result, which we do not reproduce here, is based on a cyclicity condition for tensor products of $Y_\hbar(\sl_2)$-modules obtained in earlier work of Chari and Pressley; see \cite[Cor.~3.8]{chari-pressley-RepsChar}. 

It follows by downward induction on $j$ that the condition \eqref{v-low-high} will be satisfied provided $\hbar d_{i_j}\notin \zeroes(Q_{i_j}(u))-\zeroes(P_{w_j}(u))$ for all $1\leq j\leq p$. 
On the other hand since, by Lemma \ref{L:path}, each weight $\ul{\mu}^{w_j}$ is $i_j$-dominant, Proposition \ref{P:q-Poles} yields
\begin{equation*}
\bigcup_{i_j=k}\zeroes(P_{w_j}(u))\subset \spec_{k}(V) \quad \forall \quad k\in \bfI,
\end{equation*}
where the union is taken over all $j\in \bfI$ such that $i_j=k$. 
It follows that $V\otimes W$ will be a highest weight module with highest weight vector $v\otimes w$ provided
\begin{equation*}
\hbar d_i\notin \zeroes(Q_i(u))-\spec_i(V) \quad \forall\quad i\in \bfI,
\end{equation*}%
and therefore if $\zeroes(Q_i(u+\hbar d_i))\subset \C\setminus \spec_i(V)$ for all $i\in \bfI$.
\qedhere
\end{pf}
\subsection{Poles and irreducibilty of tensor products} \label{ssec:irr-cond}
 We now apply Theorem \ref{T:cyclic} to obtain an irreducibility criterion for tensor products of simple $\Yhg$-modules, expressed in terms of the poles of the underlying modules.
\begin{cor}\label{C:irr-cond}
Let $V$ and $W$ be finite-dimensional irreducible $\Yhg$-modules with Drinfeld polynomials 
 $\ul{P}=(P_i(u))_{i\in \bfI}$ and $\ul{Q}=(Q_i(u))_{i\in \bfI}$, respectively. Suppose that 
\begin{equation*}
\zeroes(Q_i(u+\hbar d_i))\subset \C\setminus \sigma_i(V) 
\aand 
\zeroes(P_i(u+\hbar d_i))\subset \C\setminus \sigma_i(W) 
\end{equation*}
for all $i\in \bfI$. Then $V\otimes W$ is irreducible and isomorphic to the module $L(\ul{PQ})$, where $\ul{PQ}=(P_i(u)Q_i(u))_{i\in \bfI}$. 
\end{cor}
\begin{pf}
If $V\otimes W$ is irreducible, then it follows from Proposition 2.15 of \cite{chari-pressley-singularities}, for instance, that it is necessarily isomorphic to $L(\ul{PQ})$. 
Hence, it suffices to establish the simplicity of $V\otimes W$ given the conditions of the corollary. 
By Proposition 3.8 of \cite{chari-pressley-dorey}, an arbitrary finite-dimensional $\Yhg$-module $U$ is irreducible if and only if both $U$ and $U^\ast$ are highest weight modules. Here $U^\ast$ is the (left) dual of $U$, with $\Yhg$ action given by 
\begin{equation*}
(x\cdot f)(w)=f(S(x)\cdot w) \quad \forall \quad  x\in \Yhg,\; f\in V^\ast \; \text{ and }\; w\in U, 
\end{equation*} 
where $S$ is the antipode of the Hopf algebra $\Yhg$. Applying this result to $U=V\otimes W$, we can conclude using Theorem \ref{T:cyclic} that it is sufficient to prove that $(V\otimes W)^\ast$ is a highest weight module provided 
\begin{equation}\label{VxW-irr}
\zeroes(P_i(u+\hbar d_i))\subset \C\setminus \sigma_i(W) 
\quad \forall\quad i\in \bfI.
\end{equation}
To this end, we note that there is a sequence of $\Yhg$-module isomorphisms 
\begin{equation*}
(V\otimes W)^\ast\cong W^\ast \otimes V^\ast \cong W^\omega(-\kappa)\otimes V^\omega(-\kappa)\cong (W \otimes V)^\omega(-\kappa),
\end{equation*}
where $\omega$ denotes the diagram automorphism $i\mapsto i^\ast$ induced by longest element of the Weyl group of $\g$; see \S\ref{ssec:DDauto} and \S\ref{ssec:CPthm}. Here the first isomorphism follows from the fact that $S$ is a coalgebra anti-automorphism of $\Yhg$, the second isomorphism is a consequence of \cite[Prop.~3.5]{chari-pressley-dorey}, and the third isomorphism is due to the fact that $\tau_{-\kappa}$ (see \S\ref{ssec:shift}) and the automorphism of $\Yhg$ induced by $\omega$ (see \S\ref{ssec:DDauto}) are both morphisms of Hopf algebras. 

It is clear from the definition of these automorphisms that $(W \otimes V)^\omega(-\kappa)$, and therefore $(V\otimes W)^\ast$, is a highest weight module if and only if $W\otimes V$ is. By Theorem \ref{T:cyclic}, this will be the case provided the condition \eqref{VxW-irr} satisfied. \qedhere
\end{pf}

\subsection{Remarks} \label{ssec:cyclic-rmks}

Let us now give a sequence of remarks relevant to Theorem \ref{T:cyclic} and Corollary \ref{C:irr-cond}:
\begin{enumerate}\setlength{\itemsep}{5pt}\setlength{\parskip}{3pt}

\item\label{cyc-rmks:1} Taking $W=L_{\varpi_i}(s)$ in Theorem \ref{T:cyclic} for a fixed $i\in \bfI$ and $s\in \C$, we find that the module $V\otimes L_{\varpi_i}(s)$ is a highest weight module provided
\begin{equation*}
s-\hbar d_i\notin\sigma_i(V).
\end{equation*}
We note that, in light of Theorem \ref{thm:Comb}, this gives an explicit cyclicity criterion, expressed in terms of the Drinfeld polynomials of $V$ and the Cartan data of $\g$. 

\item\label{cyc-rmks:2} In general, the property \eqref{cyc-rmks:1} above does not characterize the set of poles $\sigma_i(V)$. That is, the set $\mathsf{C}_{V,i}(\g)-\hbar d_i$, where 
\begin{equation*}
\mathsf{C}_{V,i}(\g)=\{s\in \C: V\otimes L_{\varpi_i}(s) \text{ is \textit{not} highest weight}\}\subset \C
\end{equation*}
is generally a \textit{strict} subset of $\sigma_i(V)$. For instance, if $V=L_{\varpi_j}$ for any $j\in \bfI$ satisfying $d_j>d_i$, then $\sigma_j(L_{\varpi_i})\subsetneq \sigma_i(L_{\varpi_j})$ by Corollary \ref{C:Qij-symm}; see Remark \ref{R:Qij->poles}. It then follows from Proposition \ref{P:fun-denom} below that, for such a pair of indices $i,j\in \bfI$, one has 
\begin{equation*}
\mathsf{C}_{L_{\varpi_j},i}(\g)-\hbar d_i\subsetneq \sigma_i(L_{\varpi_j}).
\end{equation*}
In Section \ref{ssec:R-fun}, we will give a more precise conjectural relationship between $\mathsf{C}_{L_{\varpi_j},i}(\g)$ and the sets $\sigma_i(L_{\varpi_j})$ and $\sigma_j(L_{\varpi_i})$ which is strongly supported by recent results for the quantum loop algebra $U_q(L\g)$; see Proposition \ref{P:fun-denom} and Conjecture \ref{conj:dij}. In particular, if $\g$ is simply laced then the equality $\mathsf{C}_{L_{\varpi_j},i}(\g)-\hbar d_i=\sigma_i(L_{\varpi_j})$ is expected to hold for all $i,j\in \bfI$. 

\item \label{cyc-rmks:3}As an example, consider the case where $\g=\sl_n$. Then by Theorem 6.2 of \cite{chari-pressley-dorey}, one has
\begin{equation*}
\mathsf{C}_{L_{\varpi_j,i}}(\sl_n)-\hbar
=\hbar \left\{ \frac{i+j}{2}-b: b\in [i+j+1-n,i]\cap [1,j]\right\}=\sigma_i(L_{\varpi_j})
\end{equation*}
for each $i,j\in \bfI=\{1,\ldots,n-1\}$ satisfying $i\geq j$, 
where we have employed Corollary \ref{C:sln-comb} in the second equality. Since $\sigma_j(L_{\varpi_i})=\sigma_i(L_{\varpi_j})$ and, by Parts \eqref{fun-denom:1} and \eqref{fun-denom:2} of Proposition \ref{P:fun-denom} below, $\mathsf{C}_{L_{\varpi_j,i}}(\sl_n)=\mathsf{C}_{L_{\varpi_i,j}}(\sl_n)$, the above equality holds for all $i,j\in \bfI$. 

\item Taking $V=L_{\varpi_j}(s_j)$ and $W=L_{\varpi_i}(s_i)$ in Corollary \ref{C:irr-cond} for some $i,j\in \bfI$ and $s_i,s_j\in \C$, we find that $L_{\varpi_j}(s_j)\otimes L_{\varpi_i}(s_i)$ will be irreducible provided 
\begin{equation*}
s_i-s_j-\hbar d_i \notin \sigma_i(L_{\varpi_j})\aand s_j-s_i-\hbar d_j\notin  \sigma_j(L_{\varpi_i}).
\end{equation*}
For $\g=\sl_n$, these conditions are both sufficient and necessary; see \cite[Thm.~5.21]{guay-tan} and \eqref{cyc-rmks:3}. More generally, this is expected to be true for all simply laced types; see Conjecture \ref{conj:dij}.
For arbitrary $\g$, Theorem \ref{thm:Comb} implies that 
\begin{equation*}
\sigma_i(L_{\varpi_j})\subset\frac{\hbar}{2}\{b\in \Z:d_i-d_j\leq b\leq 2\kappa-d_i-d_j\}.
\end{equation*}
Therefore, the above ireducibility condition yields a strengthening of the $\Yhg$-analogue of the irreducibility condition from \cite[Prop.~6.15]{frenkel-mukhin} when only two modules are considered. 
\end{enumerate}
Finally, as the proof of Theorem \ref{T:cyclic} passes through the arguments, due to Chari \cite{chari-braid}, used to establish \cite[Thm.~4.8]{tan-braid} and \cite[Thm.~5.2]{guay-tan}, we expect that the cyclicity condition it establishes is equivalent to that obtained in \cite{tan-braid}. In the case where $V$ and $W$ are both fundamental representations, this could be verified case-by-case using Theorem \ref{thm:Comb}, known expressions for the integers $v_{ij}^{(r)}$ from \cite[\S3]{fujita-oh} and \cite[\S A]{sachin-valerio-III}, and the formulas obtained in \cite[Thm.~5.17]{guay-tan} and \cite[Lemma.~5.1]{tan-braid}. For $\g=\sl_n$, this is verified by Corollary \ref{C:sln-comb}, as spelled out in \eqref{cyc-rmks:3} above. 

However, it would be interesting to obtain a uniform identification of these conditions, for any $V$ and $W$, by interpreting the sets $\sigma_i(V)$ and polynomials $\mathcal{Q}_{i,V}^\g(u)$ in terms of the rational counterpart of Chari's braid group action; see \cite[\S3]{chari-braid} and \cite[\S3]{tan-braid}.

\subsection{The normalized abelian $R$-matrix}\label{ssec:R-matrix}
We now turn our attention to obtaining a partial interpretation of some of the above results in the language of rational $R$-matrices. Along the way, we shall translate a number of well-known results for the quantum loop algebra $U_q(L\g)$ to the setting of $\Yhg$.

Let $\mathcal{R}(u)\in \Yhg^{\otimes 2}[\![u^{-1}]\!]$ denote the universal $R$-matrix of the Yangian, as first introduced by Drinfeld in \cite{drinfeld-qybe}. We refer the reader to Appendix \ref{ssec:App-uniR} for a summary of its defining properties. In this section, we shall rely on the constructive proof of the existence of $\mathcal{R}(u)$ obtained in Section 7.4 of the authors joint work \cite{sachin-valerio-curtis} with Toledano Laredo. Therein, $\mathcal{R}(u)$ was rebuilt via a direct construction of the components in its Gaussian decomposition
\begin{equation*}
\mathcal{R}(u)=\mathcal{R}^+(u)\mathcal{R}^0(u)\mathcal{R}^-(u).
\end{equation*} 
The element $\mathcal{R}^-(u)\in (Y^-_\hbar(\g)\otimes Y^+_\hbar(\g))[\![u^{-1}]\!]$ was constructed in  \cite[\S4]{sachin-valerio-curtis}, the factor $\mathcal{R}^+(u)$ is given by $\mathcal{R}^+(u)=\mathcal{R}^-_{21}(-u)^{-1}$, and the diagonal factor $\mathcal{R}^0(u)\in (Y^0_\hbar(\g)\otimes Y^0_\hbar(\g))[\![u^{-1}]\!]$ was formally obtained in \cite[\S6]{sachin-valerio-curtis} using the earlier work  \cite{sachin-valerio-III} of the first author and Toledano Laredo. Here $Y^\pm_\hbar(\g)$ and $Y^0_\hbar(\g)$ are the subalgebras of $\Yhg$ generated by $\{x_{j,r}^\pm\}_{j\in \bfI,r\in \N}$ and  $\{\xi_{j,r}\}_{j\in \bfI,r\in \N}$, respectively. 

In what follows, we shall only be concerned with the evaluation of $\mathcal{R}(u)$ on tensor products of the form $V\otimes L_{\varpi_i}$, where $V$ is simple. Furthermore, we temporarily narrow our focus to the abelian $R$-matrix $\mathcal{R}^0(u)$, in order to clarify its relation with $T_i(u)$ and $A_i(u)$ from Section \ref{ssec:Ti}. For each $i\in \bfI$ and finite-dimensional irreducible $\Yhg$-module $V$, we set 
\begin{equation*}
\mathsf{R}_{V,L_{\varpi_i}}^0(u)=f(u)^{-1}\mathcal{R}^{0}(-u)|_{V\otimes L_{\varpi_i}}\in \End(V\otimes L_{\varpi_i})[\![u^{-1}]\!],
\end{equation*}
where $f(u)\in \C[\![u^{-1}]\!]$ denotes the eigenvalue of $\mathcal{R}^{0}(-u)$ on $\Omega^+\otimes v_i$, with $\Omega^+\in V$ and $v_i\in L_{\varpi_i}$ highest weight vectors. This is the expansion at $u=\infty$ of an $\End(V\otimes L_{\varpi_i})$-valued rational function of $u$ \cite{drinfeld-qybe, sachin-valerio-curtis}. Moreover, by \cite[\S5.8]{sachin-valerio-III} (see also \cite[\S6.6]{sachin-valerio-curtis}), it satisfies the  abelian difference equation 
\begin{equation}\label{R-A}
\mathsf{R}_{V,L_{\varpi_i}}^0(u+2\kappa\hbar)=\bar{\mathsf{A}}_{V,L_{\varpi_i}}(u)^{-1}\cdot \mathsf{R}_{V,L_{\varpi_i}}^0(u),
\end{equation}
where $\bar{\mathsf{A}}_{V,L_{\varpi_i}}(u)=g(u)^{-1}\mathsf{A}_{V,L_{\varpi_i}}(u)$, with $\mathsf{A}_{V,L_{\varpi_i}}(u)$ the operator on $V\otimes L_{\varpi_i}$ defined by 
\begin{equation}\label{A-op}
\mathsf{A}_{V,L_{\varpi_i}}(u)=\prod_{\substack{k,j\in \bfI\\r\in \Z}}\exp\left( \oint_{\mathcal{C}} t_j\!\left(v+u+\kappa\hbar+\frac{r\hbar }{2}\right)\otimes \frac{d t_k (v)}{dv} dv\right)^{-c_{kj}^{(r)}},
\end{equation}
and $g(u)$ the eigenvalue of $\mathsf{A}_{V,L_{\varpi_i}}(u)$ on $\Omega^+\otimes v_i$. 
Here $\mathcal{C}$ is a contour enclosing the poles of $\xi_k(u)^{\pm 1}\in \End(V)(u)$, and $t_j(u)$ denotes the logarithm $\log(\xi_j(u))$, defined precisely in \cite[\S5.4]{sachin-valerio-III}. 

Next, recall that, for each $i\in \bfI$, $A_i(u)$ is the $\End(V)$-valued rational function given explicitly by \eqref{eq:Ai} (see also Remark \ref{rem:T-motive2} below). 
The following lemma gives a characterization of $A_i(u)$ in terms of $\mathsf{A}_{V,L_{\varpi_i}}(u)$, and thus in terms of $\mathcal{R}^0(u)$. 
\begin{lem}\label{L:A->Ai}
Let $V$ be a finite-dimensional irreducible $\Yhg$-module and fix $i\in \bfI$. Then $A_i(u)$ is the unique $\End(V)$-valued rational function of $u$ satisfying
\begin{equation}\label{A->Ai}
\mathsf{A}_{V,L_{\varpi_i}}(u)|_{V\otimes (L_{\varpi_i})_{\varpi_i}}=A_i(u- \hbar d_i)A_i(u)^{-1}\otimes \Id_{(L_{\varpi_i})_{\varpi_i}}
\end{equation}
in addition to  $A_i(\infty)=\Id_V$. 
\end{lem}
\begin{proof}[\sc Proof]
That $A_i(u)$ satisfies \eqref{A->Ai} is deduced directly from the definition of $\mathsf{A}_{V,L_{\varpi_i}}(u)$ using Cauchy's integral formula and the identity
\begin{equation*}
\frac{d t_k (v)}{dv}|_{(L_{\varpi_i})_{\varpi_i}}=\xi_k(v)^{-1}\xi_k^\prime(v)|_{(L_{\varpi_i})_{\varpi_i}}=-\delta_{i,k}\frac{\hbar d_i}{v(v+\hbar d_i)}\Id_{(L_{\varpi_i})_{\varpi_i}}.
\end{equation*}
If $B_i(u)$ is another such solution then $B_i(u)^{-1}A_i(u)$ is an $\End(V)$-valued function which takes value $\Id_V$ at $u=\infty$ and is simultaneously rational and periodic with period $-\hbar d_i$. This is only possible if $B_i(u)^{-1}A_i(u)=\Id_V$.
\end{proof}
\begin{rem}\label{rem:T-motive2}
In the notation of \cite{sachin-valerio-III}, we have $\mathsf{A}_{V,L_{\varpi_i}}(u)=(1\,2)\circ \mathcal{A}_{L_{\varpi_i},V}(u)\circ (1\,2)$; see Theorem 5.5 therein. 
We further note that taking the right-hand side of \eqref{A-op} and replacing $\xi_k(v)$ by  $v^{\delta_{k,i}}$ in the second tensor factor with $\mathcal{C}$ a small contour enclosing $v=0$, gives
\begin{equation*}
\prod_{\substack{j\in \bfI\\r\in \Z}}\exp\left( \oint_{\mathcal{C}} t_j\!\left(v+u+\kappa\hbar+\frac{r\hbar }{2}\right) \frac{dv}{v}\right)^{-c_{kj}^{(r)}}\!=\prod_{\begin{subarray}{c} j\in\bfI\\ r\in\Z\end{subarray}}
\xi_j\lp u + \kappa\hbar + \frac{\hbar}{2} r\rp^{-c_{ij}^{(r)}}\!=A_i(u).
\end{equation*}
Equivalently, $A_i(u)$ is obtained from $\mathcal{A}_{L_{\varpi_i},V}(u)$ by making the formal substitution $\xi_k(v)\mapsto v^{\delta_{k,i}}$ in the first tensor factor. This is consistent with Remark \ref{rem:T-motive}, and is the observation which motivated our definition of $T_i(u)$ in Section \ref{ssec:Ti}. 
\end{rem}
\subsection{Denominators of normalized $R$-matrices}\label{ssec:denom}

With $V$ and $i\in \bfI$ as above, we now set $\mathsf{R}_{V,L_{\varpi_i}}^\pm(u)=\mathcal{R}^{\pm}(-u)|_{V\otimes L_{\varpi_i}}$. By Theorem 4.1 of \cite{sachin-valerio-curtis}, these are the expansions at $u=\infty$ of operator valued rational functions. 
We may thus introduce a rational, normalized $R$-matrix $\mathsf{R}_{V,L_{\varpi_i}}(u)\in \End(V\otimes L_{\varpi_i})(u)$ by setting
\begin{equation*}
\mathsf{R}_{V,L_{\varpi_i}}(u)=\mathsf{R}^+_{V,L_{\varpi_i}}(u)\mathsf{R}_{V,L_{\varpi_i}}^0(u)\mathsf{R}^-_{V,L_{\varpi_i}}(u).
\end{equation*}
Here we note that the negative signs have been introduced in order to be consistent with the literature on quantum loop algebras. Namely, if $s\in \C$ is not a pole of $\mathsf{R}_{V,L_{\varpi_i}}(u)$, then 
\begin{equation*}
(1\,2)\circ \mathsf{R}_{V,L_{\varpi_i}}(s): V\otimes L_{\varpi_i}(s)\to L_{\varpi_i}(s)\otimes V
\end{equation*}
is a $\Yhg$-module homomorphism. 

Now let $\mathsf{d}_{V,i}(u)\in \C[u]$ be the \textit{denominator} of  $\mathsf{R}_{V,L_{\varpi_i}}(u)$. That is, $\mathsf{d}_{V,i}(u)$ is the monic polynomial in $u$ of minimal degree such that 
\begin{equation*}
\mathsf{d}_{V,i}(u)\mathsf{R}_{V,L_{\varpi_i}}(u)\in \End(V\otimes L_{\varpi_i})[u].
\end{equation*}
In particular, the set of roots $\mathsf{Z}(\mathsf{d}_{V,i}(u))$ consists precisely of the poles of $\mathsf{R}_{V,L_{\varpi_i}}(u)$. The following standard fact gives a partial relation between this set and the cyclicity of $V\otimes L_{\varpi_i}(s)$. 
\begin{lem}\label{L:cyc-R} Let $i\in \bfI$, and suppose that $V$ is a finite-dimensional irreducible $\Yhg$-module and that $s\in \C$ is such that $V\otimes L_{\varpi_i}(s)$ is a highest weight module. Then the $R$-matrix $\mathsf{R}_{V,L_{\varpi_i}}(u)$ does not have a pole at $u=s$. That is, $\mathsf{d}_{V,i}(s)\neq 0$. 
\end{lem}
\begin{pf}
This is a well-known result; as the argument is brief we repeat it here for the sake of completeness. 
Suppose $V\otimes L_{\varpi_i}(s)$ is highest weight, but that $\mathsf{R}_{V,L_{\varpi_i}}(u)$ has a pole of order $N>0$ at $u=s$. Then
\begin{equation*}
\bar{\mathsf{R}}_{V,L_{\varpi_i}}(s):=(1\,2)\circ \lim_{u\to s}(u-s)^N \mathsf{R}_{V,L_{\varpi_i}}(u): V\otimes L_{\varpi_i}(s) \to L_{\varpi_i}(s)\otimes V
\end{equation*}
is a nonzero $\Yhg$-module homomorphism satisfying $\bar{\mathsf{R}}_{V,L_{\varpi_i}}(s)(\Omega^+\otimes v_i)=0$, which is impossible as $\Omega^+\otimes v_i$ generates $V\otimes L_{\varpi_i}(s)$. \qedhere
\end{pf}

\subsection{Specialized Baxter polynomials and $R$-matrices}
Suppose now that $\underline{\mu}=(\mu_i(u))_{i\in \bfI}$ is a weight of the $\Yhg$-module $V$. Recall from  Section \ref{sec:Combinatorial} that the specialized Baxter polynomial $\mathcal{Q}_{i,V[\ul{\mu}]}^\g(u)$ is the eigenvalue of $\overline{T}_i(u)$ on $V[\underline{\mu}]$; see \eqref{def:QiV^mu}.

Let us fix an ordered basis of $V[\underline{\mu}]$ for which each $\xi_j(u)$ is upper triangular, and extend this to a $\g$-weight basis of $V$. Similarly, we fix a $\g$-weight basis of $L_{\varpi_i}$ containing the highest weight vector $v_i$. Let $\Omega_\mu\in V[\underline{\mu}]$ be any fixed basis vector, and set $\Omega_{\mu,i}=\Omega_\mu\otimes v_i$. 
The counterpart of the second result below for $U_q(L\g)$ is an immediate corollary of Lemma 2.6 of \cite{frenkel-mukhin} and Proposition 5.8 of \cite{frenkel-hernandez}; see also Theorem 5.11 and Remark 5.12 (i) of \cite{frenkel-hernandez}.
\begin{prop}\label{P:R-pole}
Let $V$ be a finite-dimensional irreducible $\Yhg$-module, fix $i\in \bfI$, and let $\Omega_{\mu,i}\in V\otimes L_{\varpi_i}$ be as above. Then:
\begin{enumerate}[font=\upshape]
\item  \label{cyc-R:1} The set of poles of $\mathsf{R}_{V,L_{\varpi_i}}(u+\hbar d_i)$ is a subset of $\sigma_i(V)$:
\begin{equation*}
\mathsf{Z}(\mathsf{d}_{V,i}(u+\hbar d_i))\subset \sigma_i(V).
\end{equation*}
\item\label{cyc-R:2} The diagonal entry of $\mathsf{R}_{V,L_{\varpi_i}}(u)$ associated to $\Omega_{\mu,i}$ is 
\begin{equation*}
\left(\mathsf{R}_{V,L_{\varpi_i}}(u)\right)_{\Omega_{\mu,i},\Omega_{\mu,i}}=\frac{\mathcal{Q}_{i,V[\ul{\mu}]}^\g(u)}{\mathcal{Q}_{i,V[\ul{\mu}]}^\g(u-\hbar d_i)}. 
\end{equation*}
\end{enumerate}
\end{prop}
\begin{pf}
Let $\mathsf{C}_{V,i}(\g)$ be the subset of $\C$ consisting of all $s$ for which $V\otimes L_{\varpi_i}(s)$ is \textit{not} a highest weight module, as in Section \ref{ssec:cyclic-rmks}.
To prove Part \eqref{cyc-R:1}, we invoke Lemma \ref{L:cyc-R} and Remark \eqref{cyc-rmks:1} of Section \ref{ssec:cyclic-rmks}, which yield
\begin{equation*}
\mathsf{Z}(\mathsf{d}_{V,i}(u+\hbar d_i))=\mathsf{Z}(\mathsf{d}_{V,i}(u))-\hbar d_i\subset \mathsf{C}_{V,i}(\g)-\hbar d_i\subset \sigma_i(V).
\end{equation*}
Consider now Part \eqref{cyc-R:2}. By the same argument as given in the proof of \cite[Lemma~2.6]{frenkel-mukhin}, the diagonal entry of $\mathsf{R}_{V,L_{\varpi_i}}(u)$ associated to $\Omega_{\mu,i}$ coincides with the corresponding entry of $\mathsf{R}_{V,L_{\varpi_i}}^0(u)$, and to compute the latter we can assume without loss of generality that $\Omega_\mu$ is a common eigenvector of each $\xi_j(u)$.
Using Lemma \ref{L:A->Ai} and the defining equation \eqref{def:QiV^mu} for $\mathcal{Q}_{i,V[\ul{\mu}]}^\g(u)$, we find that the eigenvalue of $\bar{\mathsf{A}}_{V,L_{\varpi_i}}(u)$ on $\Omega_{\mu,i}$ is 
\begin{equation*}\label{A-Q}
\frac{\upnu_{i,V[\ul{\mu}]}^\g(u-\hbar d_i)}{\upnu_{i,V[\ul{\mu}]}^\g(u)}=\frac{\mathcal{Q}_{i,V[\ul{\mu}]}^\g(u)}{\mathcal{Q}_{i,V[\ul{\mu}]}^\g(u-\hbar d_i)}
\cdot
\frac{\mathcal{Q}_{i,V[\ul{\mu}]}^\g(u+2\kappa\hbar-\hbar d_i)}{\mathcal{Q}_{i,V[\ul{\mu}]}^\g(u+2\kappa \hbar)}.
\end{equation*}
Hence, we deduce from \eqref{R-A} that the eigenvalue of $\mathsf{R}_{V,L_{\varpi_i}}^0(u)$ on $\Omega_{\mu,i}$ is given by  $\mathcal{Q}_{i,V[\ul{\mu}]}^\g(u)\mathcal{Q}_{i,V[\ul{\mu}]}^\g(u-\hbar d_i)^{-1}$,  as claimed. \qedhere
\end{pf}
\begin{rem}\label{R:d-divides-Q}
Since each polynomial $\mathcal{Q}_{i,V[\ul{\mu}]}^\g(u)$ divides $\mathcal{Q}_{i,V}^\g(u)$, Part \eqref{cyc-R:2} of Proposition \ref{P:R-pole} implies that the diagonal entry of $\mathcal{Q}_{i,V}^\g(u-\hbar d_i)\mathsf{R}_{V,L_{\varpi_i}}(u)$ associated to any vector of the form $\Omega_{\mu,i}$ is a polynomial:
\begin{equation*}
\left(\mathcal{Q}_{i,V}^\g(u-\hbar d_i)\mathsf{R}_{V,L_{\varpi_i}}(u)\right)_{\Omega_{\mu,i},\Omega_{\mu,i}}\in \C[u].
\end{equation*}
Both this observation and the  inclusion $\zeroes(\mathsf{d}_{V,i}(u+\hbar d_i))\subset \zeroes(\mathcal{Q}_{i,V}^\g(u))$ provided by Part \eqref{cyc-R:1}  support the hypothesis that the denominator $\mathsf{d}_{V,i}(u)$ divides $\mathcal{Q}_{i,V}^\g(u-\hbar d_i)$. Note that this is necessarily the case when $\mathsf{d}_{V,i}(u)$ has no multiple roots. 
\end{rem}
\subsection{Fundamental modules}\label{ssec:R-fun}
%
Despite the above results, the equality $\mathsf{d}_{V,i}(u)=\mathcal{Q}_{i,V}^\g(u-\hbar d_i)$ does not hold for arbitrary $\g$, $V$ and $i\in \bfI$. In this section, we will see that this is already the case when $V$ is a fundamental module $L_{\varpi_j}$ for some $j\in \bfI$ satisfying $d_j>d_i$, which produces a family of counterexamples whenever $\g$ is \textit{not} simply laced; see Remark \ref{R:fail}. At the same time, we will apply some of the properties of the sets $\sigma_i(L_{\varpi_j})$ in order to recover a more precise description, well-known for $U_q(L\g)$, of the relationship between the polynomials
\begin{equation*}
\mathsf{d}_{ji}(u):=\mathsf{d}_{L_{\varpi_j},i}(u)
\end{equation*}
and the cyclicity and irreducibility of the tensor product $L_{\varpi_j}\otimes L_{\varpi_i}(s)$.
Given a fixed, ordered pair of indices $(i,j)\in \bfI\times \bfI$, let us define $\bar\imath=\bar\imath(i,j)\in \bfI$ and $\bar\jmath=\bar\jmath(i,j)\in \bfI$ by the formula 
\begin{equation*}
(\bar\imath,\bar\jmath)=
\begin{cases}
(i,j) \; &\text{ if }\; d_i\geq d_j,\\
(j,i)\;  &\text{ if }\; d_i< d_j.
\end{cases}
\end{equation*}
The $U_q(L\g)$-counterparts of Parts \eqref{fun-denom:1}--\eqref{fun-denom:3} of the following result were established in \cite{AK-qAffine} using two conjectures, which were later proven in \cite{chari-braid,frenkel-mukhin,kashiwara-level-zero,vv-standard}. In particular, the trigonometric analogues of Parts \eqref{fun-denom:2} and \eqref{fun-denom:3} are a special case of \cite[Prop.~9.4]{kashiwara-level-zero}. 
\begin{prop}\label{P:fun-denom} For each $i,j\in \bfI$, we have 
\begin{equation*}
\mathsf{Z}(\mathsf{d}_{ji}(u+\hbar d_i))\subset \sigma_i(L_{\varpi_j})\cap \sigma_j(L_{\varpi_i})=\sigma_{\bar\imath}(L_{\varpi_{\bar{\jmath}}}).
\end{equation*}
Moreover, the denominators $\mathsf{d}_{ij}(u)$ have the following properties:
\begin{enumerate}[font=\upshape]
\item\label{fun-denom:1} For each diagram automorphism $\omega$ of $\g$, we have 
\begin{equation*}
\mathsf{d}_{ij}(u)=\mathsf{d}_{\omega(i),\omega(j)}(u) \quad \text{ and }\quad \mathsf{d}_{ji}(u+\hbar d_i)=\mathsf{d}_{ij}(u+\hbar d_j).
\end{equation*}
\item\label{fun-denom:2} The module $L_{\varpi_j}\otimes L_{\varpi_i}(s)$ is highest weight if and only if $\mathsf{d}_{ji}(s)\neq 0$. 
\item\label{fun-denom:3} The module $L_{\varpi_j}\otimes L_{\varpi_i}(s)$ is simple if and only if $\mathsf{d}_{ji}(s)\neq 0$ and  $\mathsf{d}_{ij}(-s)\neq 0$.

\end{enumerate}
\end{prop}
\begin{pf}
The equality $\sigma_i(L_{\varpi_j})\cap \sigma_j(L_{\varpi_i})=\sigma_{\bar\imath}(L_{\varpi_{\bar{\jmath}}})$ is immediate from Corollary \ref{C:Qij-symm} and that $\sigma_i(L_{\varpi_j})=\zeroes(\mathcal{Q}_{ij}^\g(u))$ for all $i,j\in \bfI$. Moreover, by Part \eqref{cyc-R:1} of Proposition \ref{P:R-pole}, we have $\mathsf{Z}(\mathsf{d}_{ji}(u+\hbar d_i))\subset \sigma_i(L_{\varpi_j})$. Thus, the first assertion follows from the second identity of Part \eqref{fun-denom:1}. 

Our proof of the relations in Part \eqref{fun-denom:1} is based on the approach applied in \cite[\S A]{AK-qAffine} to establish their $U_q(L\g)$-analogues.  As it requires a non-trivial identification of $Y_{-\hbar}(\g)$ with $\Yhg$ and some additional facts concerning the universal $R$-matrix of $\Yhg$, we defer it to Appendix \ref{sec:App-B}; see Proposition \ref{P:cyc-symm}. 

%
Part \eqref{fun-denom:3} is a consequence of the first identity in  Part \eqref{fun-denom:1}, Part \eqref{fun-denom:2}, the identity  $(L_{\varpi_j}\otimes L_{\varpi_i}(s))^\ast\cong (L_{\varpi_{i^\ast}}\otimes L_{\varpi_{j^\ast}}(-s))(s-\kappa)$, and that a finite-dimensional $\Yhg$-module is simple if and only if it is cyclic and co-cyclic; see the proof of Corollary \ref{C:irr-cond}. 

Consider now Part \eqref{fun-denom:2}. We shall follow the proof of \cite[Cor.~2.2]{AK-qAffine}, where the role of Conjecture 1 therein is played instead by the following  general fact: if $s\notin \mathbf{Z}_{ij}:=\frac{\hbar}{2}\Z_{\geq 3d_i-d_j}$, then  $L_{\varpi_j}\otimes L_{\varpi_i}(s)$ is a highest weight module. This is a consequence of the inclusion $\sigma_i(L_{\varpi_j})+\hbar d_i\subset \mathbf{Z}_{ij}$ and Theorem \ref{T:cyclic}.
%
%

By Lemma \ref{L:cyc-R} it suffices to show that if $\mathsf{d}_{ji}(s)\neq 0$, then $L_{\varpi_j}\otimes L_{\varpi_i}(s)$ is a highest weight module. From the above general fact, we can assume that $s\in \mathbf{Z}_{ij}$. This assumption guarantees that $-s\notin \mathbf{Z}_{ji}$, and so the module $L_{\varpi_i}\otimes L_{\varpi_j}(-s)\cong (L_{\varpi_i}(s)\otimes L_{\varpi_j})(-s)$ is highest weight. By Lemma \ref{L:cyc-R}, $\mathsf{d}_{ij}(-s)\neq 0$. Since we also have $\mathsf{d}_{ji}(s)\neq 0$ by assumption, we obtain $\Yhg$-module homomorphisms
\begin{align*}
\bar{\mathsf{R}}_{ij}(-s):=(1\,2)\circ \mathsf{R}_{L_{\varpi_i},L_{\varpi_j}}(-s): L_{\varpi_i}(s)\otimes L_{\varpi_j}\to L_{\varpi_j}\otimes L_{\varpi_i}(s)\\
\bar{\mathsf{R}}_{ji}(s):=(1\,2)\circ \mathsf{R}_{L_{\varpi_j},L_{\varpi_i}}(s):L_{\varpi_j}\otimes L_{\varpi_i}(s)\to L_{\varpi_i}(s)\otimes L_{\varpi_j}
\end{align*}
As the endomorphism $\bar{\mathsf{R}}_{ji}(s)\circ \bar{\mathsf{R}}_{ij}(-s)$ of the cyclic module $L_{\varpi_i}(s)\otimes L_{\varpi_j}$ fixes the tensor product of highest weight vectors $v_i\otimes v_j$, it is the identity map. Hence, $\bar{\mathsf{R}}_{ij}(-s)$ is invertible and $L_{\varpi_j}\otimes L_{\varpi_i}(s)$ is isomorphic to the highest weight module $L_{\varpi_i}(s)\otimes L_{\varpi_j}$. \qedhere 
\end{pf}
\begin{rem}\label{R:fail}
If $d_j>d_i$, then $\sigma_{\bar{\imath}}(L_{\varpi_{\bar{\jmath}}})\subsetneq \sigma_i(L_{\varpi_j})$ (see Remark \eqref{cyc-rmks:2} of Section \ref{ssec:cyclic-rmks}), and hence it follows from the first assertion of the proposition that $\mathsf{Z}(\mathsf{d}_{ji}(u+\hbar d_i))\subsetneq \sigma_i(L_{\varpi_j})$ and $\mathsf{d}_{ji}(u)\neq \mathcal{Q}_{ij}^\g(u-\hbar d_i)=\mathcal{Q}_{i,L_{\varpi_j}}^\g(u-\hbar d_i)$. 
\end{rem}
\subsection{} \label{ssec:conjecture}

Though any further rigorous study of the polynomials $\mathsf{d}_{ji}(u)$ lay outside the scope of the present article, we shall conclude this section with a conjectural description of the specialized Baxter polynomials $\mathcal{Q}_{ij}^\g(u)$ in terms of such denominators which is consistent with Proposition \ref{P:fun-denom}. To this end, recall from \eqref{def:Qij} that the polynomial $\mathcal{Q}_{ij}^\g(u)$ admits the uniform formula
\begin{equation*}
\mathcal{Q}_{ij}^\g(u)
= \prod_{s=d_i}^{2\kappa-d_i}\left(u-(s-d_j)\frac{\hbar}{2}\right)^{\!v_{ij}^{(s)}}.
\end{equation*}
 The trigonometric counterpart of this formula has recently been used in \cite{fujita2020,fujita-oh} to uniformly describe the $U_q(L\g)$-analogue of the denominator $\mathsf{d}_{V,j}(u)$ in the case where $V$ is a certain Kirillov--Reshetikhin module of $U_q(L\g)$ determined by the index $i\in \bfI$; see Theorem 2.10 of \cite{fujita2020}, in addition to Proposition 5.5, Conjecture 5.7 and Theorem 5.9 of \cite{fujita-oh}. 
 
 To make this more precise, recall that, for each $i\in \bfI$ and $\ell>0$,  $L_{\ell \varpi_i}$ is the Kirillov--Reshetikhin module introduced at the beginning of Section \ref{ssec:Qij}. In addition,  we set 
 $r_i=m/d_i$ for each $i\in \bfI$, where $m$ is one half the square length of a long root, as in Section \ref{ssec:CPthm}.
By applying the transformation $z\mapsto q^{2z/\hbar}$ to the roots of $\mathcal{Q}_{ij}^\g(u-\hbar d_j)$, one obtains the polynomial
\begin{equation*}
\mathscr{Q}_{ij}^\g(uq^{-2d_j})=\prod_{s=d_i}^{2\kappa-d_i}\left(u-q^{s+d_j}\right)^{\!v_{ij}^{(s)}}=\prod_{s=0}^{2\kappa}\left(u-q^{s+d_j}\right)^{\!v_{ij}^{(s)}},
\end{equation*}
where we have employed Part \eqref{vij:4} of Corollary \ref{C:vij}. 
By Proposition 6.5 of \cite{fujita-oh}, this computes exactly the $U_q(L\g)$-counterpart of the denominator
 $\mathsf{d}_{L_{r_i\varpi_i},j}(u)$ which, in the notation\footnote{The factor of $q^{m-d_j}$ is due to different conventions on labeling finite-dimensional irreducible modules; the $U_q(L\g)$-modules $V_{r_i}^{(i)}$ and $V_1^{(j)}$ correspond to the $\Yhg$-modules $L_{r_i\varpi_i}(m\hbar/2)$ and $L_{\varpi_j}(d_j\hbar/2)$, respectively.} given therein, is $d_{i^{r_i},j}\left(q^{m-d_j}u\right)=d_{V_{r_i}^{(i)}(q^{-m}),V_1^{(j)}(q^{-d_j})}(u)$, up to renormalization by a power of $q$.  This observation is the primary motivation behind the first part of the below conjecture. In exactly the same way, one arrives at the second assertion using the formulas of Conjecture 6.7 and Theorem 6.9 in \cite{fujita-oh}.

\begin{conj}\label{conj:dij}
For each $i,j\in \bfI$, the polynomial $\mathcal{Q}_{ij}^\g(u)$ coincides with the denominator of the rational $R$-matrix $\mathsf{R}_{L_{r_i\varpi_i},L_{\varpi_j}}(u)$:
\begin{equation*}
\mathcal{Q}_{ij}^\g(u)=\mathsf{d}_{L_{r_i\varpi_i},j}(u+\hbar d_j).
\end{equation*}
Moreover, one has the identities 
\begin{equation*}
\mathsf{d}_{ji}(u+\hbar d_i)=\mathcal{Q}_{\bar\imath\bar\jmath}^\g(u)\quad \text{ and }\quad \zeroes(\mathsf{d}_{ji}(u+\hbar d_i))=\sigma_{\bar\imath}(L_{\varpi_{\bar\jmath}})
\end{equation*}
unless $d_i=1=d_j$ and $\g$ is of type\footnote{As we follow the Bourbaki convention for the labels of Dynkin diagrams \cite{bourbaki-lie456}, our $\mathsf{B}_r$ is $\mathsf{C}_r$ in \cite[Conj.~6.7]{fujita-oh}.} $\mathsf{B}_r$, $\mathsf{F}_4$ or $\mathsf{G}_2$.
\end{conj}

Note that when $\alpha_i$ is a long root, Part \eqref{fun-denom:1} of Proposition \ref{P:fun-denom} implies that the first assertion of the conjecture is equivalent to $\mathsf{d}_{ji}(u+\hbar d_i)=\mathcal{Q}_{ij}^\g(u)$, which also equals $\mathcal{Q}_{\bar\imath\bar\jmath}^\g(u)$. Hence, in this case the second assertion of the conjecture follows from the first. 
Similary, if $\alpha_j$ is a long root,  Part \eqref{Qij-symm:3} of Corollary \ref{C:Qij-symm}, together with the observation that $r_i=d_j/d_i=r_{ij}$, implies that the first assertion of the conjecture is instead equivalent to
\begin{equation*}
\mathcal{Q}_{j,L_{r_i\varpi_i}}^\g(u)=\mathsf{d}_{L_{r_i\varpi_i},j}(u+\hbar d_j),
\end{equation*}
which is consistent with the inclusion $\zeroes(\mathsf{d}_{L_{r_i\varpi_i},j}(u+\hbar d_j))\subset \zeroes(\mathcal{Q}_{j,L_{r_i\varpi_i}}^\g(u))$ provided by Proposition \ref{P:R-pole}. 
Note that the exceptions to the second statement of the conjecture only arise in the case where $\alpha_i$ and $\alpha_j$ are both short roots and $\g$ is not simply laced. 
When $\g$ is simply laced, the conjecture reduces to the assertion that
\begin{equation*} 
\mathsf{d}_{ji}(u)
= \prod_{s=1}^{h^{\scriptscriptstyle{\vee}}-1}\left(u-(s+1)\frac{\hbar}{2}\right)^{\!v_{ji}^{(s)}}.
\end{equation*}
 The trigonometric version of this equality was established in Theorem 2.10 of \cite{fujita2020} using the geometry of graded quiver varieties. This result inspired the generalizations given in \cite{fujita-oh}, referred to above, and unified several results appearing independently in \cite{AK-qAffine, DaOk94} (type $\mathsf{A}$), \cite{kkk} (type $\mathsf{D}$) and \cite{OhScr19} (type $\mathsf{E}$).
 
 Finally we note that the main obstruction to applying the results of \cite{fujita2020,fujita-oh} in conjunction with those of \cite{sachin-valerio-2,sachin-valerio-III} to indirectly deduce that Conjecture \ref{conj:dij} is a theorem is the fact that it has not been established that the functor relating finite-dimensional representations $U_q(L\g)$ and $\Yhg$ constructed in \cite{sachin-valerio-2} is compatible with the denominators of the normalized $R$-matrices ouput by the representation theories of $U_q(L\g)$ and $\Yhg$.
\section{Representations of \texorpdfstring{$\mathrm{D}\Yhg$}{DY(g)}}\label{sec:yangian-Dbl}

We now turn our attention to the Yangian double $\mathrm{D}\Yhg$, which first appeared in the work \cite{khoroshkin-tolstoy} of Khoroshkin and Tolstoy. Our first main goal is to prove Proposition \ref{pr:DYg-Yg}, which asserts that the category of finite-dimensional representations of $\mathrm{D}Y_\hbar(\g)$ is isomorphic to the full subcategory of $\Ryang$ consisting of all $V$ for which $\sigma(V)\subset \C^\times$.
Our second main goal is to apply this result and Theorem \ref{thm:Comb} to classify the finite-dimensional irreducible representation of $\mathrm{D}\Yhg$. This goal will be realized in Theorem \ref{T:DYg-class}. Along the way, we will obtain an equivalent characterization of the $i$-th set of poles $\sigma_i(V)$ of any $V\in \Ryang$ in Corollary \ref{C:DYgi-pole}, and deduce several additional properties for $\mathrm{Rep}_{fd}(\mathrm{D}\Yhg)$ in Corollary \ref{C:RepDYg}.

\subsection{The Yangian double \texorpdfstring{$\mathrm{D}\Yhg$}{DYg}}\label{ssec:DYhg-def}

The Yangian double $\mathrm{D}\Yhg$ is defined to be the unital associative algebra over $\C$, generated
by  $\{\xi_{i,r},x_{i,r}^{\pm}\}_{i\in\bfI,r\in\Z}$, subject to the relations \ref{Y1}--\ref{Y6} of $\Yhg$ with the second index of each generator taking values in $\Z$.

This definition is such that the assignment $x_{i,r}^\pm \mapsto x_{i,r}^\pm$ and $\xi_{i,r}\mapsto \xi_{i,r}$, for each $i\in \bfI$ and $r\in \N$,  extends to an algebra homomorphism 
\begin{equation*}
\imath:\Yhg\to \mathrm{D}\Yhg.
\end{equation*}
As a consequence of Proposition \ref{pr:formalshift} below, this homomorphism is injective, as suggested by our notation for  generators of $\mathrm{D}\Yhg$. The defining relations of $\mathrm{D}\Yhg$ can be expressed equivalently in terms of generating series as follows. For each $\in \bfI$, introduce $\xi_i^\pm(u)\in \mathrm{D}\Yhg[\![u^{\mp 1}]\!]$ and $\X_i^\pm(u)\in \mathrm{D}\Yhg[\![u,u^{-1}]\!]$ by 
\begin{gather*}
\xi_i^+(u)=1+\hbar\sum_{r\geq 0} \xi_{i,r} u^{-r-1}, \quad \xi_i^-(u)=1-\hbar\sum_{r<0} \xi_{i,r} u^{-r-1}\\
\X_i^\pm(u)=\hbar\sum_{r\in \Z}x_{i,r}^\pm u^{-r-1}
\end{gather*}
and let $\delta(u)=\sum_{n\in \Z}u^n\in \C[\![u^{\pm 1}]\!]$ denote the formal delta function. The following proposition is a straightforward consequence of the defining relations of $\mathrm{D}\Yhg$; see \cite[Prop.~2.3]{sachin-valerio-2} and \cite[Rem.~2.6]{Curtis-DYg}, for instance.  
\begin{prop}\label{pr:DYg-rel}
The defining relations of $\mathrm{D}\Yhg$ are equivalent to the following formal series identities:
\begin{enumerate}[label=(D\arabic*), font=\upshape]
\item\label{DY1} For each $i,j\in \bfI$, we have
\begin{equation*}
[\xi_i^\pm(u),\xi_j^\pm(v)]=0=[\xi_i^+(u),\xi_j^-(v)].
\end{equation*}

%
\item\label{DY3} For each $i,j\in \bfI$, $\epsilon\in \{\pm\}$ and $a=\hbar d_i a_{ij}/2$, we have 
\begin{equation*}
(u-v\mp a)\xi_i^\epsilon(u)\X_j^\pm(v)=(u-v\pm a)\X_j^\pm(v)\xi_i^\epsilon(u).
\end{equation*}

\item\label{DY4} For each $i,j\in \bfI$ and $a=\hbar d_i a_{ij}/2$, we have 
\begin{equation*}
(u-v\mp a)\X_i^\pm(u)\X_j^\pm(v)=(u-v\pm a)\X_j^\pm(v)\X_i^\pm(u).
\end{equation*}

\item\label{DY5} For each $i,j\in \bfI$, we have 
\begin{equation*}
[\X_i^+(u),\X_j^-(v)]=\hbar \delta_{ij} u^{-1}\delta(u/v)(\xi_i^+(v)-\xi_i^-(v)).
\end{equation*}
\item\label{DY6} For each $i\neq j\in \bfI$ and $m=1-a_{ij}$, we have 
\begin{equation*}
\sum_{\pi \in \mathfrak{S}_m}\lb\X_i^\pm(u_{\pi(1)}),\lb\X_i^\pm(u_{\pi(2)}),\lb\cdots\lb\X_i^\pm(u_{\pi(m)}), \X_j^\pm(v)\rb\cdots \rb\rb\rb=0.
\end{equation*}
\end{enumerate}
\end{prop}
\noindent Here the relations \ref{DY1}--\ref{DY5} hold in the formal series space $\mathrm{D}\Yhg[\![u^{\pm 1},v^{\pm 1}]\!]$, while \ref{DY6} should be understood as a relation in $\mathrm{D}\Yhg[\![u_1^{\pm 1},\ldots,u_m^{\pm 1},v^{\pm 1}]\!]$.

\subsection{}

We shall further decompose each generating series $\X_i^\pm(u)$ as a difference $\X_i^\pm(u)=\X_i^\pm(u)_+-\X_i^\pm(u)_-$, where $\X_i^\pm(u)_{\epsilon}\in \mathrm{D}\Yhg[\![u^{-\epsilon}]\!]$ are defined by
\begin{equation*}
\X_i^\pm(u)_+=\hbar\sum_{r\geq 0}x_{i,r}^\pm u^{-r-1}\quad \text{ and }\quad \X_i^\pm(u)_-=-\hbar\sum_{r<0}x_{i,r}^\pm u^{-r-1} \quad \forall \; i \in \bf{I}. 
\end{equation*}
In particular, one has $\imath(\xi_i(u))=\xi_i^+(u)$ and $\imath(x_i^\pm(u))=\X_i^\pm(u)_+$ for all $i\in \bfI$. Proposition \ref{pr:DYg-rel} then admits the following corollary.
\begin{cor}\label{C:fX=X0}
For each $i\in \bfI$, define the operator 
\begin{equation*}
\mathrm{T}_i(u):=\mathrm{Ad}(\xi_i^+(u))^{-1}-\id: \mathrm{D}\Yhg\to u^{-1}\mathrm{D}\Yhg[\![u^{-1}]\!].
\end{equation*}
Then, for every $i,j\in \bfI$ and $\epsilon\in \{\pm\}$,  we have
\begin{equation}\label{fX=X0}
\left(2a\mp (u-v\pm a)\mathrm{T}_i(u)\right)\X_j^\pm(v)_\epsilon= \mp\mathrm{T}_i(u)\cdot \hbar x_{j,0}^\pm
\end{equation}
in $\mathrm{D}\Yhg[\![v^{-\epsilon}, u^{-1}]\!]$, where $a=\hbar d_i a_{ij}/2$. 
\end{cor}
\begin{pf} From the relation \ref{DY3} with $\epsilon=+$, we obtain the identity 
\begin{align*}
(u-v\pm a) &\X_j^\pm(v)_+\xi_i^+(u)-(u-v\mp a)\xi_i^+(u) \X_j^\pm(v)_+\\
=
&(u-v\pm a) \X_j^\pm(v)_-\xi_i^+(u)-(u-v\mp a)\xi_i^+(u) \X_j^\pm(v)_-.
\end{align*}
The left-hand side belongs to $\mathrm{D}\Yhg[\![v^{-1},u^{-1}]\!]$ and the right-hand side belongs to $\mathrm{D}\Yhg[\![v,u^{-1}]\!]$. Hence, the left-hand side (resp. right-hand side) is equal to the constant term of the right-hand side (resp. left-hand side) with respect to $v$. Since these two constant terms themselves coincide and equal $-\hbar[x_{j,0}^\pm,\xi_i^+(u)]$, we obtain
\begin{equation*}
(u-v\pm a) \X_j^\pm(v)_\epsilon\xi_i^+(u)-(u-v\mp a)\xi_i^+(u) \X_j^\pm(v)_\epsilon
=-\hbar[x_{j,0}^\pm,\xi_i^+(u)].
\end{equation*}
Left multiplying by $\xi_i^+(u)^{-1}$ then yields the relation \eqref{fX=X0}.  \qedhere
\end{pf}
Note that, since the substitution $u\mapsto v\mp a$ yields an algebra homomorphism $\mathrm{D}\Yhg[\![v^{-1},u^{-1}]\!]\to \mathrm{D}\Yhg[\![v^{-1}]\!]$, the relation \eqref{fX=X0} with $j=i$ implies that
\begin{equation}\label{fX=X0'}
\X_i^\pm(v)_+=\mp (2d_i)^{-1} \mathrm{T}_i(v\mp \hbar d_i)\cdot  x_{i,0}^\pm \quad \forall\; i\in \bfI.
\end{equation}

\subsection{The formal shift operator}\label{ssec:FS-PBW}
%
Let $\tau_z:\Yhg\to\Yhg[z]$ be the algebra embedding obtained by replacing $a\in \C$ by a formal variable $z$ in the definition of the shift automorphism $\tau_a$, given in Section \ref{ssec:shift}. Let $\Yhg[z;z^{-1}]\!]$ denote the algebra of formal Laurent series in $z^{-1}$ with coefficients in $\Yhg$. The following proposition is a consequence of Theorem 4.3 and Corollary 4.6 from \cite{Curtis-DYg}.
\begin{prop}\label{pr:formalshift}
There is a unique algebra homomorphism 
\begin{equation*}
\Phi_z:\mathrm{D}\Yhg\to \Yhg[z;z^{-1}]\!]
\end{equation*}
satisfying $\Phi_z\circ \imath =\tau_z$. It is determined by the formulae 
\begin{equation*}
\Phi_z(\xi^-_i(u))=\exp(-u\partial_z)\xi_i(-z), \quad \Phi_z(\X_i^\pm(u)_-)=\exp(-u\partial_z)x_i^\pm(-z)
\quad \forall \; i \in \bfI.
\end{equation*}
\end{prop}
As indicated in Section \ref{ssec:DYhg-def}, this result implies that the natural homomorphism 
$\imath:\Yhg\to \mathrm{D}\Yhg$ is injective. Namely, $\tau_z$ is injective and one has $\tau_z=\Phi_z\circ \imath$.

\subsection{Representations of \texorpdfstring{$\mathrm{D}\Yhg$}{DYg}}\label{ssec:DYhg-rep}

Given $a\in \C$, let $\mathrm{Rep}_{fd}^a(\Yhg)$ denote the full subcategory of the category  $\mathrm{Rep}_{fd}(\Yhg)$ of finite-dimensional $\Yhg$-modules consisting of representations $V$ whose full set of poles
\begin{equation*}
\sigma(V)=\bigcup_{i\in \bfI} \sigma_i(V)
\end{equation*}
satisfies $\sigma(V)\subset \C\setminus\{a\}$. These categories are essentially independent of the choice of $a\in \C$ since, for any $b\in \C$, the pull-back functor $\tau_{b-a}^\ast$ defines an isomorphism of categories 
\begin{equation*}
\tau_{a-b}^\ast:\mathrm{Rep}_{fd}^a(\Yhg)\longisom \mathrm{Rep}_{fd}^b(\Yhg).
\end{equation*}
The following proposition provides the main result of this section. 
\begin{prop}\label{pr:DYg-Yg}
 Let $V\in \mathrm{Rep}_{fd}^0(\Yhg)$. Then the $\Yhg$-action on $V$ extends to a $\mathrm{D}\Yhg$-action, uniquely determined by the property that, for each $i\in \bfI$, the series 
\begin{equation*}
\xi_i^-(u)\in \End(V)[\![u]\!] \quad \text{ and } \quad \X_i^\pm(u)_-\in \End(V)[\![u]\!]
\end{equation*}
are the expansions at $0$ of the rational functions $\xi_i(u)$ and $x_i^\pm(u)$, respectively.
Moreover, every finite-dimensional $\mathrm{D}\Yhg$-module arises in this way. 
\end{prop}
\begin{pf}
Let $\pi_V:\Yhg\to \End(V)$ be the underlying representation morphism. By Proposition \ref{pr:rational}, the assumption $\sigma(V)\subset \C^\times$,  and the formulae of Proposition \ref{pr:formalshift}, the composite $\pi_V\circ \Phi_z:\mathrm{D}\Yhg\to\End(V)[z;z^{-1}]\!]$ satisfies 
\begin{equation*}
(\pi_V\circ \Phi_z)(\mathrm{D}\Yhg)\subset \End(V)\otimes \C[z]_{z},
\end{equation*}
where $\C[z]_z\subset \C(\!(z^{-1})\!)$ is the localization of $\C[z]$ at the maximal ideal $z\C[z]$.
 Letting $\mathrm{ev}:\C[z]_z\to \C$ denote the evaluation homomorphism $f(z)\mapsto f(0)$, we obtain a $\mathrm{D}\Yhg$-module structure on $V$ given by the algebra homomorphism 
\begin{equation*}
\Gamma_{\!V}:=(\id\otimes \mathrm{ev}) \circ \pi_V\circ \Phi_z: \mathrm{D}\Yhg\to \End(V).
\end{equation*}
Since $\Phi_z\circ \imath=\tau_z$ evaluates to the identity $\id_{\Yhg}$ at $z=0$, we indeed have $\Gamma_{\!V}\circ \imath=\pi_V$. Moreover, the formulae of Proposition \ref{pr:formalshift} show that, for each $i\in \bf{I}$, $\xi_i^-(u)$ and $\X_i^\pm(u)_-$ operate as the Taylor expansions
\begin{equation*}
\sum_{n\geq 0}\left.\frac{\partial_v^{n}\xi_i(v)u^n }{n!}\right|_{v=0}\quad \text{ and }\quad \sum_{n\geq 0}\left.\frac{\partial_v^{n}x_i^\pm(v)u^n }{n!}\right|_{v=0},
\end{equation*}
respectively, of the rational functions $\xi_i(u)$ and $x_i^\pm(u)$ at $0$. This completes the proof of the first part of the proposition. 

Let us now turn to the second assertion. Let $V$ be an arbitrary finite-dimensional $\mathrm{D}\Yhg$-module and fix $i\in \bfI$. By virtue of Proposition \ref{pr:rational},  the series 
\begin{equation*}
\xi_i^+(u)\in \End(V)[\![u^{-1}]\!] \quad \text{ and }\quad \X_i^\pm(u)_+\in \End(V)[\![u^{-1}]\!]
\end{equation*}
are the expansions at $\infty$ of rational functions of $u$, which we again denote by $\xi_i(u)$ and $x_i^\pm(u)$, respectively. To complete the proof, it is enough to show that $\xi_i^-(u)$ and $\X_i^\pm(u)_-$ are the expansions of these same rational functions at $0$.
 Indeed, this will imply that the $\Yhg$-module $\imath^\ast(V)$ belongs to 
$\mathrm{Rep}_{fd}^0(\Yhg)$ and that $V$ is equal to $\mathrm{D}\Yhg$-module obtained by extending the $\Yhg$-action on $\imath^\ast(V)$ as in the first part of the proposition.

Let $\mathrm{T}_i(u)=\mathrm{Ad}(\xi_i^+(u))^{-1}-\id$, as in Corollary \ref{C:fX=X0}. 
By the rationality of $\xi_i^+(u)$, there is a nonzero polynomial $P(u)\in \C[u]$ such that $\mathrm{T}_i^P(u):=P(u)\mathrm{T}_i(u)$ satisfies
\begin{equation*}
\mathrm{T}_i^P(u)\cdot \mathsf{X}\in \End(V)[u] \quad \forall\quad \mathsf{X}\in \End(V).
\end{equation*}
Left multiplying \eqref{fX=X0} by $P(u)$ and setting $j=i$ and $\epsilon=-$, we obtain the identity  
\begin{equation*}
\left(2\hbar d_i P(u)\mp (u-v\pm \hbar d_i)\mathrm{T}_i^P(u)\right)\X_i^\pm(v)_-= \mp\mathrm{T}_i^P(u)\cdot \hbar x_{i,0}^\pm
\end{equation*}
in $\End(V)[u][\![v]\!]$. Applying the automorphism $u\mapsto u\mp \hbar d_i$ of $\End(V)[u]$ followed by 
the homomorphism $\End(V)[u][\![v]\!]\to \End(V)[\![v]\!]$ determined by the evaluation $u\mapsto v$, we arrive at relation
\begin{equation*}
P(v\mp \hbar d_i)\X_i^\pm(v)_-= \mp (2 d_i)^{-1}\mathrm{T}_i^P(v\mp \hbar d_i)\cdot x_{i,0}^\pm
=P(v\mp \hbar d_i)x_i^\pm(v)
\end{equation*}
 in $\End(V)[v]\subset \End(V)[\![v]\!]$, where the second equality is due to the relation \eqref{fX=X0'}. This implies that $\X_i^\pm(v)_-$ is the expansion of the rational function $x_i^\pm(v)$ at $0$, as desired. Moreover, the relation \ref{DY5} of Proposition \ref{pr:DYg-Yg} implies that 
\begin{equation*}
 \xi_i^\pm(u)=1+[\X_i^+(u)_\pm,x_{i,0}^-],
\end{equation*}
from which it follows immediately that the series $\xi_i^-(u)$ is the expansion at $0$ of the rational function $\xi_i(u)$. \qedhere
\end{pf}
\subsection{}\label{ssec:sigmai-dbl}
As a first consequence of Proposition \ref{pr:DYg-Yg}, we obtain the following equivalent characterization of the $i$-th set of poles $\sigma_i(V)$ for any $V\in \Ryang$ and $i\in \bfI$. 
\begin{cor}\label{C:DYgi-pole}
Let $V\in \Ryang$ and fix $i\in \bfI$. Then $a\in \sigma_i(V)$ if and only if the $Y_{d_i\hbar}(\sl_2)$-action on $\varphi_i^\ast(V(-a))$ fails to extend to a $\mathrm{D}Y_{d_i \hbar}(\sl_2)$-action. That is,  
\begin{equation*}
\sigma_i(V)=\{a\in \C: \nexists\; \Psi_i:\mathrm{D}Y_{d_i\hbar}(\sl_2)\to \End V \; \text{ with }\; \Psi_i\circ \imath=\pi \circ \tau_{-a} \circ \varphi_i\}
\end{equation*}
where $\pi:\Yhg\to \End V$ is the action homomorphism, and $\Psi_i$ is required to be an algebra homomorphism.
\end{cor}
\begin{pf}
By Proposition \ref{pr:DYg-Yg}, the $Y_{d_i\hbar}(\sl_2)$-action on $\varphi_i^\ast(V(-a))$ fails  to extend to an action of $\mathrm{D}Y_{d_i\hbar}(\sl_2)$ if and only if $0\in\sigma(\varphi_i^\ast(V(-a)))=\sigma_i(V(-a))$. Since $\sigma_i(V(-a))=\sigma_i(V)-\{a\}$, this occurs precisely when $a\in \sigma_i(V)$. 
\end{pf}

\subsection{}
%
%
In categorical terms, Proposition \ref{pr:DYg-Yg} outputs an isomorphism of categories
\begin{equation*}
\Gamma: \mathrm{Rep}_{fd}^0(\Yhg)\longisom \mathrm{Rep}_{fd}(\mathrm{D}\Yhg)
\end{equation*}
which commutes with the forgetful functor to vector spaces and has inverse given by the pull-back functor $\imath^\ast$. Three consequences of this interpretation relevant to our discussion are given by the following corollary. 
\begin{cor}\label{C:RepDYg}
\leavevmode 
\begin{enumerate}[font=\upshape]
\item\label{RepDYg-tens} There is a unique tensor structure on $\mathrm{Rep}_{fd}(\mathrm{D}\Yhg)$ such that $\Gamma$ is a strict tensor functor. The tensor product $\dot\otimes$ is  given by 
\begin{equation*}
V\dot\otimes W=\Gamma(\imath^\ast(V)\otimes \imath^\ast(W)).
\end{equation*}
\item\label{RepDYg-irr} A $\mathrm{D}\Yhg$-module $V\in \mathrm{Rep}_{fd}(\mathrm{D}\Yhg)$ is irreducible if and only if there is an $\bfI$-tuple of Drinfeld polynomials $\ul{P}=(P_i(u))_{i\in \bfI}$ such that 
\begin{equation*}
V\cong \Gamma(L(\ul{P})) \quad \text{ where }\quad \sigma(L(\ul{P}))\subset \nC.
\end{equation*}
\item\label{RepDYg-com} Let $V\in \mathrm{Rep}_{fd}(\Yhg)$, with composition factors $\{V_1,\ldots,V_N\}$.
 Then the $\Yhg$-action on $V$ extends to a $\mathrm{D}\Yhg$-action if and only if 
\begin{equation*}
\sigma(V_j)\subset \C^\times \quad \forall\quad  1\leq j\leq N.
\end{equation*}
\end{enumerate}
\end{cor}
\begin{pf}
 The first assertion of Corollary \ref{cor:sigma} implies that $\mathrm{Rep}_{fd}^0(\Yhg)$ is a tensor subcategory of $\mathrm{Rep}_{fd}(\Yhg)$. This result, together with Proposition \ref{pr:DYg-Yg}, implies Part \eqref{RepDYg-tens} of the corollary. 
 
Similarly, Parts \eqref{RepDYg-irr} and \eqref{RepDYg-com} of the corollary follow from Proposition \ref{pr:DYg-Yg}, Theorem \ref{thm:dp} and the second assertion of Corollary \ref{cor:sigma}. \qedhere
\end{pf}

\subsection{Classification of simple modules}\label{ssec:DYg-class}
We now apply Theorem \ref{thm:Comb} and Corollary \ref{C:RepDYg} to obtain a concrete classification of finite-dimensional irreducible $\mathrm{D}\Yhg$-modules. 
Recall from Section \ref{ssec:poles-compute} that the formal series  
\begin{equation*}
v_{ij}(q)=\sum_{r\geq d_i} v_{ij}^{(r)}q^r\in q^{d_i}\Z[\![q]\!]
\end{equation*}
arise as the entries of the matrix $\mathbf{E}(q)=([a_{ij}]_{q^{d_i}})_{i,j\in \bfI}^{-1}$.
Moreover, by Theorem \ref{thm:Comb}, the full set of poles $\sigma(L_{\varpi_j})$ of the $j$-th fundamental representation $L_{\varpi_j}$ is given explicitly by
\begin{equation*}
\sigma(L_{\varpi_j})=\hbar\left\{\frac{r-d_j}{2}:1\leq r \leq 2\kappa-1\; \text{ and }\; \sum_{i\in \bfI} v_{ij}^{(r)}> 0\right\}, 
\end{equation*}
where it is understood that $r$ takes integer values. With this description in mind, the following theorem yields the desired classification result. 
\begin{thm}\label{T:DYg-class}
The isomorphism classes of finite-dimensional irreducible representations of $\mathrm{D}\Yhg$ are parametrized by $\bfI$-tuples of monic polynomials $\ul{P}=(P_i(u))_{i\in \bfI}$ satisfying the condition
\begin{equation*}
\zeroes(P_j(-u))\subset \C\setminus \sigma(L_{\varpi_j}) \quad \forall\; j\in \bfI. 
\end{equation*}
Moreover, if $\g$ is simply laced then $\sigma(L_{\varpi_j})$ is independent of $j$ and given by 
\begin{equation*}
\sigma(L_{\varpi_j})=\hbar\left\{\frac{k}{2}:k\in \Z\;\text{ with }\;  0\leq k\leq h^\vee-2\right\}. 
\end{equation*}
\end{thm}
\begin{pf}
Part \eqref{RepDYg-irr} of Corollary \ref{C:RepDYg} implies that 
the isomorphism classes of finite-dimensional irreducible representations of $\mathrm{D}\Yhg$ are parametrized by those $\bfI$-tuples of monic polynomials $\ul{P}=(P_i(u))_{i\in \bfI}$ for which the condition
\begin{equation*}
\sigma(L(\ul{P}))\subset \C^\times 
\end{equation*} 
is satisfied. By Theorem \ref{thm:Comb}, the full set of poles $\sigma(L(\ul{P}))$ is given explicitly by
\begin{equation*}
\sigma(L(\ul{P}))=\bigcup_{j\in \bfI}\left(\zeroes(P_j(u))+\sigma(L_{\varpi_j})\right),
\end{equation*}
and hence is contained in $\C^\times$ if and only if  $\zeroes(P_j(-u))\subset \C\setminus \sigma(L_{\varpi_j})$ for all $j\in \bfI$. This proves the first part of the theorem. 

Consider now the second assertion. When $\g$ is simply laced, we have $d_i=1$ for all $i\in \bfI$ and  $2\kappa$ coincides with the dual Coxeter number $h^\vee$; see Section \ref{ssec:CPthm}. If $\g=\sl_n$, then $h^\vee=n$ and the stated formula for $\sigma(L_{\varpi_j})$ follows easily from the description of $\sigma_i(L_{\varpi_j})$ given in Corollary \ref{C:sln-comb}:
\begin{equation*}
\sigma_i(L_{\varpi_j})=\hbar \left\{ \frac{i+j}{2}-b: b\in [i+j+1-n,i]\cap [1,j]\right\}.
\end{equation*}
If $j=1$, this collapses to $\sigma_i(L_{\varpi_1})=\frac{\hbar}{2}\{i-1\}$, and taking the union over $i\in \bfI=\{1,\ldots,n-1\}$ thus recovers the claimed formula for $\sigma(L_{\varpi_1})$. More generally, the above description of $\sigma_i(L_{\varpi_j})$ yields the recursion
\begin{equation*}
 \sigma_i(L_{\varpi_j}^{(n)})=
 \begin{cases}
 \sigma_{i-1}(L_{\varpi_{j-1}}^{(n-1)}) &\text{ if }\; i+j>n\\
  \sigma_{i-1}(L_{\varpi_{j-1}}^{(n-1)}) \cup \hbar \left\{\frac{i+j}{2}-1\right\}  \; &\text{ if }\; i+j\leq n
 \end{cases}
\end{equation*}
where $L_{\varpi_j}^{(n)}$ is the $\Yhsl{n}$-module $L_{\varpi_j}$ and  $\sigma_a(L_{\varpi_b}^{(k)})=\emptyset$ if $a$ or $b$ is zero. It follows readily that 
\begin{equation*}
\sigma(L_{\varpi_j}^{(n)})=\sigma(L_{\varpi_{j-1}}^{(n-1)})\cup \hbar\left\{\frac{n}{2}-1\right\},
\end{equation*}
from which the claimed identity for $\sigma(L_{\varpi_j})$ is deduced by induction on $j$.

We shall, however, give a uniform proof of the positivity of $\sum_{i\in \bfI} v_{ij}^{(r)}$, for each fixed $1\leq r\leq h^\vee-1$, in Proposition \ref{P:Fuj-Her} below. This applies to all simply laced types and, as a consequence of the description of $\sigma(L_{\varpi_j})$ provided above the statement of the theorem, yields the claimed description of $\sigma(L_{\varpi_j})$ in general. \qedhere
\end{pf}
\begin{rem}
Before proving Proposition \ref{P:Fuj-Her}, let us provide some additional context which sheds light on the nontriviality of Theorem \ref{T:DYg-class}. When $\g=\sl_2$, it implies that the finite-dimensional irreducible representations of $\mathrm{D}\Yhsl{2}$ are parametrized by monic polynomial $P(u)$ which have nonzero constant term. This agrees with the rank one instance of Conjecture 1 from \cite{Io96} for the formal, $\hbar$-adic, analogue of $\mathrm{D}\Yhsl{n}$, which asserts that such modules are parametrized by $\bfI$-tuples of monic polynomials $\ul{P}=(P_i(u))_{i\in\bfI}$ satisfying 
\begin{equation*}
\zeroes(P_i(u))\subset \C^\times \quad \forall \quad i\in \bfI. 
\end{equation*}
Theorem \ref{T:DYg-class} shows that this conjecture does not specialize well to the situation where $\hbar\in \C^\times$ beyond the rank one setting. Indeed, it demonstrates that the condition  $\zeroes(P_i(u))\subset \C^\times$ is not a sufficient condition for extending the $\Yhsl{n}$ action on $L(\ul{P})$ to a $\mathrm{D}\Yhsl{n}$ action for any $n\geq 3$. 
\end{rem}

\subsection{The \texorpdfstring{$q$}{q}-Cartan matrix revisited}\label{ssec:Fuj-Her}

Throughout this section, $\g$ is assumed to be simply laced. Recall that to complete the proof of Theorem \ref{T:DYg-class} we are left to show that, for each fixed $1\leq r\leq h^\vee-1$ and $j\in \bfI$, one has $\sum_{i\in\bfI}v_{ij}^{(r)}>0$. By Corollary \ref{C:vij}, one has  $v_{ij}^{(r)}\geq 0$ and $v_{ij}^{(r)}=v_{ji}^{(r)}$ for all $i\in \bfI$, and hence the desired result follows from Proposition \ref{P:Fuj-Her}, stated below. 

 Our proof of this result relies heavily on a remarkable combinatorial description of the numbers $v_{ij}^{(r)}$ obtained in in the work \cite{hernandez-leclerc2} of Hernandez and Leclerc. In Proposition 2.1 therein, it was shown that, given an orientation $\mathsf{Q}$ (\textit{i.e.}, a quiver structure) on the Dynkin diagram of $\g$ and a Coxeter element $\tau_\mathsf{Q}$ \textit{adapted} to this orientation, $v_{ij}^{(r)}$ can be expressed in terms of the action of $\tau_\mathsf{Q}$ on the set $R_+$ of positive roots  of $\g$.
However, we will only be concerned with a particular sink-source orientation, in which case the formula for $v_{ij}^{(r)}$ takes the form \eqref{Her-Lec-vij}.
\begin{prop}\label{P:Fuj-Her}
Let $\g$ be simply laced. Then for each $i\in \bfI$ and $1\leq r\leq h^\vee-1$, there is $j\in \bfI$ with $v_{ij}^{(r)}>0$. 
\end{prop}
\begin{pf}

Fix $i\in \bfI$ and $1\leq r\leq h^\vee-1$. Furthermore, equip the Dynkin diagram of $\g$ with the structure of a quiver with a sink-source orientation in which $i$ is a sink. This naturally partitions the set of vertices $\bf{I}$ as $\bf{I}=\dot{\bfI}\sqcup \ddot{\bfI}$, where 
\begin{equation*}
\dot{\bfI}=\{j\in \bfI:j \text{ is a sink}\} \quad \text{ and }\quad \ddot{\bfI}=\{j\in \bfI:j \text{ is a source}\}.
\end{equation*}
Let $\tau$ denote the Coxeter element $\tau=\ddot{w}\dot{w}$, where $\dot{w}=\prod_{j\in \dot{\bfI}}s_j$ and $\ddot{w}= \prod_{j\in \ddot{\bfI}}s_j$, and let $\gamma_i$ denote the positive root $\gamma_i=\alpha_i-\sum_{j\in \ddot{\bfI}}a_{ji}\alpha_j$. By \cite[Prop.~2.1]{hernandez-leclerc2}, we then have
\begin{equation}\label{Her-Lec-vij}
v_{ij}^{(r)}=
\begin{cases}
(\tau^{(r-1)/2}(\gamma_i),\varpi_j) \quad &\text{ if } j\in \dot{\bfI} \; \text{ and }\; r-1\in 2\Z,\\
(\tau^{(r-2)/2}(\gamma_i),\varpi_j) \quad &\text{ if } j\in \ddot{\bfI} \; \text{ and }\; r\in 2\Z,\\
0 &\text{ otherwise}. 
\end{cases}
\end{equation}
As observed in \cite[Lemma 3.7]{fujita2020}, the set $\{\tau^k(\gamma_i):0\leq k\leq \lfloor \frac{h^\vee-2}{2} \rfloor\}$ consists of positive roots. We now use this fact to prove the proposition, beginning with the case where $r$ is odd. By the above formulas for $v_{ij}^{(r)}$, it is enough to show that, for each $0\leq k\leq \lfloor \frac{h^\vee-2}{2}\rfloor$,  $(\tau^k(\gamma_i),\varpi_j)\neq 0$ for some $j\in \dot{\bfI}$. If not, then $\tau^k(\gamma_i)$ is a positive root of the form
\begin{equation*}
\tau^k(\gamma_i)=\sum_{j\in \ddot{\bfI}}n_j^k \alpha_j.
\end{equation*} 
This is impossible unless $\tau^k(\gamma_i)=\alpha_j$ for some $j\in \ddot{\bfI}$. If $k=0$, this contradicts the definition of $\gamma_i$. If $k>0$, then applying $\tau^{-1}=\dot{w}\ddot{w}$ yields 
\begin{equation*}
\tau^{k-1}(\gamma_i)=\tau^{-1}(\alpha_j)=\dot{w}\ddot{w}(\alpha_j)=-\dot{w}(\alpha_j)=-\alpha_j+\sum_{m\in \dot{\bfI}}a_{mj}\alpha_m,
\end{equation*}
where we have used that $\ddot{w}(\alpha_j)=-\alpha_j$ and that $\dot{w}(\alpha_j)=\alpha_j-\sum_{m\in \dot{I}}a_{mj}\alpha_m$ for $j\in \ddot{\bfI}$. 
This contradicts the fact that $\tau^{k-1}(\gamma_i)$ is a positive root. Hence, we may conclude that the proposition holds when $r$ is odd. 

Suppose instead that $r$ is even. 
If the dual Coxeter number $h^\vee$ is odd (\textit{i.e.,} if $\g=\sl_n$ with $n$ even) then $h^\vee-r$ is odd and the assertion of the proposition follows from Part \eqref{vij:3} of Corollary \ref{C:vij}, which gives $v_{ij}^{(r)}=v_{i,j^\ast}^{(h^\vee-r)}$, and the $r$ odd instance of the proposition, proven above. It thus suffices to consider the case where $h^\vee$ is even. To begin, note that for any $j\in \ddot{\bfI}$ with $a_{ij}=-1$, we have
\begin{equation*}
v_{ij}^{(2)}=(\gamma_i,\varpi_j)=1=v_{i,j^\ast}^{(h^\vee-2)}. 
\end{equation*}
Therefore, we are left to show that, for each fixed $1\leq k\leq \frac{h^\vee-4}{2}$, there is $j\in \ddot{\bfI}$ with $(\tau^k(\gamma_i),\varpi_j)\neq 0$. If no such $j$ exists, then analogously to the $r$ odd case, we deduce that $\tau^k(\gamma_i)=\alpha_j$ for some $j\in \dot{\bfI}$. Applying $\tau$, we obtain 
\begin{equation*}
\tau^{k+1}(\gamma_i)=\ddot{w}\dot{w}(\alpha_j)=-\ddot{w}(\alpha_j)=-\alpha_j+\sum_{m\in \ddot{\bfI}}a_{mj}\alpha_m.
\end{equation*}
Since $\tau^{k+1}(\gamma_i)$ is a positive root for $1\leq k\leq \frac{h^\vee-4}{2}$, this is impossible. Hence, there is $j\in \dot{\bfI}$ such that $(\tau^k(\gamma_i),\varpi_j)>0$.
\end{pf}

\appendix

\section{Tensor products of fundamental modules}\label{sec:App-B}
Let $\g$ be an arbitrary simple Lie algebra, and let $\mathsf{d}_{ij}(u)\in \C[u]$ denote the denominator of the normalized $R$-matrix $\mathsf{R}_{L_{\varpi_i},L_{\varpi_j}}(u)\in \End(L_{\varpi_i}\otimes L_{\varpi_j})(u)$, as defined in Sections \ref{ssec:R-matrix} and \ref{ssec:R-fun}.
The goal of this appendix is to prove Proposition \ref{P:cyc-symm}, which in particular asserts that the denominator $\mathsf{d}_{ij}(u+\hbar d_j)$ is symmetric in the indices $i$ and $j$. As indicated in the proof of Proposition \ref{P:fun-denom}, the $U_q(L\g)$-analogue of this result is well-known, having been established in Appendix A of the foundational paper \cite{AK-qAffine} using Conjecture 2 therein, which is now a theorem. 
We begin with three lemmas pertinent to our proof of this proposition.

\subsection{}

For each $i\in \bfI$, define $\nu_i\in \Yhg$ by  
\begin{equation*}
\nu_i=\kappa \xi_{i,0}-\xi_{i,0}^2 +\sum_{\alpha\in R_+}(\alpha_i,\alpha)x_\alpha^- x_\alpha^+,
\end{equation*}
where $R_+$ is the set of positive roots of $\g$ and $x^{\pm}_{\alpha}
\in\g_{\pm\alpha}\subset \Yhg$ are chosen so as to have $(x^+_{\alpha},x^-_{\alpha})=1$, as in Section \ref{ssec:delta}. 
\begin{lem}\label{L:vartheta}
There is an algebra isomorphism $\vartheta_\hbar:Y_{-\hbar}(\g)\isom \Yhg$ uniquely determined by ${\vartheta_\hbar}|_{\g}=\id_{\g}$ and 
\begin{equation*}
\vartheta_\hbar(\xi_{i1})=\xi_{i1}+\hbar \nu_i \quad \forall \quad i\in \bfI.
\end{equation*}
Moreover, for each $a\in \C$, $\vartheta_\hbar$ satisfies the relations
\begin{equation*}
\vartheta_{\hbar}^{-1}=\vartheta_{-\hbar},\quad (\vartheta_\hbar\otimes \vartheta_\hbar)\circ \Delta=\Delta^{\mathrm{op}}\circ \vartheta_\hbar \quad \text{ and }\quad \tau_a\circ \vartheta_{\hbar}=\vartheta_{\hbar}\circ \tau_a.
\end{equation*}
\end{lem}
\begin{pf}
For each fixed $i\in \bfI$, define $J(d_ih_i)=\xi_{i1}+\frac{\hbar}{2}\nu_i\in \Yhg$. Then the set of elements $\{x,J(d_i h_i)\}_{x\in \g, i\in \bfI}$ generates $\Yhg$, and by \cite[Thm.~1]{drinfeld-yangian-qaffine} and \cite[\S3.2]{guay-nakajima-wendlandt} (see also \cite[Thm.~2.6]{GRW-Equivalences}), it is a subset of the generators $\{x,J(x)\}_{x\in \g}$ in Drinfeld's original presentation of $\Yhg$: see \cite{drinfeld-qybe} and \cite[Def.~2.1]{GRW-Equivalences} (here we follow the notation of \cite{GRW-Equivalences}). Since the defining relations in this presentation, as given in \cite[Def.~2.1]{GRW-Equivalences}, depend only on $\hbar^2$, the assignment 
\begin{equation*}
x\mapsto x \aand J(d_i h_i)\mapsto J(d_i h_i) \quad \forall \quad i\in \bfI \; \text{ and }\; x\in \g
\end{equation*}
uniquely extends to a homomorphism of algebras $\vartheta_\hbar:Y_{-\hbar}(\g)\to \Yhg$.
By definition of $\{J(d_i h_i)\}_{i\in \bfI}$, $\vartheta_\hbar$ is uniquely determined by the requirement that $\vartheta|_\g=\id_\g$ and $\vartheta_\hbar(\xi_{i1})=\xi_{i1}+\hbar \nu_i$ for all $i\in \bfI$, as claimed. Moreover, $\vartheta_\hbar$ is clearly invertible with $\vartheta_\hbar^{-1}=\vartheta_{-\hbar}$ and $\tau_a\circ \vartheta_{\hbar}=\vartheta_{\hbar}\circ \tau_a$. We are thus left to prove that $(\vartheta_\hbar\otimes \vartheta_\hbar)\circ \Delta=\Delta^{\mathrm{op}}\circ \vartheta_\hbar$. This is a simple consequence of the fact that the coproduct $\Delta$ of $\Yhg$  is given on $J(d_i h_i)$  by
\begin{equation*}
\Delta(J(d_i h_i))=J(d_i h_i)\otimes 1 + 1\otimes J(d_i h_i) +\frac{\hbar}{2}[d_i h_i\otimes 1, \Omega],
\end{equation*} 
where $\Omega\in (\g\otimes \g)^\g \subset \Yhg\otimes \Yhg$ is the Casimir tensor.
\end{pf}

\subsection{}\label{ssec:App-uniR}

Now let $\mathcal{R}(u)\in \Yhg^{\otimes 2}[\![u^{-1}]\!]$ be the universal $R$-matrix of $\Yhg$, as in Section \ref{ssec:R-matrix}. More precisely, $\mathcal{R}(u)$ is the unique formal series in $1+u^{-1}\Yhg^{\otimes 2}[\![u^{-1}]\!]$
satisfying the intertwiner equation 
\begin{equation}\label{R-inter}
\tau_u\otimes \id \circ \Delta^{\mathrm{op}}(x)= \mathcal{R}(u) \cdot \tau_u\otimes \id \circ \Delta(x) \cdot \mathcal{R}(u)^{-1} \quad \forall\; x\in \Yhg
\end{equation}
in $\Yhg^{\otimes 2}[u;u^{-1}]\!]$, in addition to the cabling identities 
\begin{equation}\label{R-cabling}
\begin{aligned}
\Delta\otimes \id (\mathcal{R}(u))&= \mathcal{R}_{13}(u)\mathcal{R}_{23}(u)\\
\id\otimes \Delta (\mathcal{R}(u))&= \mathcal{R}_{13}(u)\mathcal{R}_{12}(u)
\end{aligned}
\end{equation}
in $\Yhg^{\otimes 3}[\![u^{-1}]\!]$. Here $\tau_u:\Yhg\to \Yhg[u]$ is the algebra homomorphism obtained replacing $a$ by a formal variable $u$ in the definition of the shift automorphism $\tau_a$ from Section \ref{ssec:shift}. The existence and uniqueness of $\mathcal{R}(u)$ was established by Drinfeld in \cite[Thm.~3]{drinfeld-qybe} using cohomological techniques, though the proof was not published. A constructive proof based on the Gauss decomposition of $\mathcal{R}(u)$  may be found in  Theorem 7.4 and Appendix B of \cite{sachin-valerio-curtis}.

 In addition $\mathcal{R}(u)$ satisfies 
\begin{equation}\label{R-matrix-props}
\mathcal{R}(u)^{-1}=\mathcal{R}_{21}(-u) \quad \text{ and }\quad (\tau_a\otimes\tau_b)\mathcal{R}(u)=\mathcal{R}(u+a-b)
\end{equation}
for all $a,b\in \C$. 
In what follows, we write $\mathcal{R}^\hbar(u)$ for the universal $R$-matrix $\mathcal{R}(u)$ of $\Yhg$ to emphasize its dependence on $\hbar$.
\begin{lem}\label{L:R-vtheta}
The universal $R$-matrices of $\Yhg$ and $Y_{-\hbar}(\g)$ are related by 
\begin{equation*}
\mathcal{R}^\hbar(u)=(\vartheta_{\hbar}\otimes \vartheta_\hbar)\left(\mathcal{R}^{-\hbar}_{21}(-u)\right).
\end{equation*}
\end{lem}
\begin{pf}
By the first relation of \eqref{R-matrix-props}, this is equivalent to the identity  
\begin{equation*}
\mathcal{R}^\hbar(u)=(\vartheta_{\hbar}\otimes \vartheta_\hbar)\left(\mathcal{R}^{-\hbar}(u)\right)^{-1}.
\end{equation*}
 It follows readily from the relations $(\vartheta_\hbar\otimes \vartheta_\hbar)\circ \Delta=\Delta^{\mathrm{op}}\circ \vartheta_\hbar$ and  $\tau_a\circ \vartheta_{\hbar}=\vartheta_{\hbar}\circ \tau_a$ of Lemma \ref{L:vartheta} that $(\vartheta_{\hbar}\otimes \vartheta_\hbar)\left(\mathcal{R}^{-\hbar}(u)\right)^{-1}$ satisfies the intertwiner equation \eqref{R-inter} and the cabling identities \eqref{R-cabling} satisfied by $\mathcal{R}^\hbar(u)$. As it also lies in $1+u^{-1}\Yhg^{\otimes 2}[\![u^{-1}]\!]$, it coincides with $\mathcal{R}^\hbar(u)$ by uniqueness. \qedhere
\end{pf}

\subsection{}
Given a representation $V$ of $\Yhg$, we shall denote the associated representation $\vartheta_\hbar^\ast(V)$ of $Y_{-\hbar}(\g)$ by $V^{\vartheta}$. 
\begin{lem}\label{L:L^theta} Let $V\in \Ryang$ and fix $i\in \bfI$ and $s\in \C$. Then:
\begin{enumerate}[font=\upshape]
\item\label{L^theta:1} $V$ is a highest weight module if and only if $V^{\vartheta}$ is.
\item\label{L^theta:2} The $\Yhg$-module $L_{\varpi_i}(s)$ satisfies 
\begin{equation*}
L_{\varpi_i}(s)^{\vartheta}\cong L_{\varpi_i}(s+\hbar\kappa-\hbar d_i)
\end{equation*}
\end{enumerate}
\end{lem}
\begin{pf}
Since $\vartheta_\hbar^{-1}=\vartheta_{-\hbar}$, to prove Part \eqref{L^theta:1} it suffices to show that $V^\vartheta$ is a highest weight module if $V$ is. To this end, let us suppose $V$ is a highest weight module with highest weight vector $v\in V$. Let $\lambda\in \h^\ast$ be the $\g$-weight of $v$. Since the weight space $V_\lambda=(V^\vartheta)_\lambda$ is one-dimensional and preserved by each operator $\vartheta_\hbar(\xi_i(u))$, there is an $\bfI$-tuple $(\mu_i(u))_{i\in \bfI}$, with $\mu_i(u)\in 1+u^{-1}\C[\![u^{-1}]\!]$ for each $i$, such that 
\begin{equation*}
\vartheta_\hbar(\xi_i(u)) v=\mu_i(u) v \quad \forall \quad i\in \bfI. 
\end{equation*}
Since $\vartheta_\hbar$ is an isomorphism, the $Y_{-\hbar}(\g)$-module $V^\vartheta$ is generated by $v$. We are thus left to show that $\vartheta_\hbar(x_i^+(u))v=0$ for all $i\in \bfI$. 
This follows from \eqref{fX=X0'}, for instance, and that $\vartheta_\hbar(\xi_i(u)) v=\mu_i(u) v $ and $\vartheta_\hbar(x_{i,0}^+)v=x_{i,0}^+v=0$ for all $i\in \bfI$. 

Consider now Part \eqref{L^theta:2}. Since $L_{\varpi_i}(s)^{\vartheta}$ is a finite-dimensional irreducible $Y_{-\hbar}(\g)$-module, there is an $\bfI$-tuple of monic polynomials $\ul{P}=(P_i(u))_{i\in \bfI}$ such that $L(\ul{P})\cong L_{\varpi_i}(s)^{\vartheta}$. Letting $v\in L_{\varpi_i}(s)^{\vartheta}$ be a highest weight vector, we have $v\in L_{\varpi_i}(s)^{\vartheta}_{\lambda}$ with $\lambda=\sum_{j\in \bfI} \deg P_j(u)\varpi_j$. On the other hand, the proof of Part \eqref{L^theta:1} shows that $v$ is a highest weight vector for $L_{\varpi_i}(s)$, and therefore has weight $\varpi_i$. It follows that $\deg P_j(u)=\delta_{ij}$, and thus that $L_{\varpi_i}(s)^{\vartheta}\cong L_{\varpi_i}(b)$ for some $b\in \C$. Moreover, since 
\begin{equation*}
\vartheta_h(\xi_i(u))v=\frac{u-b-\hbar d_i}{u-b}=1-\hbar d_i\sum_{r\geq 0} b^r u^{-r-1},
\end{equation*}
the value of $b$ is determined by $\vartheta_h(\xi_{i1})v=d_i bv$. Since $\vartheta_h(\xi_{i1})=\xi_{i1}+\hbar \nu_i$,  $\xi_{i1} v=d_i sv$ and $\nu_i v=d_i \kappa - d_i^2$, we can conclude  that $b=s+\hbar \kappa -\hbar d_i$. \qedhere
\end{pf}

\subsection{}

With the above machinery at our disposal, we are now prepared to prove the main result of this appendix. 
\begin{prop}\label{P:cyc-symm}
For each $i,j\in \bfI$ and diagram automorphism $\omega$ of $\g$, one has 
\begin{equation*}
\mathsf{d}_{ij}(u)=\mathsf{d}_{\omega(i)\omega(j)}(u)\quad \text{ and }\quad \mathsf{d}_{ji}(u+\hbar d_i)=\mathsf{d}_{ij}(u+\hbar d_j).
\end{equation*}
\end{prop}
\begin{pf}
Recall from Section \ref{ssec:DDauto} that any diagram automorphism $\omega$ defines an automorphism of $\Yhg$, which is readily seen to be a homomorphism of coalgebras (see \S\ref{ssec:delta}). As $\omega$ comutes with the shift automorphism $\tau_a$ for each $a\in \C$, it follows that $(\omega\otimes \omega)\mathcal{R}(u)$ satisfies the defining relations \eqref{R-inter} and \eqref{R-cabling} of $\mathcal{R}(u)$, and so coincides with it by uniqueness. The first equality of the proposition now follows readily from the definition of $\mathsf{d}_{ij}(u)$ (see Sections \ref{ssec:R-matrix} and \ref{ssec:R-fun}) and the identification 
\begin{equation*}
(L_{\varpi_i}\otimes L_{\varpi_j}(s))^\omega\cong L_{\varpi_i}^\omega\otimes L_{\varpi_j}(s)^\omega\cong L_{\varpi_{\omega(i)}}\otimes L_{\varpi_{\omega(j)}}(s).
\end{equation*}

Let us now turn to the second identity of the proposition. 
To make the dependencies on $\hbar$ clear, let us write $\mathsf{R}_{L_{\varpi_j},L_{\varpi_i}}^\hbar(u)$ for  the normalized $R$-matrix $\mathsf{R}_{L_{\varpi_j},L_{\varpi_i}}(u)$ of  $\Yhg$ defined in Section \ref{ssec:R-matrix} and $\mathsf{d}_{ji}^\hbar(u)$ for its denominator. 
It follows from the second identity of \eqref{R-matrix-props} and Lemmas \ref{L:vartheta} and \ref{L:L^theta} that one has the equality of normalized $R$-matrices 
\begin{equation*}
\mathsf{R}_{L_{\varpi_j},L_{\varpi_i}}^{-\hbar}(u+\hbar d_i-\hbar d_j)= (1\,2)\circ \mathsf{R}_{L_{\varpi_i},L_{\varpi_j}}^{\hbar}(-u) \circ (1\,2),
\end{equation*}
and therefore  $\mathsf{d}_{ji}^{-\hbar}(u+\hbar d_i-\hbar d_j)=(-1)^{\deg \mathsf{d}_{ij}^\hbar}\mathsf{d}_{ij}^\hbar(-u)$. The desired equality now follows from Theorem \ref{thm:Comb} and Part \eqref{cyc-R:1} of Proposition \ref{P:R-pole}, which give  
\begin{equation*}
\mathsf{Z}(\mathsf{d}_{ji}^\hbar(u))\subset \sigma_i(L_{\varpi_j})+\hbar d_i\subset\frac{\hbar}{2}\Z
\end{equation*}
 and hence $\mathsf{d}_{ji}^{-\hbar}(u)=(-1)^{\deg \mathsf{d}_{ij}^\hbar}\mathsf{d}_{ji}^{\hbar}(-u)$. \qedhere
\end{pf}


\bibliographystyle{amsplain}
\providecommand{\bysame}{\leavevmode\hbox to3em{\hrulefill}\thinspace}
\providecommand{\MR}{\relax\ifhmode\unskip\space\fi MR }
\providecommand{\MRhref}[2]{%
  \href{http://www.ams.org/mathscinet-getitem?mr=#1}{#2}
}
\providecommand{\href}[2]{#2}

\end{document}